\documentclass[12pt]{amsart}
\usepackage{amsmath,amssymb,latexsym, amsthm, amscd, mathrsfs, stmaryrd, bbold}
\usepackage[linktocpage=true]{hyperref}
\usepackage{amsfonts}
\usepackage[all]{xy}
\usepackage{soul}
\usepackage[makeroom]{cancel}

\newcommand{\smxylabel}[1]{{\text{\small$#1$}}}

\theoremstyle{definition}
\newtheorem{rem}[subsubsection]{Remark}
\theoremstyle{plain}
\newtheorem{prop}[subsubsection]{Proposition}
\newtheorem{thm}[subsubsection]{Theorem}
\newtheorem{lem}[subsubsection]{Lemma}
\newtheorem{cor}[subsubsection]{Corollary}

\newtheorem{thrm}{Positivity}

\newcommand{\mbf}{\mathbf}
\newcommand{\mbb}{\mathbb}
\newcommand{\mcal}{\mathcal}

\newcommand{\mrm}{\mathrm}

\newcommand{\A}{\mbf A}
\newcommand{\B}{\jmath}
\newcommand{\D}{\mathcal D}
\newcommand{\E}{\mathbf e}
\newcommand{\F}{\mbf  f}
\newcommand{\G}{\mbf G}
\newcommand{\K}{\mbf k}
\newcommand{\bfL}{\mbf L}

\newcommand{\X}{\mathcal X}
\newcommand{\TA}{\phi_{d, d-n,\bv}}

\newcommand{\e}{\varepsilon}

\newcommand{\cphi}{\phi^{\imath}_{d, d- \ell }}
\newcommand{\cSii}{\mbb{S}^{\imath}_{d- \ell , \ell}}
\newcommand{\cSi}{\mbb{S}^{\imath}_{d, \ell }}

\newcommand{\Sj}{\mbb S^{\jmath}_{d,n}}
\newcommand{\bE}{\mbf E}
\newcommand{\bF}{\mbf F}
\newcommand{\bK}{\mbf K}
\newcommand{\bH}{\mbf H}
\newcommand{\cbS}{\mbb S_{d, \ell}}

\newcommand{\bh}{\mbf h}
\newcommand{\bk}{\mbf k}
\newcommand{\bj}{\mbf j}
\newcommand{\bJ}{\mbf J}
\newcommand{\bt}{\mbf t}
\newcommand{\Di}{\Delta^{\imath}}

\newcommand{\Dj}{\Delta^{\jmath}}
\newcommand{\tDj}{\widetilde \Delta^{\jmath}}

\newcommand{\mk}{\mbb k}
\newcommand{\mDi}{\mbb \Delta^{\imath}}
\newcommand{\mDj}{\mbb \Delta^{\jmath}}
\newcommand{\mE}{\mbb E}
\newcommand{\mF}{\mbb F}
\newcommand{\mK}{\mbb K}
\newcommand{\mU}{\mbb U}
\newcommand{\mX}{\mbb X}
\newcommand{\mY}{\mbb Y}
\newcommand{\hmY}{\widehat{\mY}}
\newcommand{\hmX}{\widehat{\mX}}

\newcommand{\hI}{\widehat{I}}
\newcommand{\mUi}{\mbb U^{\imath}}
\newcommand{\mUj}{\mbb U^{\jmath}}
\newcommand{\mB}{\mbb B}
\newcommand{\mBi}{\mbb B^{\imath}}
\newcommand{\mBj}{\mbb B^{\jmath}}

\newcommand{\aSB}{\widehat{\mbb  S}}

\newcommand{\ro}{\mrm{ro}}
\newcommand{\co}{\mrm{co}}

\newcommand{\U}{\mbb U}

\newcommand{\V}{\mbf  V}

\newcommand{\bv}{\mbf v}

\usepackage{color}
\newcommand{\nc}{\newcommand}
\nc{\redtext}[1]{\textcolor{red}{#1}}
\nc{\bluetext}[1]{\textcolor{blue}{#1}}
\nc{\greentext}[1]{\textcolor{green}{#1}}
\nc{\yl}[1]{\redtext{ #1}}
\nc{\zb}[1]{\redtext{From zb: #1}}

\title[Positivity under   comultiplication]{Positivity  of  canonical bases under  comultiplication}

\author{Zhaobing Fan and Yiqiang Li}
\address{School of science\\ Harbin Engineering University\\ Harbin, China 150001}
\email{fanz@math.ksu.edu (Z. Fan)}
\address{Department of Mathematics\\ University at Buffalo, SUNY  \\Buffalo, NY 14260}
\email{yiqiang@buffalo.edu (Y. Li)}

\date{\today}
\keywords{}
\subjclass{}

\begin{document}

\begin{abstract}
We show the positivity  of  the canonical basis for  a modified quantum affine $\mathfrak{sl}_n$  under the comultiplication.
Moreover, we  establish the positivity of  the i-canonical basis in ~\cite{LW14}
with respect to the coideal subalgebra structure.
\end{abstract}

\maketitle

\section*{Introduction}

The geometric study of the modified quantum $\mathfrak{sl}_n$ via perverse sheaves on partial flag varieties of type $A$ is initiated in the work ~\cite{BLM90} of  Beilinson, Lusztig and MacPherson.
It is then generalized to quantum affine $\mathfrak{sl}_n$ by  Ginzburg-Vasserot ~\cite{GV93} and
Lusztig ~\cite{L99}, ~\cite{L00}, independently, by considering the geometry of affine partial flag varieties of type $A$.
This line of research is culminated in the work of Schiffmann-Vasserot  ~\cite{SV00} and McGerty ~\cite{M10} showing that the canonical basis of modified quantum affine $\mathfrak{sl}_n$ defined geometrically via transfer maps can be identified with the one defined algebraically by Lusztig ~\cite{L93} (see also Kashiwara ~\cite{Ka94}).
As a consequence, one obtains the positivity of the structure constants of the canonical basis of quantum affine $\mathfrak{sl}_n$ with respect to multiplication, which is conjectured in ~\cite[25.4.2]{L93}.

In a recent remarkable work ~\cite{BW13} of Bao-Wang,   a quantum-Schur-like duality relating a type-$B/C$ Hecke algebra and a coideal subalgebra of quantum $\mathfrak{sl}_n$ is obtained, and moreover an $\imath$-canonical basis for the representations of coideal subalgebras involved is constructed.
The desire to geometrize  Bao-Wang's work and
to describe the convolution algebras of certain  perverse sheaves on  partial flag varieties of classical type
lead to the work ~\cite{BKLW13}, where
the approach in ~\cite{BLM} is revived and adapted to give a geometric construction
of the (modified) coideal subalgebra of quantum $\mathfrak{gl}_n$ and a $stably$ canonical basis  by using certain perverse sheaves  of partial flag varieties of type $B/C$.
Since a modified coideal subalgebra of quantum $\mathfrak{gl}_n$ can be regarded as
a direct sum of infinitely many copies of its $\mathfrak{sl}_n$ version,
one  obtains infinitely many stably canonical bases of the modified coideal subalgebra of quantum $\mathfrak{sl}_n$.
As a consequence, the $\imath$-canonical basis of the tensor space in the duality in ~\cite{BW13} admits a geometric incarnation as certain intersection cohomology complexes.

Despite of many favorable properties of the stably canonical bases of modified quantum $\mathfrak{gl}_n$ and its coideal subalgebras,
they do not admit positivity with respect to multiplication; counterexamples can be found in   ~\cite{LW14}.
Instead,   a new basis,  called the i-canonical basis,  of the modified coideal subalgebra  of quantum $\mathfrak{sl}_n$ is constructed in   {\it loc. cit.}
following the spirit of ~\cite{L00}  and ~\cite{M10} (see also ~\cite{SV00}). This basis can be regarded as an asymptotical version of the stably canonical basis since they coincide asymptotically (\cite{LW14}).
It is further shown that the i-canonical basis does admit three positivities with respect to the multiplication, a bilinear form of geometric origin in {\it loc. cit.}
and its action on the $\imath$-canonical basis of  a tensor space.

In this article, we  establish three more  positivities of i-canonical bases,
in addition to the previous ones in ~\cite{LW14},
mainly with respect to  the coideal subalgebra structure.
To be more precise,  let $\mbb B$ be the canonical basis of modified quantum $\mathfrak{sl}_n$, say $\dot \U$, and $\mbb B^{\text i}$ its coideal analogue in the modified coideal subalgebra  $\dot \U^{\text i}$.
We note that notations in the introduction are slightly different from the main body of the paper.
One can define in a natural way an algebra homomorphism
$\mbb \Delta^{\text i}: \dot \U^{\text i} \to (\dot \U^{\text i} \otimes \dot \U)^{\wedge}$, where the target is a certain variant of the tensor algebra
$\dot \U^{\text i} \otimes \dot \U$, which is an idempotented version of the coideal structure
coming from the comultiplication of quantum $\mathfrak{sl}_n$. (See  (\ref{mDj-comm}), (\ref{mDi-comm}) for precise definitions.)
In particular, if $a\in \mbb B^{\text i}$, one has
$$
\mbb \Delta^{\text i} (a) =\sum_{b\in \mbb B^{\text i}, c\in \mB} n^{b, c}_a b\otimes c, \quad  n^{b, c}_a \in \mbb Z[v, v^{-1}].
$$
The positivity with respect to the idempotented coideal structure further says that

\begin{thrm} [Theorem \ref{J-conj-25.4.2}, \ref{I-conj-25.4.2}]
\label{thrm-A}
The structure constant $n^{b, c}_a $ is in $\mbb Z_{\geq 0} [v, v^{-1}]$.
\end{thrm}

A degenerate version of $\mbb \Delta^{\text i}$ induces an imbedding
$\text i: \dot \U^{\text i} \to (\dot \U)^{\wedge}$, which reflects the subalgebra structure of the ordinary coideal subalgebra in quantum $\mathfrak{sl}_n$. (See (\ref{U-jmath}), (\ref{U-imath}).)
The positivity with respect to the idempotented subalgebra structure says that

\begin{thrm} [Joint with Weiqiang Wang; Theorem \ref{j-positivity}, ~\ref{i-positivity}]
\label{thrm-B}
If
$\mathsf i (a)
=\sum_{b \in \mB}  g_{b, a} b$, $\forall a \in \mbb B^{\mathsf i}$, then
$g_{b, a} \in \mbb Z_{\geq 0}[v, v^{-1}]$.
\end{thrm}

As  a second degeneration of $\mbb \Delta^{\text i}$,
we  make a direct connection between the geometric type $A$ duality of ~\cite{GL92} and type $B/C$ duality of ~\cite{BKLW13}, which reveals  yet another positivity:

\begin{thrm} [Theorem ~\ref{tensor-positivity}]
\label{thrm-C}
The   $\imath$-canonical basis in a tensor space is a positive sum of the canonical basis in the same tensor space.
\end{thrm}

As is shown, these positivities are boiled down to a geometric interpretation of the coideal structure coming from the comultiplication of quantum $\mathfrak{sl}_n$.
To this end, we also establish a geometric realization of the comultiplication of quantum $a \! f \!  \! f \! ine$ $\mathfrak{sl}_n$, and we obtain the following positivity on quantum affine $\mathfrak{sl}_n$:

\begin{thrm} [Theorem ~\ref{affine-positivity}]
\label{thrm-D}
The canonical basis of modified quantum affine $\mathfrak{sl}_n$ admits positivity with respect to the idempotented  comultiplication.
\end{thrm}


The proof of the positivity result on quantum affine $\mathfrak{sl}_n$ consists of two parts since the geometrically defined comultiplication on the affine Schur algebra level is a composition of
a hyperbolic localization ~\cite{B03} and  a twist of a certain $v$-power.
The positivity  on the former  is well-known by ~\cite{B03}, (see also ~\cite{L00}, ~\cite{SV00}), while we show that in the latter  it sends a canonical basis to a canonical basis up to a $v$-power. Note that the second step is trivial in the ordinary quantum $\mathfrak{sl}_n$ case, but non-trivial in the affine case as far as we can see: because
at some point, we have to invoke the multiplication formula of a semisimple generator of Du-Fu ~\cite{DF13}, for which we provide a new geometric proof.
These arguments are contained in the first two sections, with the first section devoted to quantum $\mathfrak{sl}_n$ and the second one to its affine version.

The argument of the proof on quantum affine $\mathfrak{sl}_n$ also applies with modifications  to
the various positivities of the i-canonical basis, which occupies the last three sections. The third section treats the results on the  i-Schur-algebra level, and the  fourth section lifting the results on the i-Schur-algebra level to the projective limit level for
$n$ being odd. The last section collects similar results for $n$ even.
The $\jmath$transfer maps used in ~\cite{LW14} are constructed geometrically in these  sections
and the reader can find the proof of ~\cite[Lemma 4.3]{LW14} in Proposition ~\ref{Phi-gene}.

Note that we work over the partial flag varieties of type $B$ for the i-canonical basis
and following the treatment of type $A$ in ~\cite{L00}, ~\cite{M10}.
One can obtain the same  results via partial flag varieties of type $C$ by using the principle in ~\cite{BKLW13}.

In ~\cite{FLLLW}, we shall construct and investigate geometrically
the i-canonical basis of modified coideal subalgebras of quantum affine $\mathfrak{sl}_n$ among others.

We refer to
 ~\cite{ES13} and ~\cite{FL14} for the interactions of type $D$ partial flag varieties,  coideal subalgebras and type $D$  duality.
In a forthcoming paper, we will present a type $D$ picture similar to the positivity results on i-canonical basis in this paper.

\subsection*{Acknowledgement}

Part of the work is done while Y.L. was enjoying the stay at University of Virginia during October 20-24, 2014.
He is grateful for the invitation of Weiqiang Wang and the hospitality of  the university.
We also thank Weiqiang Wang for allowing us to include the joint work Positivity ~\ref{thrm-B} in the paper.
 While preparing the paper, Professor George Lusztig informed us that Grojnowski proved Theorem \ref{conj-25.4.2} around twenty years ago in an unpublished short paper, we thank him for notifying us.

\setcounter{tocdepth}{1}
\tableofcontents

\section{Positivity for quantum $\mathfrak{sl}_n$}

In this section, we shall present a proof of the positivity of the canonical basis of quantum $\mathfrak{sl}_n$ with respect to comultiplication.

\subsection{Convolution}
\label{conv}

Let $G$ be a group, and $\X$ a  $G$-set. The $G$-action on $\X$ thus induces  a diagonal $G$-action on the product $\X \times \X$.
Let $\mcal A$ be a unital commutative ring.
We consider the set
$\mcal A_G( \X\times \X)$ of all $\mcal A$-valued $G$-invariant functions on $\X\times  \X$ supported on finitely many $G$-orbits.
Assume that  any $G$-orbit $\mathcal O$ in $\X \times \X$ has the property that  the set $\X^x_{\mathcal O} =\{ y\in \X | (x, y) \in \mathcal O\}$ is finite for one and hence any fixed $x$ in $\X$.
Then $\mcal A_G(\X\times \X)$ is  a  free $\mcal A$-module with a basis indexed by the $G$-orbits in $\X\times \X$, and  further an associative  $\mcal A$-algebra with multiplication as follows.
For any $f_1, f_2\in \mcal A_G(\X\times \X)$, the function $f_1 * f_2$  is defined by
\begin{equation}
\label{eq30}
f_1 *  f_2(x_1, x_3)=\sum_{x_2 \in \X} f_1 (x_1, x_2) f_2(x_2,x_3), \quad \forall\ x_1,x_3\in  \X.
\end{equation}
Let $\bf 1$ be the characteristic function of the diagonal $\{(x, x) | x \in \X \}$.
By definition, $\bf 1$ is the unit of the algebra $( \mcal A_G( \X\times \X), *)$.
For convenience, we will simply use the notation   $\mcal A_G( \X\times \X)$ to represent the algebra $( \mcal A_G( \X\times \X), *)$.

\subsection{$\bv$-Schur algebra}
\label{vSA}

Let $\mbb F_q$ be a finite field of $q$ elements and of odd characteristic. Let
\begin{equation}\label{A}
\bv=\sqrt{q}, \quad
\mcal A=\mbb Z[\bv, \bv^{-1}].
\end{equation}
We fix a pair $(n, d)$ of irrelevant positive integers.
Consider the set $X_d$ of $n$-step partial flags in a fixed $d$-dimensional vector space $\mbb F_q^d$ over $\mbb F_q$ of the form
$$V=(0 \equiv V_0\subseteq V_1 \subseteq \cdots \subseteq   V_{n-1} \subseteq V_n \equiv  \mbb F_q^d).$$
Denote by $\G_d = \mrm{GL}(\mbb F_q^d)$ the general linear group over $\mbb F_q$ of rank $d$.
Let $\G_d$ act from the left on the set $X_d$. By the general setting in Section ~\ref{conv}, we have a unital associative algebra
\begin{align}
\label{SA}
\mbf S_d \equiv \mcal A_{\G_d} ( X_d \times X_d).
\end{align}
It is well-known that the algebra  $\mbf S_d$ is the  $\bv$-$Schur$ $algebra$ of type $\A_{n-1}$ (\cite{BLM}).

The definitions of these objects depend on the integer $n$, but it is suppressed since it never changes,
except at Section ~\ref{i-version} where we use notations $X_{d, n}$, $\mbf S_{d, n}$,  etc.

\subsection{Coproduct on $\mbf S_d$}
\label{CoSA}
Now consider a triple $(d, d', d'')$ of positive integers such that $d' + d''= d$. We fix an isomorphism of vector spaces
$
\mbb F^{d'}_q \oplus \mbb F^{d''}_q \simeq \mbb F^{d}_q.
$
Let $\pi'$ be the projection of $\mbb F^d_q$ to $\mbb F^{d'}_q$.
Let $\pi''$ be the operation of intersection with $\mbb F^{d''}_q$, i.e., $\pi''(W) = W\cap \mbb F^{d''}_q$ for any subspace $W$ in $\mbb F^d_q$.
Given a flag $V$ in $X_d$, the notations $\pi'(V)$ and $\pi''(V)$ are thus meaningful.
For any $(V', V'') \in X_{d'}\times X_{d''}$, we set
\[
Z_{V', V''} =\{ V\in X_d |
\pi'(V)= V', \pi''(V) =V'' \}.
\]
Note that we can identify $\mbf S_{d'} \otimes \mbf S_{d''}$ with the algebra
$
\mcal A_{\G_{d'}\times\G_{d''}} (X_{d'} \times X_{d'} \times X_{d''}\times X_{d''}).
$
We define a linear map
\begin{align}
\label{delta-A}
\widetilde \Delta: \mbf S_d \to \mbf S_{d'} \otimes \mbf S_{d''}
\end{align}
by
$
\widetilde \Delta (f) (V', \tilde V', V'', \tilde V'')
=\sum_{\tilde V \in Z_{\tilde V', \tilde V''} } f(V, \tilde V),
$
for any quadruple $(V', \tilde V', V'', \tilde V'') \in X_{d'} \times X_{d'} \times X_{d''}\times X_{d''}$  where $V$ is a fixed element in $Z_{V', V''}$.
(It can be shown that the definition is independent of the choice of $V$.)
The following result can be found in ~\cite[2.2]{L00}, which is credited back to Grojnowski.

\begin{prop}
The map $\widetilde \Delta$ in (\ref{delta-A}) is a well-defined algebra homomorphism over $\mathcal A$.
\end{prop}

Let $|W|$ denote the dimension of the vector space $W$ over $\mbb F_q$.
We use the notation $W_1 \overset{a}{\subset} W_2$ to denote $W_1 \subset W_2$ and $\dim W_2/W_1 =a$. Similarly, we define the notation
$W_1 \overset{a}{\supset} W_2$.
We define the following functions in $\mbf S_d$.
For any $i\in [1, n-1]$, $a\in [1, n]$,
\begin{equation}\label{generatorA}
\begin{split}
\mbf E_i (V, V') &=
\begin{cases}
\bv^{-|V'_{i+1}/V'_i|}, &\mbox{if}\; V_i\overset{1}{\subset} V_i', V_j=V_{j'},\forall j \neq i; \\
0, &\mbox{otherwise}.
\end{cases} \\
\mbf F_i (V, V') &=
\begin{cases}
\bv^{-|V'_i/V'_{i-1}|}, &\mbox{if}\; V_i\overset{1}{\supset} V_i', V_j=V_{j'},\forall j\neq i; \\
0, &\mbox{otherwise},
\end{cases} \\
\mbf H_{ a}^{\pm 1} (V, V') & = \bv^{\pm |V_a/V_{a-1}|}  \delta_{V, V'},  \quad \forall V, V' \in X_d.\\
\mbf K_{i}^{\pm 1}  &  = \mbf H_{i+1}^{\pm 1}  \mbf H_{i}^{\mp 1}.\\
\end{split}
\end{equation}
Notices that if the subscript $d$ is replaced by $d'$ or $d''$, the functions defined above are in $\mbf S_{d'}$ or $\mbf S_{d''}$,
respectively, and will be denoted by
$\mbf H_a'$, $\mbf K_i'$,  $\mbf E_i'$, $\mbf F'_i$ or $\mbf H_a''$, $\mbf K_i''$, $\mbf E_i''$ and $\mbf F_i''$.
(This convention will be used in any similar situation appearing later.)
The following lemma is due to Lusztig  ~\cite[Lemma 1.6]{L00}.

\begin{lem}\label{D(A)}
For any $i\in [1, n-1]$, we have
\begin{align*}
\begin{split}
\widetilde \Delta ( \mbf E_i)  = \mbf E'_i \otimes  \mbf H''_{i+1} + \mbf H'^{-1}_{i+1} \otimes \mbf E''_i, \
\widetilde \Delta (\mbf F_i)  = \mbf F'_i \otimes \mbf H_{i}''^{-1}+ \mbf H_i' \otimes \mbf F_i'', \
\widetilde \Delta (\mbf K_i)  = \mbf K_i' \otimes \mbf K_i''.
\end{split}
\end{align*}
\end{lem}

Note that the functions $\mbf E_i$ and $\mbf F_i$ correspond to the functions in ~\cite[2.4]{L99} in the notations  $F_i$ and $E_i$, respectively.
Let
\begin{align} \label{Lambda-A}
\Lambda_{d, n} = \{ \mbf a=(a_1,\cdots, a_n)\in \mbb Z_{\geq 0}^n | a_1 + \cdots + a_n=d\}.
\end{align}
We can decompose $X_d$ as follows.
\[
X_d = \sqcup_{\mbf a \in \Lambda_{d, n} } X_d(\mbf a), \quad
X_d(\mbf a) =\{ V\in X_d | | V_i/V_{i-1} | = a_i, \forall 1\leq i\leq n\}.
\]
Let $\mbf S_d(\mbf b, \mbf a)$ be the subspace of $\mbf S_d$ spanned by all functions  on
$X_d(\mbf b) \times X_d(\mbf a)$.
Let
\[
\widetilde \Delta_{ \mbf b', \mbf a', \mbf b'', \mbf a''}:
\mbf S_d (\mbf b, \mbf a)  \to \mbf S_{d'}(\mbf b', \mbf a') \otimes \mbf S_{d''} (\mbf b'', \mbf a'')
\]
be the linear map obtained from $\widetilde \Delta$ by restricting $\widetilde \Delta$ to the subspace $\mbf S_d(\mbf b, \mbf a)$
and projecting down to the component  $\mbf S_{d'}(\mbf b', \mbf a') \otimes \mbf S_{d''} (\mbf b'', \mbf a'')$.
Then we have
\[
\widetilde \Delta = \oplus \widetilde \Delta_{ \mbf b', \mbf a', \mbf b'', \mbf a''},
\]
where the sum runs over $\mbf b, \mbf a, \mbf b', \mbf a', \mbf b'', \mbf a''$ such that
$\mbf a, \mbf b \in \Lambda_{d, n}$, $\mbf a', \mbf b'\in \Lambda_{d', n}$,  $\mbf a'', \mbf b'' \in \Lambda_{d'', n}$ and
$\mbf b =\mbf  b'+ \mbf b''$ and $\mbf a = \mbf a' + \mbf  a''$.
We set
\begin{align}
\label{Delta-A}
\Delta_{\mbf b', \mbf a', \mbf b'', \mbf a''}
=  \bv^{ \sum_{1\leq i\leq j\leq n} b_i' b_j'' - a_i' a_j''}  \widetilde \Delta_{\mbf b', \mbf a', \mbf b'', \mbf a''}, \quad
\Delta_{\bv} = \oplus  \Delta_{ \mbf b', \mbf a', \mbf b'', \mbf a''}.
\end{align}

The following is a refinement of Lemma \ref{D(A)}.

\begin{prop} \label{Delta_A}
The linear map $\Delta_{\bv}$ in (\ref{Delta-A})  is an algebra homomorphism. Moreover,
\begin{align}
\label{D(A)-1}
\begin{split}
\Delta_{\bv} ( \mbf E_i) & = \mbf E'_i \otimes \mbf K_i'' + 1 \otimes \mbf E''_i, \\
\Delta_{\bv} (\mbf F_i) & = \mbf F'_i \otimes 1+ \mbf K'^{-1}_{i} \otimes \mbf F_i'', \\
\Delta_{\bv}  (\mbf K_i) & = \mbf K_i' \otimes \mbf K_i'', \hspace{2cm} \forall 1\leq i \leq n-1.
\end{split}
\end{align}
\end{prop}

\begin{proof}
It is straightforward to see that $\Delta_{\bv}$ is an algebra homomorphism. We proceed to the proof of the equalities in the proposition.
Suppose that a quadruple $( \mbf b', \mbf a', \mbf b'', \mbf a'')$ satisfies the conditions that
$b'_k = a'_k - \delta_{k, i} + \delta_{k, i+1}$ and $b''_k = a''_k$ for some $i$ and  for all $1 \leq k \leq n$. Then
\[
\sum_{1\leq i\leq j\leq n} b_i' b_j'' - a_i' a_j''
=\sum_{1\leq k \leq j \leq n} (b_k' - a_k') a_j''
= -\sum_{i \leq j \leq n} a_j'' + \sum_{i+1 \leq j \leq n} a_j'' = - a_i''.
\]
So if $(V', \tilde V', V'', \tilde V'') \in X_{d'} \times X_{d'} \times X_{d''}\times X_{d''}$, then
\[
\Delta_{\mbf b', \mbf a', \mbf b'', \mbf a''} (V', \tilde V', V'', \tilde V'')
= \bv^{- a_i''}   \mbf E'_i \otimes  \mbf H''_{i+1}  ( V', \tilde V', V'', \tilde V'')
=\mbf E'_i \otimes \mbf K_i'' (V', \tilde V', V'', \tilde V'').
\]

On the other hand, if $( \mbf b', \mbf a', \mbf b'', \mbf a'')$ is a quadruple subject to $b'_k = a'_k$ and $b''_k = a''_k - \delta_{i, k} + \delta_{i+1, k}$
for some $i$ and for all $1\leq k\leq n$,
then
$\sum_{1\leq i\leq j\leq n} b_i' b_j'' - a_i' a_j'' = a'_{i+1}$.
Thus  if  $(V', \tilde V', V'', \tilde V'') \in X_{d'} \times X_{d'} \times X_{d''}\times X_{d''}$, then
\[
\Delta_{\mbf b', \mbf a', \mbf b'', \mbf a''} (V', \tilde V', V'', \tilde V'')
=
\bv^{a'_{i+1}} \mbf H'^{-1}_{i+1} \otimes \mbf E''_i (V', \tilde V', V'', \tilde V'')
=1\otimes \mbf E''_i (V', \tilde V', V'', \tilde V'').
\]
Altogether, we have
$
\Delta_{\bv} ( \mbf E_i)  = \mbf E'_i \otimes \mbf K_i'' + 1 \otimes \mbf E''_i,
$
which is the first equality in the lemma.

If the quadruple $( \mbf b', \mbf a', \mbf b'', \mbf a'')$ satisfies  that
$b'_k = a'_k + \delta_{k, i} - \delta_{k, i+1}$ and $b''_k = a''_k$ for some $i$ and  for all $1 \leq k \leq n$, then the twist
$\sum_{1\leq i\leq j\leq n} b_i' b_j'' - a_i' a_j''$ is equal to $a_i''$. So after the twist,
it makes the first term $ \mbf F'_i \otimes \mbf H_{i}''^{-1}$ of $\widetilde \Delta (\mbf F_i)$ in Lemma ~\ref{D(A)} into $\mbf F'_i\otimes 1$.
Meanwhile, if $( \mbf b', \mbf a', \mbf b'', \mbf a'')$ is a quadruple subject to $b'_k = a'_k$ and $b''_k = a''_k + \delta_{i, k} - \delta_{i+1, k}$ for some $i$ and for all $1\leq k\leq n$,
then
$\sum_{1\leq i\leq j\leq n} b_i' b_j'' - a_i' a_j'' = - a'_{i+1}$.
Hence after the twist, the second term $\mbf H'_i \otimes \mbf F''_i$ in $\widetilde \Delta (\mbf F_i)$ becomes
$\mbf K'^{-1}_i \otimes \mbf F''_i$. This verifies the second equality in the lemma.

Since the twist is zero if  $\mbf b' =\mbf a'$ and $\mbf b'' = \mbf a''$, the third equality holds.
\end{proof}

\begin{rem}
Note that if we write  ($\mbf E_i, \mbf F_i , \mbf K_i)$ as $(\mbf F_i, \mbf E_i, \mbf K^{-1}_{i})$, we have the conventional  comultiplication.
\end{rem}

For the rest of this section, we give a second interpretation of $\widetilde \Delta$ to be used
in the proof of Proposition ~\ref{Pos-S}.
Fix $V \in X_d(\mbf b) $ and set $P_{\mbf b} =\mbox{Stab}_{\G_d} (V)$. Then $P_{\mbf b}$ acts via $\G_d$ on $X_d(\mbf b)$.
Consider the imbedding
\[
i_{\mbf b, \mbf a}: X_d(\mbf a) \to X_d(\mbf b) \times X_d(\mbf a), \quad \tilde V \mapsto ( V, \tilde V).
\]
It induces a bijection between  $P_{\mbf b}$-orbits in the domain and $\G_d$-orbits in the range of $i_{\mbf b, \mbf a}$.
Hence the pullback (restriction)
\begin{align}
\label{i*}
i^*_{\mbf b, \mbf a}: \mathcal A_{\G_d}(X_d(\mbf b) \times X_d(\mbf a)) \to \mathcal A_{P_{\mbf b}} (X_d(\mbf a))
\end{align}
of the imbedding $i_{\mbf b, \mbf a}$ is an isomorphism of $\mathcal A$-modules.

Recall now that we fix a triple $(V, V', V'')$ in the definition of $\widetilde \Delta$ in (\ref{delta-A}).
We assume that $V \in X_d(\mbf b)$, $V'\in X_{d'} (\mbf b')$ and $V'' \in X_{d''} (\mbf b'')$ so that $\mbf b' + \mbf b'' =\mbf b$.
We also define $P_{\mbf b'}$ and $P_{\mbf b''}$ similar to $P_{\mbf b}$. Thus we have similar isomorphisms
\begin{equation*}
\begin{split}
& i^*_{\mbf b', \mbf a'}: \mathcal A_{\G_{d'}}(X_{d'} (\mbf b') \times X_{d'} (\mbf a')) \to \mathcal A_{P_{\mbf b'}} (X_{d'} (\mbf a')), \\
& i^*_{\mbf b'', \mbf a''}: \mathcal A_{\G_{d''}}(X_{d''} (\mbf b'') \times X_{d''} (\mbf a'')) \to \mathcal A_{P_{\mbf b''}} (X_{d''} (\mbf a'')).
\end{split}
\end{equation*}

Consider the  subset of $X_d(\mbf a)$:
\begin{align}
\label{X+}
X^+_{\mbf a, \mbf a', \mbf a''} = \{ \tilde V \in X_d (\mbf a) | \pi' (\tilde V)  \in X_{d'} (\mbf a') , \pi''( \tilde V) \in X_{d''} (\mbf a'')\}.
\end{align}
Then we have the following diagram
\[
\begin{CD}
X_d(\mbf a)  @< \iota << X^+_{\mbf a, \mbf a', \mbf a''} @>\pi >> X_{d'} (\mbf a') \times X_{d''} (\mbf a''),
\end{CD}
\]
where $\iota$ is the natural inclusion and $\pi (\tilde V) = (\pi' (\tilde V), \pi'' (\tilde V))$.
Thus the composition of the pullback $\iota^*$  of $\iota$ followed by the pushforward $\pi_!$  of $\pi$ defines a linear map
\begin{align}
\label{ik}
\pi_! \iota^* : \mathcal A_{P_{\mbf b}} (X_d(\mbf a)) \to \mathcal A_{P_{\mbf b'} \times P_{\mbf b''}} ( X_{d'} (\mbf a') \times X_{d''} (\mbf a'')),
\end{align}
where $\pi_!$ is defined by $\pi_! (f)  (\tilde V', \tilde V'')=\sum_{x \in X^+_{\mbf a, \mbf a', \mbf a''}: \pi (x) = (\tilde V', \tilde V'')} f(x) $, for all $\tilde V', \tilde V''$.
Clearly, we have an isomorphism of $\mathcal A$-modules
\[
\mathcal A_{P_{\mbf b'} \times P_{\mbf b''}} ( X_{d'} (\mbf a') \times X_{d''} (\mbf a''))
\cong
\mathcal A_{P_{\mbf b'} } ( X_{d'} (\mbf a') )
\otimes
\mathcal A_{ P_{\mbf b''}} ( X_{d''} (\mbf a'')).
\]
The following lemma makes connection between $\widetilde \Delta$ and $\pi_! \iota^*$.

\begin{lem}
\label{Delta-ik}
 We have the following commutative diagram.
\[
\xymatrixrowsep{.5in}
\xymatrixcolsep{.8in}
\xymatrix{
\mathcal A_{\G_d}(X_d(\mbf b) \times X_d(\mbf a))
\ar@{->}[r]^{i^*_{\mbf b, \mbf a}}
\ar@{->}[d]_{\widetilde \Delta_{\mbf b', \mbf a', \mbf b'', \mbf a''}}
&
\mathcal A_{P_{\mbf b}} (X_d(\mbf a))
\ar@{->}[d]^{ \pi_! \iota^*}
\\
%
%
\smxylabel{
\substack{
\mathcal A_{\G_{d'}} (X_{d'}(\mbf b') \times X_{d'} (\mbf a') ) \\
\otimes \\
\mathcal A_{\G_{d''}} (X_{d''}(\mbf b'') \times X_{d''} (\mbf a'') )
}
}
\ar@{->}[r]^-{ i^*_{\mbf b', \mbf a'} \otimes i^*_{\mbf b'', \mbf a''}}
&
\smxylabel{
\mathcal A_{P_{\mbf b'} } ( X_{d'} (\mbf a') )
\otimes
\mathcal A_{ P_{\mbf b''}} ( X_{d''} (\mbf a'')).
}
}
\]
\end{lem}

\begin{proof}
For any $f\in  \mathcal A_{\G_d} (X_d(\mbf b) \times X_d (\mbf a))$
and $(\widetilde V', \widetilde V'') \in X_{d'} (\mbf a') \times  X_{d''} (\mbf a'')$ , we have
\begin{align*}
\begin{split}
 \pi_! \iota^* i_{\mbf b,\mbf a}^* (f) (  \widetilde V', \widetilde V'')
&  = \sum_{ \widetilde V \in Z_{\widetilde V', \widetilde V''}}  i^*_{\mbf b, \mbf a}(f) (  \widetilde V)
= \sum_{ \widetilde V \in Z_{\widetilde V', \widetilde V''}} f( V, \widetilde V)   \\
& =
\widetilde \Delta_{\mbf b', \mbf a', \mbf b'', \mbf a''}
(f) (V', \widetilde V', V'',  \widetilde V'') \\
& =( i^*_{\mbf b', \mbf a'} \otimes i^*_{\mbf b'', \mbf a''})   \circ \widetilde \Delta_{\mbf b', \mbf a', \mbf b'', \mbf a''}
(f) (\widetilde V', \widetilde V'') .
\end{split}
\end{align*}
The lemma is thus proved.
\end{proof}

\begin{rem}
Note that $\pi$ is a vector bundle of rank $\sum_{1\leq i < j \leq n} a'_i a''_j$,
which is closely related to the twist in (\ref{Delta-A}).
\end{rem}

\subsection{Transfer map}
\label{Transfer}

To a pair $(V, V')$ in $X_d$, we can associate an $n$ by $n$ matrix $M=(m_{ij})$ with coefficients in $\mbb Z_{\geq 0}$ by
\begin{align}
\label{para}
m_{ij} = \left | \frac{V_{i} \cap V_{j}' }{ V_{i-1} \cap V_{j}' + V_{i} \cap V_{j-1}'} \right |, \quad \forall i, j\in [1, n].
\end{align}
Let $\Xi_d$ be the set of all matrices obtained this way.
The set $\Xi_d$ can be characterized by $M\in \Xi_d$ if and only if $m_{ij} \in \mbb Z_{\geq 0}$ and $\sum_{1\leq i, j \leq n} m_{ij} =d$.
It is shown in ~\cite{BLM90} that the set $\Xi_d$ parameterizes the $\G_d$-orbits in $X_d\times X_d$.
Let $\eta_M$ be the characteristic function of the $\G_d$-orbit in $X_d\times X_d$ indexed by $M$, for any $M\in \Xi_d$.
Let
\begin{align}
\label{chi}
\chi: \mbf S_n \to \mathcal A
\end{align}
be the algebra homomorphism defined by
$\chi ( \eta_M) = \det (M)$, for all $M \in \Xi_n$.
(Here $d$ is taken to be $n$.)
Let $\xi: \mbf S_{d - n} \rightarrow \mbf S_{d - n}$ be the $\mathcal A$-algebra isomorphism defined by
\[
\xi(f)(V, V') = \bv^{-\sum_{i=1}^n ( |V_i|- | V_i'|)} f(V,V'), \quad \forall f \in \mbf S_{d - n}, \ V, V' \in X_d.
\]
The transfer map
\begin{align}
\phi_{d, d-n, \bv}: \mbf S_d \to \mbf S_{d - n}, \quad \forall d \geq n,
\end{align}
is defined to be the composition
$
\xymatrix{ \mbf S_d\ar[r]^-{\Delta_{\bv}}& \mbf S_{d - n}  \otimes \mbf S_n
\ar[r]^-{\xi \otimes \chi}&  \mbf S_{d - n} \otimes \mathcal A =  \mbf S_{d - n}.}
$
The following lemma is quoted from ~\cite[Lemma 1.10]{L00}.

\begin{lem} \label{Transfer-A}
For any $i\in [1, n-1]$, we have
$\TA (\mbf E_i) = \mbf E_i'$, $\TA (\mbf F_i) = \mbf F_i'$ and $\TA (\mbf K_{ i}^{\pm 1} ) = \mbf K'^{\pm 1}_{i}$.
\end{lem}

\subsection{Generic version}
\label{Generic-SD}
Recall from  ~\cite{BLM90} that one can construct an associative algebra $ \mbb S_d$ over $\mbb A=\mbb Z[v, v^{-1}]$  such that
\[
\mbf S_d =\mathcal A \otimes_{\mbb A} \mbb S_d,
\]
where $\mathcal A$ is regarded as an $\mbb A$-module with $v$ acting as $\bv$.
More precisely, $\mbb S_d$ is a free $\mbb A$-module spanned by the symbols $\zeta_M$ for $M\in \Xi_d$ in Section  ~\ref{Transfer}.
The multiplication on $\mbb S_d$ is defined so that if
$
\zeta_{M_1} \zeta_{M_2} = \sum_{M \in \Xi_d} c^M_{M_1, M_2}(v) \zeta_M, \quad c^M_{M_1, M_2}(v) \in \mbb A,
$
then
$
\eta_{M_1} \eta_{M_2} = \sum_{M \in \Xi_d} c^M_{M_1, M_2}(v)|_{v=\bv}  \eta_M, \ \mbox{in $\mbf S_d$}.
$
In particular, $\mbb S_d$ has analogous elements of  $\mbf E_i$, $\mbf F_i$ and $\mbf K^{\pm 1}_{i}$, which we will use the same notations to denote them.
For a matrix $M=(m_{ij}) \in \Xi_d$, we set
\[
\ro (M)  = \left ( \sum_{j=1}^n m_{ij} \right )_{1\leq i\leq n}
\quad \mbox{and} \quad
\co (M) = \left ( \sum_{i=1}^n m_{ij} \right )_{1 \leq j \leq n}.
\]
Then we have a decomposition
\[
\mbb S_d = \oplus_{\mbf b, \mbf a \in \Lambda_{d, n}} \mbb S_d(\mbf b, \mbf a), \quad
\mbb S_d(\mbf b, \mbf a) = \mbox{span}_{\mbb A} \{ \zeta_M | \ro(M) =\mbf b, \co(M) =\mbf a\}.
\]
Note that
$\mbb S_d(\mbf b, \mbf a)$ is nothing but the generic version of $\mbf S_d(\mbf b, \mbf a)$.

By using the monomial basis in ~\cite[Theorem 3.10]{BLM90}, one can show that $\mbb S_d$ enjoys the same results for $\mbf S_d$ from the previous sections. Let us state them for later usages.
The following is the generic version of Proposition ~\ref{Delta_A}.

\begin{prop} \label{Generic-Delta_A}
Suppose that $d'+d''=d$.  There is
a unique algebra homomorphism
\begin{align}
\label{Generic-Delta}
\mbb \Delta: \mbb S_d \to \mbb S_{d'} \otimes \mbb S_{d''}
\end{align}
such that  $\mathcal A \otimes_{\mbb A} \mbb  \Delta = \Delta_{\bv}$ and satisfies the condition (\ref{D(A)-1}).
%
%
\end{prop}

The following is the generic version of Lemma  ~\ref{Transfer-A}, which is due to Lusztig ~\cite{L00}.

\begin{prop} \label{Generic-Transfer-A}
There is  a unique algebra homomorphism
\[
\phi_{d, d-n}: \mbb S_d \to \mbb S_{d - n}
\]
such that  $\mathcal A \otimes_{\mbb A} \phi_{d, d-n} = \phi_{d, d-n,\bv}$ and
\[
\phi_{d, d-n} (\mbf E_i) = \mbf E_i',\
\phi_{d, d-n} (\mbf F_i) = \mbf F_i',\
\phi_{d, d-n} (\mbf K^{\pm 1}_{i}) = \mbf K'^{\pm 1}_{i}, \quad \forall 1\leq i \leq n-1,
\]
\end{prop}


Recall the canonical basis $\mB_d \equiv \{ \{B\}_d| B\in \Xi_d\}$ of $\mbb S_d$ from ~\cite{BLM90}. We can write
\begin{align} \label{cB}
\mbb \Delta_{\mbf b', \mbf a', \mbf b'', \mbf a''} ( \{B\}_d)=
\sum_{B' \in \Xi_{d'}, B'' \in \Xi_{d''}}  c_B^{B', B''}  \{ B'\}_{d'} \otimes \{ B''\}_{d''}, \quad c_B^{B', B''} \in \mbb A.
\end{align}
We have the following proposition.

\begin{prop} \label{Pos-S}
$c_B^{B', B''} \in \mbb Z_{\geq 0} [v, v^{-1}]$.
\end{prop}

\begin{proof}
It suffices to show the generic version of $\widetilde \Delta$ has a similar positivity.
To establish the latter positivity, we switch from the finite field $\mbb F_q$ to its algebraic closure $\overline{\mbb F_q}$.
Let $\overline{\mbf G}_d$  be the general linear   group over $\overline{\mbb F_q}$ whose $\mbb F_q$-points form $\G_d$.
Similarly, we define an algebraic variety   $\overline{\mbf X}_d(\mbf a)$ over $\overline{\mbb F_q}$  for  $X_d(\mbf a)$.
We set $\overline{\mbf G}_m = \mrm{GL}(1, \overline{\mbb F_q})$.
For $d' + d'' = d$, we fix an isomorphism
$
\overline{\mbb F_q}^{d} \cong
\overline{\mbb F_q}^{d'} \oplus
\overline{\mbb F_q}^{d''}
$.
Via the isomorphism, we fix an imbedding $\overline{\mbf G}_m \to \overline{\mbf G}_d$ defined by
$t \mapsto (1_{\overline{\mbb F_q}^{d'}}, t 1_{\overline{\mbb F_q}^{d''}}
)$.
Thus $\overline{\mbf G}_m$ acts on $\overline{\mbf X}_d(\mbf a)$ via the imbedding.
It is straightforward to see that the fixed-point set  of $\overline{\mbf G}_m$ in $\overline{\mbf X}_d(\mbf a)$ is
$\sqcup_{\mbf a' + \mbf a'' = \mbf a } \overline{\mbf X}_{d'} (\mbf a') \times \overline{\mbf X}_{d''}(\mbf a'')$.
Moreover, the attracting set of  $\overline{\mbf X}_{d'} (\mbf a') \times \overline{\mbf X}_{d''}(\mbf a'')$, i.e., those points $x$ such that $\lim_{t\to 0} t. x \in \overline{\mbf X}_{d'} (\mbf a') \times \overline{\mbf X}_{d''}(\mbf a'')$,
 is exactly the algebraic variety whose $\mbb F_q$-point is $X^+_{\mbf a, \mbf a', \mbf a''}$ in (\ref{X+}).
Thus, the linear map $ \pi_! \iota^*$  in (\ref{ik})
is the function version of the hyperbolic localization functor attached to the data $(\overline{\mbf X}_d(\mbf a), \overline{\mbf G}_m)$ in \cite{B03}.
On the other hand,
the function $i^*_{\mbf b, \mbf a} (\{A\}_d)$ is nothing but the function version of the intersection cohomology complex attached to the $P_{\mbf b}$-orbit in $X_d(\mbf a)$ indexed by $A$.
Now the result in ~\cite{B03} says that  a hyperbolic localization functor sends a simple perverse sheaf to a semisimple complex. Therefore, we have the positivity for the generic version of  $ \pi_! \iota^*$, hence for the generic version of  $\widetilde \Delta$ and therefore the proposition.
\end{proof}

\subsection{Positivity for $\dot{\mbb U}$}
\label{Type-A-positivity}

By definition, the quantum $\mathfrak{sl}_n$,  denoted by $\mbb U \equiv \mbb U(\mathfrak{sl}_n)$,  is an associative  algebra over $\mbb Q( v)$  generated by the generators:
\[
\mbb E_i, \mbb F_i, \mbb K_i, \mbb K^{-1}_{i}, \quad \forall 1\leq i\leq n-1,
\]
and subject to the following relations.  For $1\leq i, j \leq n-1$,
\begin{align}
\label{q-Serre}
\begin{split}
\mbb K_{i} \mbb K^{-1}_{i} & =\mbb K^{-1}_{i} \mbb K_{i} =1, \\
\mbb K_{i}  \mbb K_{j}  &=   \mbb K_{j}  \mbb K_{i}, \\
\mbb  K_{i} \mbb E_{j} &
 = v^{2 \delta_{i, j} - \delta_{i, j+1} - \delta_{i, j-1} } \mbb E_{j} \mbb K_i, \\
 \mbb K_{i} \mbb F_{j}   &
= v^{-2\delta_{i, j} + \delta_{i, j+1} + \delta_{i, j-1} } \mbb F_{j} \mbb K_i, \\
\mbb E_{i} \mbb F_{j} - \mbb F_{j} \mbb E_{i} &= \delta_{i,j} \frac{ \mbb K_i
 - \mbb K^{-1}_{i} }{v-v^{-1}},          \\
\mbb E_{i}^2 \mbb E_{j} +\mbb E_{j} \mbb E_{i}^2 &= (v + v^{-1})   \mbb E_{i} \mbb E_{j} \mbb E_{i},   \ \text{if }  |i-j|=1,\\
\qquad \qquad  \mbb F_{i}^2 \mbb F_{j} + \mbb F_{j} \mbb F_{i}^2 &= (v + v^{-1})  \mbb F_{i} \mbb F_{j} \mbb F_{i},\ \text{if } |i-j|=1,\\
\mbb E_{i} \mbb E_{j} &= \mbb E_{j} \mbb E_{i},     \ \text{if } |i-j| \neq 1, \\
\mbb F_{i} \mbb F_{j}  &=\mbb F_{j}  \mbb F_{i},  \ \text{if } |i-j| \neq 1.
\end{split}
\end{align}

Moreover, $\mbb U$ admits a Hopf-algebra structure,  whose comultiplication is defined by
\begin{align} \label{D-K}
\begin{split}
\mbb \Delta ( \mbb E_i) & = \mbb E_i \otimes \mbb K_i + 1 \otimes \mbb E_i, \\
\mbb \Delta (\mbb F_i) & = \mbb F_i \otimes 1+ \mbb K^{-1}_{i} \otimes \mbb F_i, \\
\mbb \Delta (\mbb K_i) & = \mbb K_i\otimes \mbb K_i, \quad \forall 1 \leq i \leq n-1.
\end{split}
\end{align}

\begin{rem}
If one rewrites $\mbb E_i, \mbb F_i,$ and $ \mbb K_i$
as $E_i, F_i,$ and $K_{i} K_{i+1}^{-1}$, respectively, then
the resulting presentation is a subalgebra of  the quantum $\mathfrak{gl}_n$  used  in ~\cite[4.3]{BKLW13}.
\end{rem}

It is well-known from ~\cite{BLM90} that the assignments
\[
\mbb E_i \mapsto \mbf E_i,
\mbb F_i \mapsto \mbf F_i,
\mbb K_{\pm i} \mapsto \mbf K_{\pm i},
\quad \forall 1\leq i\leq n-1,
\]
define a surjective algebra homomorphism
\begin{align} \label{Psi}
\phi_d: \mbb U \to \ _{\mbb Q(v)} \mbb S_d,
\end{align}
where $_{\mbb Q(v)} \mbb S_d$ is the algebra obtained from $\mbb S_d$ in Section ~\ref{Generic-SD}
by extending the ground ring $\mbb A$ to $\mbb Q(v)$.
By using Proposition ~\ref{Generic-Delta_A},   (\ref{D-K}) and tracing along the generators, we obtain the following commutative diagram.
\begin{align} \label{U-S-raw}
\begin{CD}
\mbb U @> \mbb \Delta >> \mbb U \otimes \mbb U\\
@V\phi_d VV @VV \phi_{d'}\otimes \phi_{d''} V\\
_{\mbb Q(v)} \mbb S_d @>\mbb  \Delta >> _{\mbb Q(v)} \mbb S_{d'} \otimes \ _{\mbb Q(v)} \mbb S_{d''}
\end{CD}
\end{align}
where $d'+d''=d$ and $\Delta$ for $_{\mbb Q(v)}\mbb S_d$ is defined as  in (\ref{Generic-Delta}).

We define an equivalence relation $\sim$ on $\mbb Z^n$ by
$\mu \sim \nu$ if and only if $\mu - \nu = p (1, \cdots, 1)$ for some $p\in \mbb Z$.
Let
\[
\mbb X = \mbb Z^n/ \sim,
\]
be the set of all equivalence classes. Let $\overline{\mu}$ denote the equivalence class of  $\mu \in \mbb Z^n$.
Let
\[
\mbb Y = \{\nu \in \mbb Z^n | \sum_{1\leq i\leq n} \nu_i =0\}.
\]
Then the standard dot product on $\mbb Z^n$  induces a  pairing
$
\cdot: \mbb Y \times \mbb X \to \mbb Z.
$
Set $I=\{1,\cdots, n-1\}$. We define two injective maps
$
I \to \mbb Y, \quad I \to \mbb X,
$
by
$
i \mapsto - s_i + s_{i+1}, i \mapsto -\overline{s_i} + \overline{s_{i+1}}, \quad \forall 1\leq i\leq n-1,
$
respectively, where $s_i$ is the $i$-th  standard basis element in $\mbb Z^n$.
We thus obtain a root datum of type $\mbf A_{n-1}$  in  ~\cite[2.2]{L93}. It is both $\mbb X$-regular and $\mbb Y$-regular.

Recall from ~\cite[23.1.1]{L93} that $\mbb U$ admits a decomposition
$\mbb U = \oplus_{\nu \in \mbb Z[I]} \mbb U(\nu)$ defined by
the following rules:
\[
\mbb U(\nu') \mbb U(\nu'') \subseteq \mbb U(\nu' + \nu''),\
\mbb K_{\pm i} \in \mbb U(0),\  \mbb E_i \in \mbb U( i),\  \mbb  F_i \in \mbb U(- i).
\]
For a triple $\nu', \nu'', \nu$ in $\mbb Z[I]$ such that $\nu' + \nu'' =\nu$, we can have a linear map
\[
\mbb \Delta_{\nu', \nu''}: \mbb U (\nu) \to \mbb U(\nu') \otimes \mbb U(\nu''),
\]
obtained from $\mbb \Delta$ by restricting to $\mbb U(\nu)$ and projecting to $\mbb U(\nu') \otimes \mbb U(\nu'')$.
Moreover, the restriction of the algebra homomorphism  $\phi_d$ in (\ref{Psi}) to $\mbb U(\nu)$ induces a linear map, still denoted by $\phi_d$,
\[
\phi_d: \mbb U(\nu) \to \oplus_{\overline{\mbf b} -\overline{\mbf a} =\nu}  \ _{\mbb Q(v)} \mbb S_d(\mbf b, \mbf a).
\]
where $\mbb Z[I]$ is treated  as a subset in $\mbb X$ via the imbedding $I\to \mbb X$.

\begin{lem}
\label{Delta-U-S}
The commutative diagram (\ref{U-S-raw}) can be refined to the following commutative diagram.
\begin{align}
\label{U-S}
\begin{CD}
\mbb U(\nu)  @>\mbb \Delta_{ \nu', \nu''} >> \mbb U (\nu')  \otimes \mbb U (\nu'') \\
@V\phi_d VV @VV \phi_{d'}\otimes \phi_{d''} V\\
\oplus_{\overline{\mbf b} -\overline{\mbf a} =\nu} \ _{\mbb Q(v)} \mbb S_d (\mbf b, \mbf a)
@> \oplus \mbb \Delta_{ \mbf b', \mbf a', \mbf b'', \mbf a''}  >>
\oplus_{
\substack{
\mbf b' + \mbf b''=\mbf b, \mbf a' +\mbf a''=\mbf a  \\
\overline{\mbf b'} - \overline{\mbf a'}=\nu', \overline{\mbf b''} - \overline{\mbf a''} =\nu''
}
}
\  _{\mbb Q(v)} \mbb S_{d'} (\mbf b', \mbf a')  \otimes \ _{\mbb Q(v)} \mbb S_{d''}(\mbf b'', \mbf a'')
\end{CD}
\end{align}
where $ \mbb \Delta_{ \mbf b', \mbf a', \mbf b'', \mbf a''} $ is defined similar to
(\ref{Delta-A}).
\end{lem}

Now set
\begin{align*}
\begin{split}
\dot{\mbb U} & = \oplus_{\overline \mu, \overline  \lambda \in \mbb X}\  _{\overline \mu}\mbb U_{ \overline \lambda}, \\
_{\overline \mu} \mbb U_{\overline \lambda}
&= \mbb U / \left ( \sum_{1 \leq i \leq n-1}  (\mbb K_i - v^{-\mu_i + \mu_{i+1}} ) \mbb U   + \sum_{1\leq i \leq n-1} \mbb U ( \mbb K_i - v^{- \lambda_i +  \lambda_{i+1}}) \right ).
\end{split}
\end{align*}
This is the modified/idempotented form of $\mbb U$ defined in ~\cite[23.1.1]{L93},
see also  ~\cite{BLM90}. Recall from ~\cite[23.1.5]{L93},
the comultiplication $\mbb \Delta$ induces a linear map
\begin{align} \label{Delta-dot}
\mbb \Delta_{\overline{\mu'}, \overline{\lambda'}, \overline{\mu''}, \overline{\lambda''}}: \ _{\overline{\mu}} \mbb U_{\overline{\lambda}} \to \ _{\overline{\mu'}} \mbb U_{\overline{\lambda'}} \otimes \ _{\overline{\mu''}} \mbb U_{\overline{\lambda''}},
\end{align}
and makes  the following diagram commutative.
\begin{align} \label{U-U-dot}
\begin{CD}
\mbb U(\nu )  @> \mbb \Delta_{ \nu', \nu''} >> \mbb U(\nu') \otimes \mbb U(\nu'')\\
@V\pi_{\overline \mu, \overline \lambda} VV @VV\pi_{\overline{\mu'}, \overline{\lambda'}} \otimes \pi_{\overline{\mu''}, \overline{\lambda''}} V \\
 \ _{\overline{\mu}} \mbb U_{\overline{\lambda}} @>  \mbb \Delta_{\overline{\mu'}, \overline{\lambda'}, \overline{\mu''}, \overline{\lambda''}}>>   _{\overline{\mu'}} \mbb U_{\overline{\lambda'}} \otimes \ _{\overline{\mu''}} \mbb U_{\overline{\lambda''}}
\end{CD}
\end{align}
where
$\overline{\mu} - \overline{\lambda}=\nu$,
$\overline{\mu'} - \overline{\lambda'}=\nu'$,
$\overline{\mu''} - \overline{\lambda''}=\nu''$,
and
$\pi_{\overline{\mu},\overline{\lambda}}$ is the projection from $\mbb U$ to $_{\overline{\mu}} \mbb U_{\overline{\lambda}}$.

We write
$1_{\overline{\lambda}} = \pi_{\overline \lambda, \overline \lambda}(1)$.
It is well-known that $\dot{\mbb U}$ and $_{\mbb Q(v)} \mbb S_d$  are  $\mbb U$-bimodules.
So the notations  $\mE_i 1_{\overline \lambda}$ and   $\mF_i 1_{\overline \lambda}$ in $\dot{\mU}$ are meaningful,
and  so are  $\mE_i \zeta_M$, $\mF_i \zeta_M$ in $_{\mbb Q(v)} \mbb S_d$ where the notation $\zeta_{M}$ is from Section ~\ref{Generic-SD}.
Recall from  ~\cite{L00} (see also ~\cite{LW14}) that there is a surjective algebra homomorphism
$
\widetilde \phi_d: \dot{\mbb U} \to \ _{\mbb Q(v)} \mbb S_d
$
defined by
\begin{align*}
\begin{split}
\widetilde \phi_d( 1_{\overline \lambda}) & =
\begin{cases}
\zeta_{M_{\mbf a}}, & \mbox{if} \ \overline \lambda =\overline{\mbf a}, \ \mbox{for some}\  \mbf a \in \Lambda_{d, n},\\
0, & \mbox{o.w.}
\end{cases}\\
\widetilde \phi_d( \mbb E_i 1_{\overline \lambda}) & =
\begin{cases}
\mE_i \zeta_{M_{\mbf a}}, & \mbox{if} \ \overline \lambda =\overline{\mbf a}, \ \mbox{for some}\  \mbf a \in \Lambda_{d, n},\\
0, & \mbox{o.w.}
\end{cases}\\
\widetilde \phi_d( \mbb F_i 1_{\overline \lambda}) & =
\begin{cases}
\mF_i \zeta_{M_{\mbf a}}, & \mbox{if} \ \overline \lambda =\overline{\mbf a}, \ \mbox{for some}\  \mbf a \in \Lambda_{d, n},\\
0, & \mbox{o.w.}
\end{cases}
\end{split}
\end{align*}
where $M_{\mbf a}$ is the diagonal matrix whose diagonal is $\mbf a$.
Moreover, $\widetilde \phi_d$ induces a linear map, still denoted by $\widetilde \phi_d$,
\[
\widetilde \phi_d: \ _{\overline{\mbf b}} \mU_{\overline{\mbf a}} \to \ _{\mbb Q(v)} \mbb S_d(\mbf b, \mbf a).
\]
By definition,  we have the following lemma.

\begin{lem}
If $\overline \mu = \overline{\mbf b} $, $\overline \lambda =\overline{\mbf a} $ and
$\overline \mu -\overline \lambda =\nu$, then the following diagram is commutative.
\begin{align}
\label{U-U-dot-S}
\begin{split}
\begin{CD}
\mbb U( \nu)
@>>>
_{\overline \mu} \mbb U_{\overline{\lambda}}  \\
@V\phi_d VV @V \widetilde \phi_d VV\\
\oplus_{\overline{\mbf b}-\overline{\mbf a} = \nu} \ _{\mbb Q(v)} \mbb S_d(\mbf b, \mbf a)
@>>>  _{\mbb Q(v)} \mbb S_d(\mbf b, \mbf a)
\end{CD}
\end{split}
\end{align}
where the bottom row is the natural projection.
\end{lem}

Note that $\overline{\mbf b} -\overline{\mbf a} \in \mbb Z[I] \subseteq \mbb X$.
By piecing  together (\ref{U-S}), (\ref{U-U-dot}) and (\ref{U-U-dot-S}),  we have the following cube.
\begin{align}
\label{cube}
\begin{split}
\xymatrixrowsep{.1in}
\xymatrixcolsep{.1in}
\xymatrix{
& _{\overline \mu}\mbb U_{\overline \lambda}  \ar@{->}[rr]  \ar@{.>}[dd]
&&  _{\overline{\mu'}}\mbb U_{\overline{\lambda'}} \otimes _{\overline{\mu''}}\mbb U_{\overline{\lambda''}}
\ar@{->}[dd] \\
\mbb U(\nu) \ar@{->}[ur]^{\pi_{\overline \mu, \overline \lambda}}  \ar@{->}[rr]  \ar@{->}[dd]
& & \mbb U(\nu') \otimes \mbb U(\nu'') \ar@{->}[ur] \ar@{->}[dd] & \\
& \mbb S_d(\mbf b, \mbf a)  \ar@{.>}[rr]  && \mbb S_{d'}(\mbf b', \mbf a') \otimes \mbb S_{d''}(\mbf b'', \mbf a'') \\
\oplus \mbb S_d(\mbf b, \mbf a) \ar@{.>}[ur] \ar@{->}[rr]  && \oplus  \mbb S_{d'}(\mbf b', \mbf a') \otimes  \mbb S_{d''}(\mbf b'', \mbf a'') \ar@{->}[ur] &
}
\end{split}
\end{align}
where each of the $\mbb S$  in the bottom square has a subscript $\mbb Q(v)$ on the left.
In (\ref{cube}), the front square  is (\ref{U-S}), the top square  is (\ref{U-U-dot}),  the two side squares  are  (\ref{U-U-dot-S}) and the commutativity of the  bottom square is obvious.
Since $\pi_{\overline \mu, \overline \lambda}$ is surjective and
each square is commutative except  the one in the back,
we have immediately the following proposition by diagram-chasing.

\begin{prop}  \label{U-dot-S-A}
The  square in the back of the cube (\ref{cube}) is commutative.
\begin{align}
\label{U-dot-S}
\begin{CD}
_{\overline{\mbf b}} \mbb U_{\overline{\mbf a}}
@>\mbb \Delta_{\overline{\mbf b'}, \overline{\mbf a'}, \overline{\mbf b''}, \overline{\mbf a''}}  >>
_{\overline{\mbf b'}} \mbb U_{\overline{\mbf a'}} \otimes \ \! _{\overline{\mbf b''}}\mbb U_{\overline{\mbf a''}} \\
@V \widetilde \phi_d VV @VV \widetilde \phi_{d'}\otimes \widetilde  \phi_{d''} V\\
_{\mbb Q(v)} \mbb S_d(\mbf b, \mbf a) @>\mbb \Delta_{ \mbf b', \mbf a', \mbf b'', \mbf a''} >>
_{\mbb Q(v)} \mbb S_{d'}(\mbf b', \mbf a') \otimes \ _{\mbb Q(v)}\mbb S_{d''}(\mbf b'', \mbf a'')
\end{CD}
\end{align}
\end{prop}

By   using Proposition ~\ref{U-dot-S-A}, we can prove the following positivity of the canonical basis of $\dot{\mbb U} =\dot{\mbb U}(\mathfrak{sl}_n)$ with respect to the comultiplication.
Let $\mbb B$ be the canonical basis of $\dot{\mU}$ defined in ~\cite[25.2.4]{L93}.

\begin{thm} [Grojnowski] \label{conj-25.4.2}
Let $b \in \mbb B$.
If
$
\mbb \Delta_{\overline{\mu'}, \overline{\lambda'}, \overline{\mu''}, \overline{\lambda''}} (b)
= \sum_{b', b''} \hat m_{b}^{b', b''} b'\otimes b''
$,
then  $\hat m_b^{b', b''} \in \mbb Z_{\geq 0} [v, v^{-1}]$.
\end{thm}

\begin{proof}
Let $\mathcal I=\{ (b',  b'')|  \hat m_b^{b', b''} \neq 0\}$. Clearly, $\# \mathcal I < \infty$.
By ~\cite[Proposition 7.8]{M10}, we can find $d$, $d'$ and $d''$ large enough such that
\[
\widetilde \phi_d (b) =\{B\}_d, \ \widetilde \phi_{d'} (b') =\{B'\}_{d'},\
\widetilde \phi_{d''} (b'') =\{B''\}_{d''}, \quad \forall (b', b'')\in \mathcal I,
\]
where $\{B\}_d$, $\{B'\}_{d'}$ and $\{B''\}_{d''}$ are certain canonical basis elements
in $\mbb S_d$, $\mbb S_{d'}$ and $\mbb S_{d''}$, respectively.
Then by (\ref{U-dot-S}),  we have
\begin{align} \label{last-step}
\begin{split}
(\widetilde \phi_{d'}\otimes \widetilde  \phi_{d''}) \mbb \Delta_{\overline{\mbf b'}, \overline{\mbf a'}, \overline{\mbf b''}, \overline{\mbf a''}} (b)
&= \sum_{(b', b'') \in \mathcal I} \hat m_{b}^{b', b''}  \{ B'\}_{d'} \otimes \{ B''\}_{d''}
 = \mbb \Delta_{ \mbf b', \mbf a', \mbf b'', \mbf a''} ( \{B\}_d) .
\end{split}
\end{align}
By comparing (\ref{cB})   with (\ref{last-step}), $\hat m_{b}^{b', b''} =c_B^{B', B''}$ and hence   the theorem by Proposition ~\ref{Pos-S}.
\end{proof}

\begin{rem}
\label{Groj}
Theorem ~\ref{conj-25.4.2}  is conjectured in ~\cite[Conjecture 25.4.2]{L93} for all symmetric Cartan data.
\end{rem}

\section{Positivity for quantum  affine $\mathfrak{sl}_n$}

In this section, we shall lift the positivity result on quantum $\mathfrak{sl}_n$ in the previous section to its affine analogue.
We provide a new proof of the multiplication formula of Du-Fu ~\cite{DF13}.

\subsection{Results  from [L00]}
Similar to the previous section, we fix a pair $(d, n)$ of non-negative integers.
We set
\[
\widehat{\Lambda}_{d,n} =\{ \mbf{\lambda}=(\lambda_i)_{i\in \mbb Z} \in \mbb Z_{\geq 0}^{\mbb Z} | \lambda_i = \lambda_{i+n}, \forall i\in \mbb Z; \sum_{1\leq i\leq n} \lambda_i=d\}.
\]
Let $\widehat{\Xi}_d$ be the set of all $\mbb Z\times \mbb Z$ matrices $A=(a_{ij})_{i, j\in \mbb Z}$ such that $a_{ij}\in \mbb Z_{\geq 0}$, $a_{ij} = a_{i+n, j+n}$,
and $\sum_{1\leq i \leq n; j\in \mbb Z} a_{ij} = d$. To each matrix $A \in \widehat{\Xi}_d$, we can associate $r(A)$ and $c(A)$ in $\widehat{\Lambda}_{d, n}$ by
$r(A)_i = \sum_{j\in \mbb Z} a_{ij}$ and $c(A)_j = \sum_{i\in \mbb Z} a_{ij}$ for all $i, j \in \mbb Z$.

We need to switch the ground field from $\mbb F_q$ to the local field $\mbb F_q((\e))$.
Let $\mbb F_q[[\e]]$ be the subring of $\mbb F_q((\e))$ of all formal power series over $\mbb F_q$.
Suppose that $\V$ is a $d$-dimensional vector space over $\mbb F_q ((\e))$.
A free $\mbb F_q[[\e]]$-module $\mathcal L$ in $\V$ is called a lattice if $\mbb F_q((\e))\otimes_{\mbb F_q[[\e]]} \mathcal L =\V$.
A lattice chain $\mbf L = (\mbf L_i)_{i\in \mbb Z}$ of period $n$ is a sequence of lattices $\mbf L_i$ in $\V$ such that $\mbf L_i \subseteq \mbf L_{i+1}$ and $\mbf L_i = \e \mbf L_{i+n}$ for all $i\in \mbb Z$.
Let $\widehat X_d $ be the collection of all lattice chains in $\V$.
Let  $\widehat{\G}_d = \mrm{GL}(\V)$ act from the left on $\widehat X_d$ in the canonical way. Then we can form the algebra
$$
\widehat{\mbf S}_d = \mcal A_{\widehat{\G}_d} (\widehat X_d \times \widehat X_d),
$$
which is the so-called affine $\bv$-Schur algebra.
It is well-known that  the $\widehat{\G}_d$-orbits in $\widehat X_d \times \widehat X_d$ are parameterized by $\widehat{\Xi}_d$ via the assignment
$(\mbf L, \mbf L') \mapsto A$, where
$a_{ij} = \left | \frac{\mbf L_i\cap \mbf L'_j}{\mbf L_{i-1} \cap \mbf L'_j + \mbf  L_i \cap \mbf L'_{j-1}} \right |$ for all $i, j\in \mbb Z$.
So we have
\[
\widehat{\mbf S}_d = \mrm{span}_{\mcal A} \{ e_A | A\in \widehat{\Xi}_d\},
\]
where $e_A$ is the characteristic function of the $\widehat{\G}_d$-orbit indexed by $A$. Furthermore, we have
$\widehat{\mbf S}_d = \oplus_{\mbf b, \mbf a \in \widehat{\Lambda}_{d,n}} \widehat{\mbf S}_d  (\mbf b, \mbf a)$ where
$\widehat{\mbf S}_d (\mbf b, \mbf a)$ is  spanned by $e_A$ such that $r(A) = \mbf b$ and $c(A) = \mbf a$.

If one lifts the functions to the sheaf level, one gets the generic version $\aSB_d$ of $\widehat{\mbf S}_d$ such that
$\mcal A \otimes_{\mbb A} \aSB_d=\widehat{\mbf S}_d$. By abuse of notation, we write $e_A$ for the unique function $x$  in $\aSB_d$ such that
$\mcal A \otimes_{\mbb A} x =e_A$ (for all $q$).

The standard basis of $\aSB_d$ consists of elements $[A] = v^{- d_A} e_A$ where $d_A = \sum_{\substack{1\leq i \leq n\\ i \geq k, j < l}} a_{ij} a_{kl}$.
Recall the Bruhat order  $\preceq$ on $\widehat{\Xi}_d$ from ~\cite{L00}: $A \preceq B$ if and only if
\begin{align*}
\sum_{i \geq r, j \leq s} a_{rs} \leq \sum_{i\geq r, j \leq s} b_{rs}, \quad \forall i < j \in \mbb Z ; \quad
\sum_{i \leq r, j \geq s} a_{rs} \leq \sum_{i\leq r, j \geq s} b_{rs}, \quad \forall i > j \in \mbb Z.
\end{align*}

Following ~\cite{L00}, one can associate a bar involution $\bar \ : \aSB_d \to \aSB_d$ such that
$\overline{[A]} = [A] + \sum_{A' \preceq A, A'\neq A} C_{A, A'} [A']$ where $C_{A, A'} \in \mbb A$.

The canonical basis $\{A\}_d$ for all $A\in \widehat{\Xi}_d$  of $\aSB_d$ is defined by the properties that $\overline{\{A\}_d} = \{A\}_d$ and $\{A\}_d = [A] + \sum_{A' \preceq A, A'\neq A} P_{A, A'} [A']$ where $P_{A, A'} \in v^{-1} \mbb Z[v^{-1}]$.

Let $E^{i,j}$ be the $\mbb Z\times \mbb Z$ matrix whose $(k, l)$-th entry is $1$ if $(k, l) = (i, j)$ mod $n$, and zero otherwise.
For any $i\in \mbb Z$, we define the following elements in $\aSB_d$:
\begin{equation}\label{generator-affA}
\begin{split}
\mbf E_i  & = \sum_{A - E^{i+1, i} \ \mbox{\small{diagonal}}}  [A], \quad
\mbf F_i  = \sum_{A - E^{i, i+1}\ \mbox{\small{diagonal}}} [A], \\
\mbf H_{ i}^{\pm 1}  & =  \sum_{A \ \mbox{\small{diagonal}}} v^{\pm c(A)_i} [A], \quad
\mbf K_{i}^{\pm 1}   = \mbf H_{i+1}^{\pm 1}  \mbf H_{i}^{\mp 1}, \quad \forall i\in \mbb Z.
\end{split}
\end{equation}
By periodicity, we have
$
\mbf E_i = \mbf E_{i+n},
\mbf F_i = \mbf F_{i+n},
\mbf H_i^{\pm 1} = \mbf H_{i+n}^{\pm 1},
$
and
$
\mbf K^{\pm 1}_i = \mbf K^{\pm 1}_{i+n},
$
for all $i \in \mbb Z$.
The following lemma is from ~\cite{L00}.

\begin{lem}
\label{D(affA)}
There is an algebra homomorphism $\widetilde \Delta: \aSB_d \to \aSB_{d'} \otimes \aSB_{d''}$ for $d'+ d'' =d $ with
\begin{align*}
\begin{split}
\widetilde \Delta ( \mbf E_i) & = \mbf E'_i \otimes  \mbf H''_{i+1} + \mbf H'^{-1}_{i+1} \otimes \mbf E''_i, \\
\widetilde \Delta (\mbf F_i) & = \mbf F'_i \otimes \mbf H_{i}''^{-1}+ \mbf H_i' \otimes \mbf F_i'', \\
\widetilde \Delta (\mbf K_i) & = \mbf K_i' \otimes \mbf K_i'', \quad \forall i\in \mbb Z.
\end{split}
\end{align*}
\end{lem}

\subsection{The coproduct  $\mbb \Delta$}

Recall the algebra homomorphism from Lemma ~\ref{D(affA)}.
If $\mbf b=\mbf b' + \mbf b''$ and $\mbf a = \mbf a' + \mbf a''$,  let
$\widetilde \Delta_{\mbf b', \mbf a', \mbf b'', \mbf a''}: \aSB_d(\mbf b, \mbf a) \to \aSB_{d'} (\mbf b', \mbf a') \otimes \aSB_{d''}(\mbf b'', \mbf a'')$
be
the composition of  the restriction of $\aSB_d$ to $\aSB_d(\mbf b, \mbf a)$
and the projection to the component $ \aSB_{d'} (\mbf b', \mbf a') \otimes \aSB_{d''}(\mbf b'', \mbf a'')$.
We set
\begin{align}
\label{Delta-dagger-affA}
\Delta^{\dagger}_{\mbf b', \mbf a', \mbf b'', \mbf a''}
=  v^{ \sum_{1\leq i\leq j\leq n} b_i' b_j'' - a_i' a_j''}  \widetilde \Delta_{\mbf b', \mbf a', \mbf b'', \mbf a''}, \quad
\Delta^{\dagger} = \oplus  \Delta^{\dagger}_{ \mbf b', \mbf a', \mbf b'', \mbf a''},
\end{align}
where the sum runs over all quadruples $(\mbf b', \mbf a', \mbf b'', \mbf a'')$ where
$\mbf b', \mbf a' \in \widehat \Lambda_{d', n}$ and $ \mbf b'', \mbf a'' \in \widehat \Lambda_{d'', n}$.

\begin{prop} \label{Delta-dagger_affA}
The linear map $\Delta^{\dagger}_{\bv}$ in (\ref{Delta-affA})  is an algebra homomorphism. Moreover,
\begin{align*}
\begin{split}
\Delta^{\dagger} ( \mbf E_i) & = v^{\delta_{i, n} d''} \mbf E'_i \otimes \mbf K_i'' + 1 \otimes v^{-\delta_{i, n}d'}  \mbf E''_i, \\
\Delta^{\dagger} (\mbf F_i) & = v^{-\delta_{i, n} d''}  \mbf F'_i \otimes 1+ \mbf K'^{-1}_{i} \otimes v^{\delta_{i, n} d'} \mbf F_i'', \\
\Delta^{\dagger}  (\mbf K_i) & = \mbf K_i' \otimes \mbf K_i'', \quad \forall i\in [1, n].
\end{split}
\end{align*}
\end{prop}

\begin{proof}
The case when $i \in [1, n-1]$ is proved in the same manner as the finite case in Proposition ~\ref{Delta_A}.
We now prove  the case when $i=n$.
Suppose that $\mbf b'' =\mbf a''$, and $(\mbf b', \mbf a')$ is chosen such that $b'_i = a'_i - \delta_{i, n} + \delta_{i, 1}$ for all $1\leq i \leq n$. Then the twist $\sum_{1\leq i\leq j\leq n} b_i' b_j'' - a_i' a_j'' $ is equal to $d'' - a_n''$.
This implies that the term $ \mbf E'_i \otimes  \mbf H''_{i+1}$ in $\widetilde \Delta (\mbf E_i)$ becomes
$ v^{\delta_{i, n} d''} \mbf E'_i \otimes \mbf K_i''$.
For the term $\mbf H'^{-1}_{i+1} \otimes \mbf E''_i$ in $\widetilde \Delta (\mbf E_i)$,
the twist contributes $a_i' - d'$ for a quadruple
$(\mbf b', \mbf a', \mbf b'', \mbf a'')$ such that $\mbf b' =\mbf a'$ and $b''_i =  a''_i  - \delta_{i, n} + \delta_{i, 1}$ for all $1\leq i \leq n$.
The formula for $\Delta^{\dagger} (\mbf E_i)$ is proved.

The proof for the formula $\Delta^{\dagger} (\mbf F_i)$ is entirely similar.
The formula for $\Delta^{\dagger} (\mbf K_i)$ is obvious.
\end{proof}

We set
\begin{align}
\label{eiA}
\e_i(A)  = \sum_{r \leq i < s} a_{r,s} - \sum_{r > i \geq s} a_{r, s}, \quad \forall i\in \mbb Z, A\in \widehat{\Xi}_d.
\end{align}
We define a linear map
\begin{align} \label{xi_dic}
\xi_{d, i, c}: \aSB_d \to \aSB_d, \quad \forall i, c \in \mbb Z,
\end{align}
by
$\xi_{d, i, c} ( [A]) = v^{c \e_i(A)} [A]$.
By ~\cite{L00},  $\xi_{d, i, c}$ is an algebra isomorphism with inverse $\xi_{d, i, -c}$. Set
\begin{align}
\label{Delta-affA}
\mbb \Delta : \aSB_d \to \aSB_{d'} \otimes \aSB_{d''}
\end{align}
to be the composition
$
\begin{CD}
 \aSB_d @> \Delta^{\dagger} >> \aSB_{d'} \otimes \aSB_{d''}
@> \xi_{d', n, d''} \otimes \xi_{d'', n, -d'} >>
\aSB_{d'} \otimes \aSB_{d''}.
\end{CD}
$

\begin{prop} \label{Delta_affA}
The linear map $\mbb \Delta$ in (\ref{Delta-affA})  is an algebra homomorphism. Moreover,
\begin{align*}
\begin{split}
\mbb \Delta ( \mbf E_i) & = \mbf E'_i \otimes \mbf K_i'' + 1 \otimes \mbf E''_i, \\
\mbb \Delta (\mbf F_i) & = \mbf F'_i \otimes 1+ \mbf K'^{-1}_{i} \otimes \mbf F_i'', \\
\mbb \Delta  (\mbf K_i) & = \mbf K_i' \otimes \mbf K_i'', \quad \forall i\in \mbb Z.
\end{split}
\end{align*}
\end{prop}

\begin{proof}
We have
$\xi_{d, n, c} (\mbf E_i) = v^{- c \delta_{i, n} } \mbf E_i$,
$\xi_{d, n, c} (\mbf F_i) = v^{ c \delta_{i, n} } \mbf F_i$ and
$\xi_{d, n, c} (\mbf K_i) = \mbf K_i$.
Proposition follows from these computations and the formulas in Proposition ~\ref{Delta-dagger_affA}.
\end{proof}

\subsection{The compatibility of  $\xi_{d, i, c}$ and the canonical basis}

We have

\begin{thm} \label{xi-positivity}
$\xi_{d, i, c} (\{A\}_d) = v^{c \e_{i}(A)} \{A\}_d$ where $\xi_{d, i, c}$ is in (\ref{xi_dic}).
\end{thm}

Theorem \ref{xi-positivity} follows from the following critical observation.

\begin{thm} \label{PAA'}
Write $\{A\}_d = \sum_{A' \preceq A} P_{A, A'} [A']$. If $P_{A, A'} \neq 0$, then $\e_{i} (A) = \e_{i} (A')$ for all $i$.
\end{thm}

We make two remarks before we prove Theorem ~\ref{PAA'}.

\begin{rem}
\label{L00-xi}
The algebra isomorphism $\prod_{1\leq i \leq n} \xi_{d, i, -1}$ is the linear map $\xi$ in ~\cite[1.7]{L00}.
In view of Theorem ~\ref{xi-positivity}, we have $\xi (\{A\}_d) = v^{-\sum_{1\leq i\leq n} \e_{i}(A)} \{A\}_d$.
\end{rem}

\begin{rem}
Even  if $A' \prec A$,  $\e_{i} (A') $ may not be the same as $\e_{i} (A)$.
For example, take $A' = 2 \sum_{1\leq i \leq n} E^{i, i}$ and $A = \sum_{1\leq i \leq n} E^{i, i}  + E^{i, i+1}$.
Then we have $A' \prec A$, $\e_i (A') =0$ and $\e_i (A) =1$ for all $1 \leq i \leq n$.
\end{rem}

The remaining part of this section is devoted to the proof of Theorem ~\ref{PAA'}.
The main ingredient is a connection between the numerical data $\e_i (A)$ in (\ref{eiA})  and the multiplication formulas in ~\cite{DF13},
which we shall recall and provide a new proof.
Before we state the formula, we need to recall a lemma from  \cite[Section 2.2]{S06} as follows.

\begin{lem}\label{lemcounting}
Let $V$ be a finite dimensional vector space over $\mathbb F_q$.
Fix a flag  $(V_i)_{1\leq i \leq n}$ in $V$ such that $\left | V_i/V_{i-1} \right | = l_i$ for all $1\leq i\leq n$.
The number of subspaces $W \subset V$ such that $| W \cap V_i| = \sum_{j = 1}^i a_j$ for all $1\leq i\leq n$ is given by
$q^{\sum_{n \geq i>j \geq 1}a_i(l_j-a_j)}\prod_{i=1}^n
\begin{bmatrix}
l_i \\ a_i
\end{bmatrix}$ where
$\begin{bmatrix} l_i \\ a_i \end{bmatrix} =
\prod_{1 \leq j \leq a_i} \frac{q^{l_i - j +1} -1}{q^j-1}$.
\end{lem}

The following multiplication formula in $\aSB_d$ is first obtained in ~\cite{DF13}.
To a matrix $T=(t_{ij})_{i, j \in \mbb Z} $, we set $\check T = (\check t_{ij})_{i,j\in \mbb Z}$ where $\check t_{ij}=t_{i-1,j}$.

\begin{prop}
\label{DF-mult}
{\rm (1)}
Suppose that  $ A =(a_{ij})$, $B = (b_{ij}) \in \widehat{\Xi}_d$ satisfy that $c(B) = r(A)$ and
$B - \sum_{i=1}^n \alpha_i E^{i, i+1}$ is  diagonal for some $\alpha_i \in \mbb Z_{\geq 0}$.    We have
\[
e_B * e_A =
\sum_T q^{\sum_{1\leq i \leq n, j > l} (a_{ij}- t_{i-1, j}) t_{i, l}} \prod_{1 \leq i \leq n, j\in \mbb Z}
\begin{bmatrix}
a_{ij} + t_{ij} - t_{i-1, j} \\ t_{ij}
\end{bmatrix}
e_{A + T - \check T},
\]
where the sum runs over all $T = (t_{ij})$ such that $t_{i+n, j+n} =t_{ij}$ and $r(T)_i =\alpha_i$ for all $1\leq i\leq n$.

{\rm  (2)}
If $C \in \widehat{\Xi}_d$ satisfies that $c( C)=r(A)$ and  $C - \sum_{i=1}^n \beta_i E^{i+1, i}$ is diagonal, then
\[
e_C * e_A = \sum_{T} q^{\sum_{1\leq i\leq n, j < l } (a_{ij} - t_{ij}) t_{i-1, l}} \prod_{1\leq i \leq n, j\in \mbb Z}
\begin{bmatrix}
a_{ij} - t_{ij} + t_{i-1,j} \\
t_{i-1, j}
\end{bmatrix}
e_{A - T + \check T}
\]
where  the sum runs over  all $T$ such that $t_{ij} = t_{i+n, j+n}$ and $r(T)_i = \beta_i$ for all $1\leq i \leq n$.
\end{prop}

\begin{proof}
(1) It suffices to show the similar statement in $\widehat{\mbf S}_d$.
Let $A'=(a_{ij}')_{i, j\in \mbb Z}$ be a matrix in $\widehat{\Xi}_d$ such that  $r(B) = r(A')$ and $c(A) = c(A')$.
Let $\mathcal O_{A'}$ be the $\widehat{\G}_d$-orbit in $\widehat{X}_d\times \widehat{\X}_d$ indexed by $A'$.
Fix  $(\mbf L, \mbf L') \in \mathcal O_{A'}$, and we denote
\[
Z=\{ \mbf L'' \in \widehat{X}_d  \ |\ \mbf L_{i-1} \subset \mbf L''_i \overset{\alpha_i}{\subset} \mbf L_i, \forall 1\leq i\leq  n\}.
\]
Note that  $(\mbf L, \mbf L'') \in \mathcal O_B$ if and only if $\mbf L'' \in Z$.
Clearly, $Z$ has a partition $Z = \bigsqcup_T Z_T$ where
\[
Z_T =
\{
\mbf L'' \in Z|  |\mbf L''_i \cap ( \mbf L_{i-1} + (\mbf L_i \cap \mbf L'_j)) / \mbf L_i'' \cap ( \mbf L_{i-1} + (\mbf L_i \cap \mbf L'_{j-1}))| =  a'_{ij} -  t_{ij},
\forall i, j \in \mbb Z
\}.
\]
and the union runs over all $T$ such that $t_{i+n, j+n} =t_{ij}$ and $r(T)_i= \alpha_i$ for all $1\leq i\leq n$.
For each $\mbf L'' \in Z_T$, we have the following identities.
\begin{equation}\label{eqdim}
\begin{split}
& a_{ij} = |\mbf L''_i \cap \mbf L'_j/ \mbf L''_i \cap \mbf L'_{j-1}| - |\mbf L''_{i-1} \cap \mbf L'_j/ \mbf L''_{i-1} \cap \mbf L'_{j-1}|,\\
& a'_{ij} = |\mbf L_i \cap \mbf L'_j/ \mbf L_i \cap \mbf L'_{j-1}| - |\mbf L_{i-1} \cap \mbf L'_j/ \mbf L_{i-1} \cap \mbf L'_{j-1}|,\\
& a'_{ij} - t_{ij}  = |\mbf L''_i \cap \mbf L'_j/ \mbf L''_i \cap \mbf L'_{j-1}| - |\mbf L_{i-1} \cap \mbf L'_j/ \mbf L_{i-1} \cap \mbf L'_{j-1}|,
\end{split}
\end{equation}
where the last identity follows from the definition of $Z_T$. By (\ref{eqdim}), we have
$$
t_{i j} = |\mbf L_i \cap \mbf L'_j / \mbf L_i \cap \mbf L'_{j-1}| - |\mbf L''_i \cap \mbf L'_j / \mbf L''_i \cap \mbf L'_{j-1}| .
$$
Thus $a'_{ij} -t_{ij} = a_{ij} - t_{i-1, j}$, i.e., $A' = A + T - \check T$.
Summing up the above analysis, we have
\begin{align}
\label{eba}
\begin{split}
e_B * e_A (\mbf L, \mbf L') =\sum_{\mbf L'' \in \widehat{\Xi}_d} e_B( \mbf L, \mbf L'') e_A(\mbf L'', \mbf L')
= \sum_{\mbf L'' \in Z} e_A (\mbf L'', \mbf L') \\
=\sum_{T} \sum_{\mbf L'' \in \mbb Z_T} e_A (\mbf L'', \mbf L')
= \sum_{T} \# Z_T \ e_{A+ T-\check T} (\mbf L, \mbf L').
\end{split}
\end{align}

So it is reduced to compute the cardinality of $Z_T$.
For each $i\in [1, n]$, we set $\mbb Z(i)=\{ k \in \mbb Z| k\in [1,i] \ \mbox{mod} \ n\}$ and
we define $Z_T^{[1, i]} $ to be the set of all lattice chains $\mbf L''=(\mbf L''_k)_{k\in \mbb Z(i)}$ such that $\mbf L''_k$ satisfies
$|\mbf L_{k -1} + \mbf L''_k \cap \mbf L'_j/ \mbf L_{k-1} + \mbf L''_k \cap \mbf L'_{j-1} | = a'_{ij} - t_{ij}$ for all $j\in \mbb Z$.
Consider
\[
\begin{CD}
Z_T = Z_T^{[1, n]} @>\pi_n >> Z^{[1,n-1]}_T @>\pi_{n-1}>> \cdots @>\pi_2>> Z^{[1, 1]}_T @>\pi_1>> \bullet,
\end{CD}
\]
where $\pi_i ((\mbf L''_k)_{k\in \mbb Z(i)}) = (\mbf L''_k)_{k\in \mbb Z(i-1)}$ and the equality is due to
$\mbf L''_i \cap ( \mbf L_{i-1} + (\mbf L_i \cap \mbf L'_j))= \mbf L_{i-1} + \mbf L''_i \cap \mbf L'_j$
We observe that the fiber of $\pi_i$ gets identified with the set of subspaces $W$ in $\mbf L_i/\mbf L_{i-1}$ such that
$| W\cap (\mbf L_{i-1} + \mbf L_i \cap \mbf L'_j)/\mbf L_{i-1} / W\cap (\mbf L_{i-1} + \mbf L_i \cap \mbf L'_{j-1})/\mbf L_{i-1} | = a'_{ij}-t_{ij}$.
Observe that
$| (\mbf L_{i-1} + \mbf L_i \cap \mbf L'_j)/\mbf L_{i-1}| - | (\mbf L_{i-1} + \mbf L_i \cap \mbf L'_{j-1})/\mbf L_{i-1} | = a'_{ij}$, and by
applying  Lemma \ref{lemcounting}, we have
$
\# \pi_i^{-1}(\mbf L^i) =
q^{ \sum_{l < j} (a'_{ij} - t_{ij}) t_{il}} \prod_{ j\in \mbb Z }
\begin{bmatrix}
 a'_{ij}\\ t_{ij}
 \end{bmatrix},
 $
where $\mbf L^i$ is any element in $Z^{[1, i-1]}_T$.
So $\pi_i$ is surjective with constant fiber. Hence
\begin{equation}\label{eqzt}
\#Z_T = \prod_{1\leq i\leq n} \# \pi_i^{-1}(\mbf L^i)
=q^{ \sum_{1 \leq i \leq n, l < j} (a'_{ij} - t_{ij}) t_{il}} \prod_{1 \leq i \leq n, j\in \mbb Z }\begin{bmatrix}
    a'_{ij}\\ t_{ij}
\end{bmatrix}.
\end{equation}
The statement (1) follows from (\ref{eba}) and (\ref{eqzt}).

Let us prove (2).
Let $A'$ be a matrix such that $r(A')=r( C)$ and $c(A') = c(A)$.
Fix $(\mbf L, \mbf L') \in \mathcal O_{A'}$. We consider the set
$
Y  = \{\mbf  L'' \widehat X_d |\mbf  L_{i-1} \overset{\beta_{i-1}}{\subseteq} \mbf  L''_{i-1} \subseteq \mbf L_i, \forall 1\leq i \leq n\}.
$
Then $Y$ admits a partition $Y = \sqcup Y_T$, where
\[
Y_T =\{ \mbf L''  \in Y|
\left |
\mbf L_{i-1} + \mbf L''_{i-1} \cap \mbf L'_j / \mbf L_{i-1} + \mbf L''_{i-1}\cap  \mbf L'_{j-1}
\right |
= t_{i-1, j}, \quad \forall i, j \in \mbb Z
\}.
\]
By applying Lemma \ref{lemcounting} and arguing similar to (1), we have
\[
\# Y_T = \prod_{1\leq i \leq n} q^{\sum_{l > j} t_{i-1, l} (a'_{ij} - t_{i-1, j})} \prod_{j\in \mbb Z}
\begin{bmatrix}
a'_{ij}\\
t_{i-1,j}
\end{bmatrix}.
\]
Moreover, for $L''\in Y_T$ such that $(\mbf L'', \mbf L') \in \mathcal O_{A}$ if and only if
$A' = A - T + \check T$
Therefore, we have (2). The proposition is thus proved.
\end{proof}

\begin{rem}
If $n=1$, Proposition ~\ref{DF-mult} shows that $e_B * e_A = e_A * e_B$.
(Here we use $(e_A * e_B)^t = e_{B^t} e_{A^t}$.) This implies that $\aSB_d$ is commutative for $n=1$, which
corresponds to the geometric Satake of type $A$.
\end{rem}

\begin{lem} \label{e-mult}
Suppose that
$[B] * [A] =\sum Q_{B, A}^C [C]$.
If  $B = \sum_{1\leq j \leq n} \beta_j E^{j, j} + \alpha_j E^{j, j+1} $ or
$ \sum_{1\leq j \leq n} \beta_j E^{j, j} + \alpha_j E^{j+1, j}$,
and  $Q_{B, A}^C \neq 0$, then $\e_i (A) + \e_i(B)  = \e_i ( C)$ for all $i\in \mbb Z$.
\end{lem}

\begin{proof}
Assume that $B = \sum_{1\leq j \leq n} \beta_j E^{j, j} + \alpha_j E^{j, j+1} $. Then we have $\e_i(B) = \alpha_i$.
If $Q_{B, A}^C \neq 0$, then by Proposition ~\ref{DF-mult} (1), the matrix $C$ is of the form $A + T - \check T$.
Thus we have
\begin{align*}
\begin{split}
\e_i ( C) & = \e_i (A) + \e_i (T - \check T)
= \e_i (A) + \sum_{r\leq i < s} t_{r,s} - \check t_{r,s} + \sum_{r>  i \geq  s} t_{r,s} - \check t_{r,s} \\
& = \e_i (A) + \sum_{i < s} t_{i,s} - \sum_{i \geq s} - \check t_{i+1,s}
= \e_i (A) + \sum_{s\in \mbb Z} t_{i,s}
= \e_i (A) + \alpha_i = \e_i (A) + \e_i (B).
\end{split}
\end{align*}
Therefore, the lemma holds for $B = \sum_{1\leq j \leq n} \beta_j E^{j, j} + \alpha_j E^{j, j+1} $.

For the case when $B = \sum_{1\leq j \leq n} \beta_j E^{j, j} + \alpha_j E^{j+1, j}$, then $\e_i (B) = - \alpha_i$ and $C $ is of the form $A - T + \check T$ if $Q^C_{B, A} \neq 0$ by Proposition ~\ref{DF-mult} (2). So we have
$
\e_i ( C)  = \e_i (A) - \e_i (T - \check T)
= \e_i (A) -  \alpha_i = \e_i (A) + \e_i (B).
$
Therefore the lemma holds in this case. We are done.
\end{proof}

Next we introduce a second numerical data.
We define
\begin{align*}
\begin{split}
\deg_i \left ( \sum_{1\leq j \leq n} \beta_j E^{j, j} + \alpha_j E^{j+1, j} \right )   = - \alpha_i , \quad
\deg_i \left  ( \sum_{1\leq j \leq n} \beta_j E^{j, j} + \alpha_j E^{j, j+1} \right ) = \alpha_i .
\end{split}
\end{align*}

Suppose that $M= [A_1] * \cdots * [A_m]$ is a $monomial$ in $[A_j]$  where $A_j$ is either
$\sum_{1\leq k \leq n} \beta_{jk} E^{k,k} + \alpha_{jk} E^{k+1, k}$, or $ \sum_{1\leq k \leq n} \beta_{jk} E^{k, k} + \alpha_{jk} E^{k, k+1} $.
We define
$
\deg_i (M) = \sum_{1\leq j \leq m} \deg_i (A_j).
$
To the same monomial $M$, we also define its length  $\ell (M)$  to be
$
\ell(M) = \sum_{1 \leq j \leq m} \sum_{1\leq k \leq n} \alpha_{jk}.
$
(We define $[A]$ to be  a monomial of length zero if $A$ is diagonal.)
Then we have

\begin{lem} \label{M-[A]}
Let  $M$ be a monomial and write $M = \sum R_A [A]$. If $R_A \neq 0$, then $\deg_i (M) = \e_{i} (A) $ for all $i\in \mbb Z$.
\end{lem}

\begin{proof}
We argue by induction on the length $\ell (M)$ of $M$.
When $\ell(M)=1$, the lemma follows from the definitions.
Assume now that $\ell (M) > 1$ and  the lemma holds for any monomial $M'$ such that $\ell (M') < \ell (M)$.
We write $M = [A_1] * [M']$, where $A$ is either
$\sum_{1\leq j \leq n} \beta_j E^{j, j} + \alpha_j E^{j+1, j}$ or
$\sum_{1\leq j \leq n} \beta_j E^{j, j} + \alpha_j E^{j, j+1}$, and
$M'$ is a monomial of the remaining terms in $M$. Thus $\ell (M') < \ell (M)$.
Suppose that $M' = \sum R_{A'} [A']$, then
we have
\[
M = [A_1] * M' = \sum_{A'} R_{A'} [A_1] * [A']
=\sum_{A', B} R_{A'} Q^B_{A_1, A'} [B].
\]
If $A_1 = \sum_{1\leq j \leq n} \beta_j E^{j, j} + \alpha_j E^{j+1, j}$, then by Lemma ~\ref{e-mult} and induction hypothesis,  we have
\[
\deg_i (M) = - \alpha_i + \deg_i (M')
=- \alpha_i + \e_i (A') = \e_i (B), \quad \mbox{if} \ Q^B_{A_1, A'} \neq 0, R_{A'} \neq 0.
\]
Similarly, if $A_1 =\sum_{1\leq j \leq n} \beta_j E^{j, j} + \alpha_j E^{j, j+1}$,  then
\[
\deg_i (M) = \alpha_i + \deg_i (M')
= \alpha_i + \e_i (A') = \e_i (B), \quad \mbox{if} \ Q^B_{A_1, A'} \neq 0, R_{A'} \neq 0.
\]
Lemma follows.
\end{proof}

By a result in ~\cite{DF13} (see also ~\cite{LL}), there exists a monomial $M_A$ such that
\begin{align} \label{M_AA}
M_A = [A] + \sum_{A' \prec A}  S_{A, A'} [A'], \quad \mbox{for some} \ S_{A, A'} \in \mbb Z[v, v^{-1}].
\end{align}
Since $[A]$ forms a basis for $\aSB_d$, the monomial $M_A$ forms a basis for $\aSB_d$. In particular,
\begin{align} \label{[A]-mono}
[A] = M_A + \sum_{A' \prec A} R_{A, A'} M_{A'}, \quad \mbox{for some} \ R_{A, A'} \in \mbb Z[v, v^{-1}].
\end{align}

Moreover, we have

\begin{lem} \label{SAA'}
Suppose that $R_{A, A'} \neq 0$ in (\ref{M_AA}), then $\e_i(A) = \e_i (A')$.
\end{lem}

\begin{proof}
We prove by induction with respect to $\preceq$ in descending order. If $A' =A$, it is trivial.
Suppose that for all $A''$ such that $A' \prec A'' \preceq A$, the statement holds.
If $[A']$ appears in $M_{A''}$ for some $A''$ such that $A' \prec A''$, then
$\e_i(A) = \e_i (A'') = \deg_i (M_{A''}) = \e_i(A')$ by induction hypothesis.
If $[A']$ does not appear in  $M_{A''}$ for all $A'' $ such that $A' \prec A''$, then the coefficient of $[A']$ in the right-hand side of
(\ref{[A]-mono}) is $0$, contradicting to the assumption. We are done.
\end{proof}

Furthermore,

\begin{lem} \label{CAA'}
Suppose that
$\overline{[A]} = [A] + \sum_{A' \prec A} C_{A, A'} [A']$ for some $C_{A, A'} \in \mbb Z[v, v^{-1}]$.
If $C_{A, A'} \neq 0$, then $\e_i (A) = \e_i(A')$ for all $i\in \mbb Z$.
\end{lem}

\begin{proof}
By (\ref{[A]-mono}) and Lemma ~\ref{SAA'}, we have
$
\overline{[A]} = M_A + \sum_{A' \prec A, \e_i(A')= \e_i(A)} \overline{R_{A, A'}} M_{A'}.
$
By Lemma \ref{M-[A]}, we have
$
M_{A'} = \sum_{A'' \preceq A', \e_i(A'') =\e_i (A') } S_{A', A''} [A''].
$
The lemma follows by putting the previous two identities together.
\end{proof}

Finally, we are ready to prove Theorem ~\ref{PAA'}.
We set
$\phi =\{ A'  | P_{A, A'} \neq 0, \e_i (A') \neq \e_i (A)\}$. We only need to show that $\phi$ is empty. Pick an element $B$ in $\phi$ that  is maximal with respect to the partial order $\preceq$. Clearly, we have $B\neq A$.
We rewrite $\{A\}$ as follows.
\[
\{ A\} = P_{A, B} [B]  +
\left ( \sum_{B \prec A'}  + \sum_{A' \prec B}  + \sum_{ A' \not\preceq B, B\not \preceq A'} \right )  P_{A, A'} [A'].
\]
Apply the bar operation to the above equality, we have
\[
\{ A\} =\overline{\{A\}}  = \overline{P_{A, B}} \overline{[B]}
+ \left ( \sum_{B \prec A'}  +
\sum_{A' \prec B}  + \sum_{ A' \not\preceq B, B\not \preceq A'} \right ) \overline{ P_{A, A'}}  \overline{[A']}.
\]
By Lemma ~\ref{CAA'}, we know that the coefficient of $[B]$ in $\overline{[A']}$ for $B \prec A'$ is zero.
Notice  $[B]$ will not appear in  the rest of the terms, except $\overline{[B]}$.
Hence, by comparing the coefficients of $[B]$ in  the previous two equalities,  we  must have $P_{A, B} = \overline{P_{A, B}}$. But $P_{A, B} \in v^{-1} \mbb Z[v^{-1}]$ forces  $P_{A, B}=0$, a contradiction to the definition of  $\phi$. Hence $\phi$ is empty.  Theorem \ref{PAA'}  follows.

\subsection{Positivity of $\widehat{\mbb \Delta}$}

We set $\widehat{X}_d(\mbf a) =\{ \mbf L \in \widehat{X}_d | | \mbf L_i/\mbf L_{i-1} | = a_i,\quad \forall i\}$
and $\widehat P_{\mbf a} = \mrm{Stab}_{\widehat{\G}_d} (\mbf L) $ for a fixed chain $\mbf L\in \widehat X_d$.
We still have the same commutative diagram as in Lemma ~\ref{Delta-ik}.
\[
\xymatrixrowsep{.5in}
\xymatrixcolsep{.8in}
\xymatrix{
\mathcal A_{G_d}(\widehat X_d(\mbf b) \times  \widehat X_d(\mbf a))
\ar@{->}[r]^{i^*_{\mbf b, \mbf a}}
\ar@{->}[d]_{\widetilde \Delta_{\mbf b', \mbf a', \mbf b'', \mbf a''}}
&
\mathcal A_{\widehat P_{\mbf b}} (\widehat X_d(\mbf a))
\ar@{->}[d]^{ \pi_! \iota^*}
\\
%
%
\smxylabel{
\substack{
\mathcal A_{G_{d'}} (\widehat X_{d'}(\mbf b') \times  \widehat X_{d'} (\mbf a') ) \\
\otimes \\
\mathcal A_{G_{d''}} (\widehat X_{d''}(\mbf b'') \times \widehat X_{d''} (\mbf a'') )
}
}
\ar@{->}[r]^-{ i^*_{\mbf b', \mbf a'} \otimes i^*_{\mbf b'', \mbf a''}}
&
\smxylabel{
\mathcal A_{\widehat P_{\mbf b'} } ( \widehat X_{d'} (\mbf a') )
\otimes
\mathcal A_{\widehat P_{\mbf b''}} ( \widehat X_{d''} (\mbf a'')).
}
}
\]
So the positivity of $\widetilde \Delta_{\mbf b', \mbf a', \mbf b'', \mbf a''}$ is reduced to the positivity of $ \pi_! \iota^*$.

Fix $\mbf b', \mbf b''$ such that $\mbf b' + \mbf b'' =\mbf b$.
Let $d' = | \mbf b'|$ and $d'' = |\mbf b'' | $.
Let $\mbf V = \mbf T \oplus \mbf W$ and $\bfL_{\mbf b} = \bfL_{\mbf b'} \oplus \bfL_{\mbf b''}$.
Thus we have
$
\pi' (\bfL_{\mbf b}) = \bfL_{\mbf b'},
\pi'' (\bfL_{\mbf b}) = \bfL_{\mbf b''}.
$
Let $\bfL_i $ be the $i$-th lattice in $\bfL_{\mbf a}$. We consider the following subset in $\widehat X_d(\mbf a)$.
\[
Y^{\bfL_0, p}_{\mbf a} : =\{ \widetilde{\bfL} \in \widehat X_d (\mbf a) | \e^p \bfL_0 \subseteq \widetilde{\bfL}_0 \subseteq \e^{-p} \bfL_0\}, \quad \forall p \in \mbb Z_{\geq 0}.
\]
It is well-known that $Y^{\bfL_0, p}$ for various $p$ is a $G_{\bfL_{\mbf b}}$-invariant algebraic variety over $\overline{\mbb F_q}$ if we replace the ground field $\mbb F_q((\e))$ by $\overline{\mbb F_q}((\e))$, which we shall assume now and for the rest of this section.
Moreover, there exists a $p_0$ such that
\[
X^{\bfL_{\mbf b}}_A:= \{ \mbf L' | (\mbf L_{\mbf b}, \mbf L')\in \mathcal O_{A}\} \subseteq Y^{\bfL_0, p}, \quad p\geq p_0.
\]

Indeed, we have $a_{0, p} = 0$ and $a_{p, 0}$ for $p>>0$ due to the fact that
$\sum_{j \in \mbb Z} a_{0, j} , \sum_{ i \in \mbb Z} a_{i, 0} < \infty$.
The first condition implies that
$
\bfL_0 \subseteq \widetilde{\bfL}_{p}, \ \mbox{if} \  \widetilde{\bfL} \in X_A^{\bfL},
$
while
$
 \widetilde{\bfL}_0 \subseteq \bfL_p, \ \mbox{if} \  \widetilde{\bfL} \in X_A^{\bfL},
$
follows from the second.
Fix an $l$ such that $p < ln$. Then we have
\[
\e^{l} \bfL_0 \subseteq  \widetilde{\bfL}_0  \subseteq \e^{-l} \bfL_0.
\]
Set $p_0=l$, then  we have  $X^{\bfL_{\mbf b}}_A \subseteq Y^{\bfL_0, p}$ for $p\geq p_0$.

Now we fix a $1$-parameter subgroup of $P_{\bfL_{\mbf b}}$:
\[
\lambda : \mrm{GL}(1, \overline{\mbb F}_q) \to P_{\bfL_{\mbf b}}, t \mapsto
\begin{pmatrix}
1_{\mbf T} & 0 \\
0 & t .1_{\mbf W}
\end{pmatrix}.
\]
The fixed point set
$
(Y^{\bfL_0, p}_{\mbf a})^{\mrm{GL}(1, \overline{\mbb F_q})}
=\sqcup_{\mbf a', \mbf a''} Y^{\mbf L_0', p}_{\mbf a'} \times Y^{\mbf L_0'', p}_{\mbf a''},
$
and the attracting set associated to $Y^{\mbf L_0', p}_{\mbf a'} \times Y^{\mbf L_0'', p}_{\mbf a''}$ is
\[
Y^{\bfL_0, p}_{\mbf a', \mbf a''} =
\{
\widetilde{\bfL} \in Y^{\bfL_0,p}_{\mbf a} | \pi' (\widetilde{\bfL} ) \in Y^{\mbf L_0', p}_{\mbf a'},
\pi''(\widetilde{\bfL} ) \in Y^{\mbf L_0'', p}_{\mbf a''}
\}.
\]
Hence we have the following cartesian diagram.

\[
\begin{CD}
Y^{\bfL_0, p}_{\mbf a} @<\iota_1 <<  Y^{\bfL_0, p}_{\mbf a', \mbf a''} @>\pi_1 >>
Y^{\mbf L_0', p}_{\mbf a'} \times Y^{\mbf L_0'', p}_{\mbf a''}\\
@VVV @VVV @VVV \\
\widehat X_d(\mbf a)  @< \iota << \widehat X^+_{\mbf a, \mbf a', \mbf a''} @>\pi >> \widehat X_{d'} (\mbf a') \times \widehat X_{d''} (\mbf a''),
\end{CD}
\]
where the vertical maps are natural inclusion and the top horizontal maps are induced from the corresponding bottom ones.

Hence the positivity of $ \pi_! \iota^*$ is boiled down to that of $ \pi_{1!} \iota_1^*$, which follows from Braden's work ~\cite{B03}, since all objects involved are in the category of algebraic varieties over $\overline{\mbb F_q}$.

\begin{prop} \label{pos-tildeDelta-affine}
If $\widetilde \Delta_{\mbf b', \mbf a', \mbf b'', \mbf a''} (\{ A\}_d ) = \sum \widetilde  m^{B,C}_A \{B\}_d \otimes \{C\}_d$, then
$\widetilde  m^{B,C}_A \in \mbb Z_{\geq 0} [v, v^{-1}]$.
\end{prop}

The following is an affine version of Proposition ~\ref{Pos-S},
which is a consequence of Proposition ~\ref{pos-tildeDelta-affine} and
Theorem ~\ref{xi-positivity}

\begin{thm}
\label{affine-positivity}
If $\mbb \Delta_{\mbf b', \mbf a', \mbf b'', \mbf a''} (\{ A\}_d  )
= \sum \widehat m^{B,C}_A \{B\}_d \otimes \{C\}_d$, then
$\widehat  m^{B,C}_A \in \mbb Z_{\geq 0} [v, v^{-1}]$.
\end{thm}

Following ~\cite{L00}, the transfer map $\widehat \phi_{d, d-n}: \aSB_d \to \aSB_{d-n}$ is the
composition of
$$
\begin{CD}
\aSB_d @>\widetilde \Delta>>  \aSB_{d-n} \otimes \aSB_{n}
@>\xi^{-1} \otimes \chi >> \aSB_{d-n} \otimes \mbb A \cong  \aSB_{d-n},
\end{CD}
$$
where
$\xi $ is in Remark ~\ref{L00-xi} and 
$\chi$ is the signed representation of $\aSB_n$ defined in ~\cite[1.8]{L00}. 

Note that by ~\cite[1.12]{L00} and an argument similar to \cite[3.3]{LW14},  $\chi$ sends a canonical basis element to $1$ or $0$.
By Remark ~\ref{L00-xi} and Proposition ~\ref{pos-tildeDelta-affine}, we have

\begin{cor}
$\widehat \phi_{d, d-n} (\{ A\}_d)=\sum c_{A, A'} \{A'\}_{d-n}$ where $c_{A, A'} \in \mbb Z_{\geq 0} [v, v^{-1}]$.
\end{cor}

\subsection{Positivity in quantum affine $\mathfrak{sl}_n$}


Let $\tilde{\mathfrak S}_n$ be the set of all $\mbf a =(a_i)_{i\in \mbb Z}$
such that $a_i \in \mbb Z$ and $a_i = a_{i+n}$ for all $i \in \mbb Z$.
Let $\hmY = \{ \mbf a  \in \tilde{\mathfrak S}_n | \sum_{1\leq i\leq n} a_i = 0\}$.
We define an equivalence relation $\sim$ on $\tilde{\mathfrak S}_n$ by declaring $ \mbf a \sim \mbf b$
if there is a $z$ in $\mbb Z$ such that $a_i - b_i = z$ for all $i$.
Let $\hmX$ be the set $\tilde{\mathfrak S}_n/\sim$ of all equivalence classes in $\tilde{\mathfrak S}_n$ with respect to $\sim$.
Let $\overline{\mbf a}$ denote   the class of $\mbf a$.
Both $\hmX$ and $\hmY$ admit a natural abelian group structure with the  component-wise addition.
Moreover, we have a bilinear form
\[
\langle -, - \rangle :
\hmY \times \hmX \to \mbb Z, \quad
\langle \mbf b, \bar{\mbf a} \rangle = \sum_{1\leq i\leq n} b_i a_i.
\]

Now set $\widehat I = \mbb Z/ n\mbb Z$.
For each $i\in \hI$, we associate an element, still denoted by $i$, in $\hmY$ whose value is $1$ for each integer in the equivalence class $i$ and zero otherwise.
This  defines a map
$
\widehat I \to \hmY.
$
The same  map induces a map $\widehat I \to \hmX$ which sends $i\in \hI$ to the equivalence class of $i\in \hmY$ in $\tilde{\mathfrak S}_n$. By abuse of notations, we still use $i$ to denote its image in $\hmY$.
The data  ($\hmY, \hmX, \langle-,-\rangle, \widehat I \subset  \hmY, \widehat I \subset \hmX$) is thus a Cartan datum of affine type $\mbf A_{n-1}$, which is neither $\hmX$-regular nor $\hmY$-regular.

By definition, the quantum affine $\mathfrak{sl}_n$ attached to the above root datum,  denoted by $\mbb U(\widehat{\mathfrak{sl}}_n)$,
is an associative  algebra over $\mbb Q( v)$  generated by the generators:
$
\mbb E_i, \mbb F_i, \mbb K_{\mu}
$
forall $i \in \widehat I, \mu \in \hmY$,
and subject to the  relations
$\mbb K_1 \mbb K_2 \cdots \mbb K_n  =1$ and (\ref{q-Serre}), for all $ i, j \in \widehat I$.
Note that the first defining relation of $\U(\widehat{\mathfrak{sl}}_n)$ is due to the degeneracy of the Cartan datum.

Moreover, $\U(\widehat{\mathfrak{sl}}_n)$ admits a Hopf-algebra structure,  whose comultiplication is defined by the following rules.
\begin{align}
\begin{split}
\mbb \Delta ( \mbb E_i)  = \mbb E_i \otimes \mbb K_i + 1 \otimes \mbb E_i, \
\mbb \Delta (\mbb F_i)  = \mbb F_i \otimes 1+ \mbb K^{-1}_{i} \otimes \mbb F_i, \
\mbb \Delta (\mbb K_i)  = \mbb K_i\otimes \mbb K_i, \quad \forall  i \in \widehat I.
\end{split}
\end{align}
Let $\dot{\U}(\widehat{\mathfrak{sl}}_n)$ be Lusztig's idempotented algebra associated to $\U (\widehat{\mathfrak{sl}}_n)$.
It is defined similar to that of quantum $\mathfrak{sl}_n$ in Section ~\ref{Type-A-positivity}.
Similar to the finite case, $\mbb \Delta$ then induces a linear map
\begin{align} \label{affine-Delta-dot}
\mbb \Delta_{\overline{\mu'}, \overline{\lambda'}, \overline{\mu''}, \overline{\lambda''}}: \ _{\overline{\mu}} \mbb U_{\overline{\lambda}} (\widehat{\mathfrak{sl}}_n)\to \ _{\overline{\mu'}} \mbb U_{\overline{\lambda'}}(\widehat{\mathfrak{sl}}_n) \otimes \ _{\overline{\mu''}} \mbb U_{\overline{\lambda''}}(\widehat{\mathfrak{sl}}_n),
\end{align}
where $_{\overline{\mu}} \mbb U_{\overline{\lambda}} (\widehat{\mathfrak{sl}})$ is defined similar to
$_{\overline{\mu}} \mbb U_{\overline{\lambda}}$ in finite case and
$\overline{\mu} = \overline{\mu'} + \overline{\mu''}$,
$\overline{\lambda} = \overline{\lambda'} + \overline{\lambda''}$ in $\widehat X$.

By the same definition as $\phi_d$ in (\ref{Psi}), we still have an algebra homomorphism
\[
\widehat{\phi}_{d} : \U(\widehat{\mathfrak{sl}}_n) \to \aSB_d.
\]
But this time $\widehat{\phi}_d$ is not surjective anymore.
Then the rest of the result in finite case can be transported to affine case. In particular, we have

\begin{thm}  \label{conj-25.4.2-affine}
Let $b$ be a canonical basis element of $\dot{\U} (\widehat{\mathfrak{sl}}_n)$.
If
$
\mbb \Delta_{\overline{\mu'}, \overline{\lambda'}, \overline{\mu''}, \overline{\lambda''}} (b)
= \sum_{b', b''} \hat m_{b}^{b', b''} b'\otimes b''
$,
then  $\hat m_b^{b', b''} \in \mbb Z_{\geq 0} [v, v^{-1}]$.
\end{thm}

By ~\cite[25.2.2]{L93}, Theorem ~\ref{conj-25.4.2-affine} remains valid over other root datum of affine type $\mbf A_{n-1}$.

\section{Coproduct for  the $\jmath$Schur  algebras}

In this section, we define the comultiplication on the $\jmath$Schur algebra level and show that it gives rise to the transfer map used in ~\cite{LW14}.
We shall also show that the comultiplication degenerates to an imbedding of a $\jmath$Schur algebra to an ordinary Schur algebra and
establish a direct connection of the type $A$ geometric duality of Grojnowski-Lusztig ~\cite{GL92} and the type $B/C$ geometric duality in ~\cite{BKLW13}.

\subsection{The $\jmath$Schur algebra $\mbf S^{\B}_d$}
\label{vSB}

In this section, we assume that $n$ and $D$ are odd, i.e.,
$$n=2r+1 \quad \mbox{and} \quad  D=2d+1.$$
We fix a non-degenerate symmetric bilinear form $Q^{\B}: \mbb F_q^D\times \mbb F_q^D\rightarrow \mbb F_q$.
Let $W^{\perp}$ stand for   the orthogonal complement of the vector subspace  $W$ in $\mbb F_q^D$ with respect to the form $Q^{\B}$.
By convention, $W$ is called isotropic if $W\subseteq W^{\perp}$.
Recall the set $X_d$ from Section ~\ref{vSA}.
Consider the subset $X_d^{\B}$ of $X_d$ defined by
$$X^{\B}_d =\{V\in  X_d | V_i=V_j^{\perp},\ \mbox{if} \  i+j=n\}.$$

Let $\G^{\B}_d$ be the orthogonal group attached to $Q^{\B} $, i.e.,
$$
\G^{\B}_d= \{g\in \G_d | Q^{\B} (gu, g u')=Q^{\B} (u,u'),\ \forall u,u'\in \mbb F_q^D\}.
$$
The group $\G^{\B}_d$ acts from the left  on $X^{\B}_d$.
It induces a diagonal action  on $X^{\B}_d \times X^{\B}_d$.
By the general construction in Section ~\ref{conv}, we have a unital associative algebra
\begin{align}
\label{SB}
\mbf S^{\B}_d \equiv \mathcal A_{\G^{\B}_d} ( X^{\B}_d \times X^{\B}_d).
\end{align}
This is the algebra first appeared in ~\cite{BKLW13}. See also ~\cite{G97} and \cite{Du-Scott}.

\subsection{Coproduct on $\mbf S^{\B}_d$}\label{cosb}

We set $\D =\mathbb F_q^D$.
We need the following auxiliary lemma.

\begin{lem}
\label{auxi}
Suppose that $\D''$ is an isotropic subspace of $\D$ and $L=(L_i| 0 \leq i\leq n) \in X^{\B}_d$. Then we can find a pair $(T, W)$  of subspaces in $\D$ such that
\begin{enumerate}
\item [(a)]  $\D= \D''\oplus T\oplus W$, $(\D'')^{\perp} = \D'' \oplus T$,
\item [(b)] $W$ is isotropic and $T\perp W$,
\item [($\mbox{c}$)] There exists bases $\{z_1,\cdots,z_s\}$ and $\{w_1,\cdots, w_s\}$ of $\D''$ and $W$, respectively, such that $Q^{\B} (z_i,w_j)=\delta_{ij}$ for any $i, j\in [1, s]$,
\item [(d)] $L_i = (L_i\cap \D'')\oplus (L_i\cap T)\oplus (L_i\cap W)$, for any $1\leq i\leq n-1$.
\end{enumerate}
\end{lem}

\begin{proof}
Assume that $n=3$. We can use an induction process to  find a subspace $T' \subset (\D'')^{\perp}$ such that $(\D'')^{\perp} = \D''\oplus T'$ and
\[
L_i \cap (\D'')^{\perp} = (L_i\cap \D'')\oplus (L_i\cap T'),\quad \forall 1\leq i\leq n-1.
\]
Moreover, the restriction of the bilinear form $Q^{\B}$ to $T'$ is automatically non-degenerate.
next, we can find a subspace $W_1\subseteq L_1$ such that
$
L_1= (L_1\cap (\D'')^{\perp} )\oplus W_1.
$
Similarly, we can find subspaces $U_2$ and $T_2$ such that
\[
L_2\cap \D'' = (L_1\cap \D'' ) \oplus U_2, \quad
L_2\cap T'= (L_1\cap T') \oplus T_2.
\]
Via the natural projection $L_2 \to L_2/L_1$, we can regard $U_2\oplus T_2$ as subspaces in $L_2/L_1$.
Now $L_2/L_1$ inherits a non-degenerate bilinear form from that of $\D$.
Moreover,  $U_2\oplus T_2$ is the orthogonal complement of $U_2$ with respect to the form on $L_2/L_1$. By a well-known fact, say ~\cite[Theorem 6.11]{J}, we can find an isotropic subspace $W'_2$ such that
$L_2/L_1=U_2\oplus T_2\oplus W'_2$,  $T_2\perp W'_2$ and
$\dim U_2= \dim W_2'$. Furthermore, the restriction of the form to $U_2+W_2'$ is non-degenerate.
Now take a subspace $W_2$ in $L_2$ such that it gets sent to $W_2'$ via the projection map. Then by comparing the dimensions, we have
\[
L_2= (L_2\cap (\D'')^{\perp}) \oplus (W_1\oplus W_2).
\]
It is clear that $W_1\oplus W_2$ is an isotropic subspace in $L_2$ and $(W_1\oplus W_2) \perp (L_2\cap T')$.

Note that $T'$ is not necessarily perpendicular to $W_1\oplus W_2$. We consider the subspace $\D'' \oplus T'\oplus W_1\oplus W_2$.
We can find a subspace $U_1$ in $V''$ such that $U_1\cap (L_2\cap \D'')=\{0\}$ and the restriction of the bilinear form to $U_1\oplus W_1$ is non-degenerate. The latter implies that we can find bases $\{u_1,\cdots,u_s\}$ and $\{w_1,\cdots,w_s\}$ in $U_1$ and $W_1$, respectively, such that $(v_i,w_j)=\delta_{ij}$.
Recall that we have bases $\{u_{r+i}|1\leq i\leq s_1\} $ and $\{w_{r+i}|1\leq i\leq s_1\}$ for $U_2$ and $W_2$ such that $(u_{r+i},w_{r+j})=\delta_{ij}$.
Fix a basis $\{t'_i\}$ for $T'$ such that $\{t'_i\}\cap (L_2\cap T')$ is a basis of $L_2\cap T'$. Let $T$ be the subspace spanned by the elements
$t_i = t_i' - \sum_{1\leq j\leq s+s_1} (t_i', w_j) u_j$.
We thus have $T\perp (W_1\oplus W_2)$ and $T$ satisfies all properties $T'$ has with respect to the flag $L$.

By ~\cite[Theorem 6.11]{J}, we can extend $W_1\oplus W_2$ to a subspace $W$ satisfying the required properties, by extending  the subspace
$(\D'')^{\perp}\oplus W_1\oplus W_2$ to the whole space $\D$. So the pair $(T, W)$ satisfies the desired properties. The lemma follows for $n=3$.
For general $n$, it can be shown by a similar argument inductively.
\end{proof}

Suppose that $\D''$ is an isotropic subspace of $\D$ of dimension $d''$.
We set $\D'=(\D'')^{\perp}/\D''$, and denote by $D'$ its dimension $D - 2 d''=2d' +1$.
Thus $\D'$ admits a non-degenerate bilinear form inherited from that of $\D$.
Given any subspace $C\subseteq D$, it induces a subspace $\pi^{\natural}(C) \in \D'$ defined by
\[
\pi^{\natural} (C) =  \frac{C \cap (\D'')^{\perp} + \D''}{ \D''}.
\]
Recall the  operation  $\pi''$ from Section ~\ref{CoSA}.
For any $L\in X^{\B}_d$, we have that
$\pi^{\natural}(L)\in  X^{\B}_{d'}$ and $\pi''(L)\in X_{d''}$.
For any pair $(L'', L')\in  X_{d''}\times X^{\B}_{d'}$, we set
\begin{align}
\label{Z}
Z^{\B}_{L',L''}=\{
L\in X^{\B}_d| \pi^{\natural}(L)=L', \pi''(L)=L''
\}.
\end{align}
We also set
$\tilde Z$ to be the set of all pairs $(T, W)$ subject to the conditions (1), (2) and (3) in Lemma ~\ref{auxi}.
To a pair $(T,W)\in \tilde Z$, we have an isomorphism $\pi: T\to \D'$.
Define a map  $\tilde Z\to Z^{\B}_{L', L''}$ by sending $(T, W)$ to $L^{T,W}$, where
\[
L_i^{T,W} = L''_i\oplus \pi^{-1} (L_i') \oplus (L''_{n-i})^{\#}, \
(L''_{n-i})^{\#} =\{ w\in W| (w,L''_{n-i})=0\},
\quad\forall 1\leq i\leq n,
\]
By Lemma ~\ref{auxi}, we see that the map $\tilde Z\to Z^{\B}_{L', L''}$ is surjective.
Let
\[
\mathcal U=\{ g\in \G^{\B}_d | g(v)=v,\forall v\in \D'', g(v_1)-v_1\in \D'', \forall v_1\in (\D'')^{\perp}\}.
\]
Clearly $\mathcal U$ acts on $\tilde Z$ and $Z^{\B}_{L',L''}$. Moreover,
it can be checked that $\mathcal U$ acts transitively on $\tilde Z$ and is compatible with the surjective map $\tilde Z\to Z^{\B}_{L',L''}$.
Therefore, we have  the following lemma, analogous to ~\cite[Lemma 1.4]{L00}.

\begin{lem}
\label{independent}
The group $\mathcal U$ acts transitively on the set $Z^{\B}_{L', L''}$.
\end{lem}

Recall the algebra $\mbf S_d $ from Section ~\ref{vSA} (\ref{SA}).
We are now ready to define the comulitiplication  $\Delta^{\B}$. This is a map
\begin{align} \label{Delta-B}
\widetilde \Delta^{\B}: \mbf S^{\B}_d \to \mbf S^{\B}_{d'} \otimes \mbf S_{d''}, \quad \forall  d' + d'' =d,
\end{align}
defined by
\[
\widetilde \Delta^{\B} (f) ( L', \check L' , L'',\check L'') = \sum_{\check L\in Z^{\B}_{\check L',\check L''} } f(L, \check L), \quad \forall L',\check L'\in  X^{\B}_{d'}, L'', \check L''\in  X_{d''},
\]
where $L$ is a fixed element in $Z^{\B}_{L', L''}$ (See (\ref{Z})  for notations).
By Lemma ~\ref{independent}, we see that the definition of $\widetilde \Delta^{\B}$ is independent of the choice of $L$.
Moreover, by an argument exactly the same way as that of Proposition 1.5 in ~\cite{L00},
we have

\begin{prop}
The map $\widetilde \Delta^{\B}: \mbf S^{\B}_d \to \mbf S^{\B}_{d'} \otimes \mbf S_{d''}$ is an algebra homomorphism.
\end{prop}

For any $i\in [1, r]$, $a\in [1, r+1]$,  we define the following functions in $\mbf S^{\B}_d$
\begin{align*}
\E_i (L, L') &=
\begin{cases}
\bv^{-|L'_{i+1} / L'_i|}, & \mbox{if} \; L_i \overset{1}{\subset} L_i', L_j = L_j', \forall  j \in [1, r] \backslash \{i\}; \\
0, & \mbox{otherwise}.
\end{cases} \\
\F_i (L,  L') &=
\begin{cases}
\bv^{-|L'_i / L'_{i-1}|}, & \mbox{if} \; L_i \overset{1}{\supset} L_i', L_j = L_j',\forall j\in [1, r]\backslash \{i\}; \\
0, &\mbox{otherwise}.
\end{cases} \\
\mathbf h^{\pm 1}_{a} (L, L') & = \bv^{\pm |L_{a}/ L_{a-1}|}\delta_{L, L'},  \quad
\K^{\pm 1}_{i}    = \mbf h^{\pm 1}_{i+1} \mbf h^{\mp 1}_{ i},
\end{align*}
for any $L, L' \in X^{\B}_d$.
We write $\E_i', \F_i'$ and $ \mathbf h'_{\pm i}$ for the elements in $\mbf S^{\B}_{d'}$
analogous to $\E_i$, $\F_i$ and $\mathbf h_{\pm i}$  in $\mbf S^{\B}_d$, respectively.
Similarly, we use the notations $\mathbf E''_i$, $\mathbf F''_i$, and $\mbf K''_i$ for $1\leq i\leq n-1$, and $\mathbf H''_{\pm i}$, for $1\leq i\leq n$,
for elements in $\mbf S_{d''}$ defined  in Section \ref{CoSA}.
We now study the effect of the application of $\widetilde \Delta^{\B}$ to the generators.

\begin{prop}
\label{D(E)}
For any $i\in [1, r]$, we have
\begin{align*}
\begin{split}
\widetilde \Delta^{\B} ( \E_i)
&= \E_i' \otimes \mbf H_{i+1}'' \mbf H''^{-1}_{n-i} + \mbf h'^{-1}_{i+1} \otimes \mbf E_i''  \mbf H''^{-1}_{n-i} +  \mathbf h'_{i+1} \otimes \mbf F''_{n-i} \mbf H''_{i+1}.  \\
\widetilde \Delta^{\B} (\F_i)
 & = \F'_i \otimes \mbf H''^{-1}_{i} \mbf H''_{n+1-i} + \mbf h'_i\otimes \mbf F''_i \mbf H''_{n+1-i} + \mbf h'^{-1}_{i} \otimes \mbf E''_{n-i} \mbf H''^{-1}_{i}. \\
\widetilde \Delta^{\B} (\mbf k_i) & = \mbf k'_i \otimes \mbf K''_i \mbf K''^{-1}_{n-i}.
\end{split}
\end{align*}
\end{prop}

\begin{proof}
By definition, we have
\[
\widetilde \Delta^{\B} (\E_i) ( L', \check L', L'', \check L'' ) =\bv^{-|\check L_{i+1}/\check L_i|} \# S,
\]
where
$
S=
\{ \check L\in Z^{\B}_{\check L',\check L''} | L_i\subset \check L_i,  |\check L_i/ L_i| =1, L_j=\check L_j, \forall 1\leq j\neq i\leq r\}.
$
The set $S$ is nonempty only when the quadruple $(L', \check L', L'',\check L'')$ is  in one of  the following three cases.
\begin{itemize}
\item [(i)]  $L'_i\subset \check L'_i$, $|\check L'_i/L'_i|=1$,   $L'_j=\check L'_j$, for all $1\leq j\neq i\leq r$, $L''_j=\check L''_j$,  for all $j$.
\item [(ii)] $L'_j =\check L'_j$, for all $j$, $L''_i\subset \check L''_i$, $|\check L''_i/L''_i|=1$, $L''_j= \check L''_j$ for all $j\neq i$.
\item [(iii)] $L'_j= \check L'_j$, for all $j$, $L''_{n-i} \supset \check L''_{n-i}$, $|L''_{n-i}/\check L''_{n-i}|= 1$, $L''_j= \check L''_j$ for all $j\neq n-i$.
\end{itemize}

We now  compute the number $\#S$ in case (i).
This amounts to count  all possible lines $<u>$, spanned by the vector $u$,  such that $L_i+ <u>$ is in $S$.
Since we want $L_i+<u>\subseteq L_{i+1}$, we should find  $u $ in $ L_{i+1}$.
Since we need $\pi^{\natural}(L_i+<u>)= \check L_i'$, we need to find those  $u$ such that  $\pi(u)= u'$, where $u'$ is a fixed element in $\D'$ such that $\check L'_i = L'_i +<u'>$.
Fix a pair $(T,W)$ in $\D$ such that it satisfies all conditions in Lemma ~\ref{auxi} with respect to the flag $L$.
In particular, $L_{i+1} = L''_{i+1} \oplus (L_{i+1}\cap T) \oplus (L_{i+1}\cap W)$.
Since $T$ gets identified with $\D'$ via the canonical projection, there is a unique $t$ in $T$ sending to $u'$. So we need to look for those $u$ such that
at component $L_{i+1}\cap T$, $u=t$, and at component $L_{i+1}\cap W$, $u=0$.
Thus $u$ is of the form $t+w$ where $w\in L''_{i+1}$.
Since adding $w$ by any vector in $L''_i$ does not change  the resulting space $L_i+<u>$, we see that
the freedom of choice for $w$ is
$L''_{i+1}$ mod $L''_i$, i.e., $L''_{i+1}/L''_i$. So we see that
the value of $\widetilde \Delta^{\B} (\E_i) ( L', \check L', L'', \check L'' )$ is equal to
\[
\bv^{-|\check L_{i+1}/\check L_i|}q^{|L''_{i+1}/L''_i|}
= \bv^{-|\check L'_{i+1}/\check L'_i|} \bv^{-|\check L''_{n-i} /\check L''_{n-i-1}| +|\check L''_{i+1}/\check L''_i|}
=(\E_i' \otimes \mbf H''^{-1}_{n-i} \mbf H''_{i+1} )(L', \check L', L'',\check L''),
\]
where we use
$
|\check L_{i+1}/\check L_i|
=
|\check L''_{i+1}/\check L''_i|+
|\check L'_{i+1}/\check L'_i|+
|\check L''_{n-i} /\check L''_{n-i-1}|.
$

For case (ii), $S$ consists of only one element, i.e., the $\check L$ such that $\check L_j=L_j$ for $1\leq j\neq i\leq r$,
 and $\check L_i = L_i + \check L''_i$. (Since $\check L''_i\subseteq L''_{n-i}$, $\check L_i$ is isotropic.)
So the  value of $\widetilde \Delta^{\B} (\E_i) ( L', \check L', L'', \check L'' )$ in case (ii)  is equal to
\[
\bv^{-|\check L_{i+1}/\check L_i|}
=
\bv^{-|\check L'_{i+1}/\check L'_i|} v^{- |\check L''_{i+1}/\check L''_i|} \bv^{-|\check L''_{n-i} /\check L''_{n-i-1}|}
=(\mathbf h'^{-1}_{i+1} \otimes \mbf E''_i \mbf H''^{-1}_{n-i} ) (L', \check L', L'',\check L'').
\]

For case (iii), we need to consider two situations, i.e., $i=r$ or $i\neq r$.
For $i=r$, the set $S$ gets identified with the set
$
S_r=\{l\in L_{r+1}/L_r: \check U_{r+1}\subset l^{\perp}, U_{r+1}\not \subset l^{\perp}\},
$
via $\check L\mapsto \check L_r/L_r$, where
$\check U_{r+1}= ( \check L''_{r+1}+L_r) /L_r$ and $U_{r+1} = (L''_{r+1}+L_r)/L_r$.
Set
\[
\tilde S_r=\{ W\subset L_{r+1}/L_r| W\  \mbox{isotropic},\ \check U_{r+1}\subset W, \dim W/\check U_{r+1} =1, W+U_{r+1} \ \mbox{not isotropic} \} .
\]
We define a map $S_r \to \tilde S_r$ by $l\mapsto l+ \check U_{r+1}$. It is clear that this is a surjective map
and its fiber is isomorphic to $\check U_{r+1}$. The set $\tilde S_r$ can be broken into the difference of the two sets
\begin{align*}
\tilde S_r &=\{ W|  W\  \mbox{isotropic}\ \check U_{r+1}\subset W, \dim W/\check U_{r+1} =1\}\\
& \quad - \{ W|W\  \mbox{isotropic}\ \check U_{r+1}\subset W, \dim W/\check U_{r+1} =1, W+U_{r+1} \ \mbox{isotropic}\}.
\end{align*}
For the first set, its order is equal to
$
\frac{q^{|L_{r+1}/L_r|-2|\check L''_{r+1}/\check L''_r| -1} -1}{q-1},
$
because $\check U_{r+1} \simeq \check L''_{r+1}/\check L''_r$.  For the second set, it is the union of $\{ W=U_{r+1}\}$ and
the subset $\{\dim W+U_{r+1}/\check U_{r+1}=2\}$. The latter has a surjection onto the set of isotropic lines in $U^{\perp}_{r+1}/U_{r+1}$ with fiber isomorphic to $\mbb F_q$, via $W\mapsto W+U_{r+1}/U_{r+1}$.
Thus the order of  the second set is equal to
$
1+ q \frac{ q^{|L_{r+1}/L_r | -2 |L''_{r+1}/L''_r| -1}-1}{q-1}.
$
So we have
\begin{align*}
\# S
&=\# \check U_{r+1}  (
\frac{q^{|L_{r+1}/L_r|-2|\check L''_{r+1}/\check L''_r| -1} -1}{q-1}
-1 -q \frac{ q^{|L_{r+1}/L_r | -2 |L''_{r+1}/L''_r| -1}-1}{q-1})\\
&=
q^{|\check L''_{r+1}/\check L''_r| + |\check L'_{r+1}/\check L'_r|}.
\end{align*}
So we see that  the  value of $\widetilde \Delta^{\B} (\E_i) ( L', \check L', L'', \check L'' )$, for $i=r$ in case (iii),  is equal to
\[
\bv^{-|\check L_{r+1}/\check L_r|} q^{|\check L''_{r+1}/\check L''_r| + |\check L'_{r+1}/\check L'_r|}
=\bv^{|\check L'_{r+1}/\check L'_r|} =
(\mbf h'_{i+1} \otimes \mbf F''_{n-i} \mbf H''_{i+1}) ( L', \check L', L'', \check L'' ).
\]

For $i\neq r$, the set
$S$ gets identified with the set $S'$ of isotropic lines  $ l$ in $L_{n-i}/L_i$ such that
\[
l \subseteq L_{i+1}/L_i, \check U\subset l^{\perp}, U\not \subset l^{\perp},
\]
where
$
\check U = \check L''_{n-i} + L_i/L_i$
and
$
U=  L''_{n-i}+L_i/L_i .
$
Notice $S'$ is the difference of the two sets:
\[ S'=
\{
l |l  \subseteq L_{i+1}/L_i, \check U\subset l^{\perp}
 \}
 -
 \{
l |l  \subseteq L_{i+1}/L_i,  U\subset l^{\perp}
 \}.
\]
We can use a similar arguments as in (i) to compute the two sets and we  get
\[
\# S=
\frac{ q^{ |L''_{i+1}/L''_i| + |L'_{i+1}/L'_i| +1} -1}{q-1} +
\frac{ q^{ |L''_{i+1}/L''_i| + |L'_{i+1}/L'_i| } -1}{q-1}
=q^{ |L''_{i+1}/L''_i| + |L'_{i+1}/L'_i| }.
\]
So we see that  the  value of $\widetilde \Delta^{\B} (\E_i) ( L', \check L', L'', \check L'' )$, for $i\neq r$ in case (iii),   is equal to
\begin{align*}
\bv^{-|\check L_{i+1}/\check L_i|} q^{ |L''_{i+1}/L''_i| + |L'_{i+1}/L'_i| }
&= \bv^{|\check L'_{i+1}/\check L'_i|} \bv^{- |\check L''_{n-i}/\check L''_{n-i-1}| + |\check L''_{i+1}/\check L''_i|}\\
&=(\mbf h'_{i+1} \otimes \mbf F''_{n-i} \mbf H''_{i+1}) ( L', \check L', L'', \check L'' ).
\end{align*}
We have the first identity.

Next, we determine  $\widetilde \Delta^{\B} (\F_i)$. By definition, we have
\begin{equation}
\label{DF}
\widetilde \Delta^{\B} (\F_i) ( L',\check L', L'', \check L'') =
\bv^{-|\check L_i/\check L_{i-1}|} \# R
\end{equation}
where
$
R=\{ \check L\in Z^{\B}_{\check L',\check L''} | L_i\supset \check L_i, |L_i/\check L_i|=1,  L_j=\check L_j, \forall 1\leq j\leq r, j\neq r\}.
$
Now the set $R$ is empty unless the quadruple $(L', \check L',L'', \check L'')$ is in one of the following cases.
\begin{itemize}
\item [(iv)] $L'_i\supset \check L'_i$, $|L'_i/\check L'_i|=1$, $L'_j=\check L'_j$, $\forall 1\leq j\leq r$, $j\neq i$, $L''_j=\check L''_j$ for all $j$.

\item [(v)] $L'_j=\check L'_j$ for all $j$, $L''_i\supset \check L''_i$, $|L''_i/\check L''_i|=1$, $L''_j=\check L''_j$, $\forall 1\leq j\neq i\leq n$.

\item [(vi)] $L'_j = \check L'_j$ for all $j$, $L''_{n-i}\subset \check L''_{n-i}$, $|\check L''_{n-i}/L''_{n-i}|=1$, $L''_j=\check L''_j$, $\forall 1\leq j\neq n-i\leq n$.
\end{itemize}

In these cases,  $\check L$ differs from   $L$ only at $i$ and $n-i$. Thus we can  identify $\check L$ with $\check L_i$.

In case (iv), to count the number of elements in $R$, we break it into two steps. We first determine all possible choices of $\check L_i\cap (\D'')^{\perp}$ for $\check L_i\in R$.
Since $\check L''_i=L''_i$, we have only one choice, i.e., $\tilde L_i=L''_i+T$, where $T$ is any subspace (of dimension $|\check L'_i|$) in $L_i\cap (\D'')^{\perp}$ maps onto $\check L'_i$ via the canonical projection.
 We  next want to determine the number of choices of $W\subseteq  L_i$ such $W+\tilde L_i\in R$.
We first observe that if $\check L\in R$, then the projection, say $\check  L'''_i$,  of
$\check L_i$ to $\D/(\D'')^{\perp}$ is the same as that of $L_i$.
Since $|L_i \cap (\D'')^{\perp}/ \check L_i\cap (\D'')^{\perp}|=1$ and $L_{i-1}\subseteq \check L_i$, we see that all possible
choices for  $W\in L_i$ such that $W+\tilde L_i\in R$ is bijective to the space
\[
\check L'''_i/L'''_{i-1} \simeq  (\check L''_{n-i})^{\#}/(\check L''_{n-i+1})^{\#}\simeq
\check L''_{n-i+1}/\check L''_{n-i},
\]
where $L'''_{i-1}$ is the projection of $L_{i-1}$ to $\D/(\D'')^{\perp}$.
Thus, in case (iv), the left hand side of (\ref{DF}) is equal to
$
\bv^{-|\check L_i/\check L_{i-1}|} q^{|\check L''_{n-i+1}/\check L''_{n-i}|}
=\F'_i \otimes \mbf H''^{-1}_{i} \mbf H''_{n+1-i}  (L', \check L',L'',\check L'').
$

In case (v), to build a subspace $\check L_i$ in $L_i$ such that it is in $R$, there are
$\check L'_i/\check L'_{i-1}$ choices to build the component $\tilde L_i=\check L_i\cap (\D'')^{\perp}$.
This is done  by using a similar argument as in the first step of case (iv) since $L''_i\supset \check L''_i$ and $|L''_i/\check L''_i|=1$.
By a similar argument as the step two in case (iv), we see that the number of choices for a subspace $W$ in $L_i$ such that $W+\tilde L_i\in R$ is
again
$\check L''_{n-i+1}/\check L''_{n-i}$ for a fixed subspace from the first step. Thus
the value of (\ref{DF}) in case (v) is equal to
\[
\bv^{-|\check L_i/\check L_{i-1}|}
q^{|\check L'_i/\check L'_{i-1}|}
q^{|\check L''_{n-i+1}/\check L''_{n-i}|}
=\mbf h'_i\otimes \mbf F''_i \mbf H''_{n+1-i} (L', \check L',L'',\check L'').
\]

In case (vi), there is only one element in $R$.
First of all, $\check L_i\cap (\D'')^{\perp} = L_i \cap (\D'')^{\perp}$ for $\check L_i\in R$.
Second of all, by fixing a decomposition of $\D=\D''\oplus T\oplus W$ as in Lemma ~\ref{auxi}, we see that if $\check L_i\in R$, then the projection of
$\check L_i$ to $\D/(\D'')^{\perp}$ is $(\check L''_{n-i})^{\#}$. Thus by an argument similar to the second step of case (iv), we see that there is  only one $\check L_i$ in $R$. This implies that the value of (\ref{DF}) in case (vi) is equal to
$
\bv^{-|\check L_i/\check L_{i-1}|}
=
 \mbf h'^{-1}_{i} \otimes \mbf E''_{n-i} \mbf H''^{-1}_{i}
 (L', \check L',L'',\check L'').
$
We see that the second identity  follows from the above computations.

Finally the last identity  follows from the definitions and
$\widetilde \Delta^{\B} (\mbf h_i) = \mbf h'_i \otimes \mbf H''_i \mbf H''_{n+1-i}$.
\end{proof}

\begin{cor}\label{coassoB}
We have $(1 \otimes \widetilde \Delta)  \widetilde \Delta^{\B} = (\widetilde \Delta^{\B} \otimes 1) \widetilde  \Delta^{\B}$.
\end{cor}

The corollary follows by checking if the relation holds for generators, which is immeidate.

\subsection{Renormalization}
\label{secrenormal}

Given a pair $(\mbf b, \mbf  a)$ in $\Lambda_{d,  n}$ for $n=2r+1$ (see (\ref{Lambda-A})), we set
\begin{align}
\label{s(a,b)}
\begin{split}
u(\mbf b, \mbf a) & = \frac{1}{2}  \left ( \sum_{i + j  \geq n+1} b_i b_j - a_i a_j +  \sum_{i\geq r+1} a_i - b_i \right )\\
& = \sum_{i>j, i+j \geq n+1} b_i b_j - a_i a_j + \frac{1}{2} (\sum_{i\geq r+1} b_i^2 - a_i^2 + a_i -b_i) \in \mbb Z.
\end{split}
\end{align}
Consider the subset $\Lambda^{\B}_{d, n}$ of $\Lambda_{2d+1, n}$ defined by
\[
\Lambda^{\B}_{d, n} = \{\mbf a \in \Lambda_{2d+1, n} | a_i = a_{n+1-i} , \ \forall 1 \leq i \in n\}.
\]
Similar to $X_d$, the set $X^{\B}_d$ admits a decomposition:
\[
X^{\B}_d = \sqcup_{\mbf a \in \Lambda^{\B}_{d, n} } X^{\B}_d(\mbf a), \quad
X^{\B}_d(\mbf a) =\{ V\in X^{\B}_d | |V_i/V_{i-1}|  = a_i, \forall 1\leq i\leq n\}.
\]
Let $\mbf S^{\B}_d(\mbf b, \mbf a)$ be the subspace of $\mbf S^{\B}_d$ spanned by all functions supported on
$X^{\B}_d(\mbf b) \times X^{\B}_d(\mbf a)$.
We have
\[
\mbf S^{\B}_d = \oplus_{\mbf b, \mbf a \in \Lambda^{\B}_{d, n}} \mbf S^{\B}_d(\mbf b, \mbf a),
\]
such that
$
\mbf S^{\B}_d (\mbf c, \mbf b) * \mbf S^{\B}_d( \mbf b, \mbf a) \subseteq \mbf S^{\B}_d(\mbf c, \mbf a).
$
So  $\widetilde \Delta^{\B}$ in (\ref{Delta-B}) can be decomposed as
\[
\widetilde \Delta^{\B} =\oplus
\widetilde \Delta_{ \mbf b', \mbf a', \mbf b'', \mbf a''}^{\B}
\]
where
$\widetilde \Delta_{ \mbf b', \mbf a', \mbf b'', \mbf a''}^{\B}$ is the component from
$\mbf S^{\B}_d(\mbf b, \mbf a)$ to $\mbf S^{\B}_d(\mbf b', \mbf a') \otimes \mbf S(\mbf b'', \mbf a'')$ such that
\[
b_i = b'_i + b''_i + b''_{n+1-i}, \quad a_i = a'_i + a''_i +  a''_{n+1-i}, \quad \forall 1\leq i\leq n.
\]
We renormalize $\widetilde \Delta^{\B}_{\bv}$ in (\ref{Delta-B}) as follows.
\begin{align}
\label{Delta-J}
\Delta^{\B}_{\bv} = \oplus_{\mbf b, \mbf a, \mbf b', \mbf a', \mbf b'', \mbf a''}
\Delta_{ \mbf b', \mbf a', \mbf b'', \mbf a''}^{\B},
\end{align}
where
$
\Delta_{\mbf b', \mbf a', \mbf b'', \mbf a''}^{\B}
=
\bv^{\sum_{1\leq i \leq j \leq n}  b_i' b_j'' - a_i' a_j'' }
\bv^{ u (\mbf b'', \mbf  a'')}
\widetilde \Delta_{\mbf b', \mbf a', \mbf b'', \mbf a''}^{\B}.
$
Note that $\Delta^{\B}_{\bv}$ is again an algebra homomorphism, due to the fact that
$u (\mbf c, \mbf a) = u (\mbf c, \mbf b) + u (\mbf b, \mbf a)$.
By a straightforward computation based on Proposition ~\ref{D(E)}, we have

\begin{prop}
\label{D(E)-renorm}
For any $i\in [1, r]$,
\begin{align*}
\begin{split}
\Delta^{\B}_{\bv} ( \E_i) & = \E_i' \otimes  \mbf K''_{i} + 1 \otimes \mbf E_i''   +  \mbf k_i'  \otimes \mbf F''_{n-i} \mbf K''_i. \\
\Delta^{\B}_{\bv} (\F_i) & = \F'_i \otimes \mbf K''_{n-i}  +  \K'^{-1}_{i} \otimes \mbf K''_{n-i}  \mbf F''_i  + 1 \otimes \mbf E''_{n-i}. \\
\Delta^{\B}_{\bv} (\mbf k_i) & = \mbf k'_i \otimes \mbf K''_i \mbf K''^{-1}_{n-i}.
\end{split}
\end{align*}
\end{prop}

\begin{proof}
Fix an $i \in [1, r]$.
Assume that we have a quadruple $(\mbf b', \mbf a', \mbf b'', \mbf a'')$ such that
$b'_k = a'_k - \delta_{k, i} + \delta_{k, i+1} + \delta_{k, n-i} - \delta_{k, n+1 -i}$ and $b''_k = a''_k$ for all $k \in [1, n]$.
We have
$u(\mbf b'', \mbf a'') =0 $ and
$
\sum_{1\leq k\leq j \leq n} b'_k a''_j - a'_k a''_j = - a''_i + a''_{n-i}.
$
So, after the twist,  the first term on the right of $\widetilde \Delta^{\B}_{\bv} (\E_i)$ in Proposition ~\ref{D(E)} becomes
$\E_i' \otimes \mbf H_{i+1}'' \mbf H''^{-1}_{n-i}|_{ \mbf b', \mbf a', \mbf b'', \mbf a''} v^{-a''_i + a''_{n-i}} = \E'_i \otimes \mbf K''_i|_{\mbf b', \mbf a', \mbf b'', \mbf a''}$ where the notation $f|_{\mbf b', \mbf a', \mbf b'', \mbf a''}$ is the restriction of $f$ to
$X^{\jmath}_{d'} (\mbf b')  \times X^{\jmath}_{d'} (\mbf a') \times X_{d''}(\mbf b'') \times X_{d''} (\mbf a'')$.

Assume that we have a quadruple $(\mbf b', \mbf a', \mbf b'', \mbf a'')$ such that
$b'_k = a'_k$ and $b''_k = a''_k - \delta_{k, i} + \delta_{k, i+1}$ for all $k\in [1, n]$.
Then we have
$ \sum_{1 \leq k \leq j \leq n} b'_k b''_j - a'_k a''_j = a'_{i+1}$ and $u(\mbf b'', \mbf a'') = a''_{n-i}$.
Thus, after the twist,  the second term on the right of $\widetilde \Delta^{\B}_{\bv} (\E_i)$ in Proposition ~\ref{D(E)} is equal to
$\mbf h'^{-1}_{i+1} \otimes \mbf E_i''  \mbf H''^{-1}_{n-i}|_{\mbf b', \mbf a', \mbf b'', \mbf a''}  v^{a'_{i+1} + a''_{n-i}} = 1 \otimes \mbf E''_i |_{\mbf b', \mbf a', \mbf b'', \mbf a''}$.

Assume that we have a quadruple $(\mbf b', \mbf a', \mbf b'', \mbf a'')$ such that
$b'_k = a'_k$ and $b''_k = a''_k + \delta_{k, n-i} - \delta_{k, n+1-i}$ for all $k\in [1, n]$.
Then we have
$
 \sum_{1 \leq k \leq j \leq n} b'_k b''_j - a'_k a''_j = - a'_i
$
and
$u(\mbf b'', \mbf a'') =  - a''_i + \delta_{i, r+1}= - a''_i$, where  the latter equality  is due to $i\in [1, r]$.
Hence,   after the twist,  the third term on the right of $\widetilde \Delta^{\B}_{\bv} (\E_i)$ in Proposition ~\ref{D(E)}
is equal to
$ \mathbf h'_{i+1} \otimes \mbf F''_{n-i} \mbf H''_{i+1} v^{-a'_i - a''_i} |_{\mbf b', \mbf a', \mbf b'', \mbf a''} = \K'_i \otimes \mbf F''_{n-i} \mbf K''_i|_{\mbf b', \mbf a', \mbf b'', \mbf a''}.$

The first equality in the proposition follows from the above analysis.

Assume that we have a quadruple $(\mbf b', \mbf a', \mbf b'', \mbf a'')$ such that
$b'_k = a'_k + \delta_{k, i} - \delta_{k, i+1} - \delta_{k, n-i} + \delta_{k, n+1-i}$ and $b''_k = a''_k$
for all $k\in [1, n]$.
Then we have
$ \sum_{1 \leq k \leq j \leq n} b'_k b''_j - a'_k a''_j = a''_i - a''_{n-i}$ and $u(\mbf b'', \mbf a'') =0$.
So, after the twist,  the first term on the right of $\widetilde \Delta^{\B}_{\bv} (\F_i)$ in Proposition ~\ref{D(E)} becomes
$\F'_i \otimes \mbf H''^{-1}_{i} \mbf H''_{n+1-i} v^{a''_i - a''_{n-i}}|_{\mbf b', \mbf a', \mbf b'', \mbf a''} =  \F'_i \otimes \mbf K''_{n-i}|_{\mbf b', \mbf a', \mbf b'', \mbf a''}$.

Assume that we have a quadruple $(\mbf b', \mbf a', \mbf b'', \mbf a'')$ such that
$b'_k = a'_k$ and $b''_k = a''_k + \delta_{k, i} - \delta_{k, i+1}$
for all $k\in [1, n]$.
Then we have
$ \sum_{1 \leq k \leq j \leq n} b'_k b''_j - a'_k a''_j = - a'_{i+1}$ and $u (\mbf b'', \mbf a'') = - a_{n-i}+ \delta_{i, r}$.
So, after the twist,  the second term on the right of $\widetilde \Delta^{\B}_{\bv} (\F_i)$ in Proposition ~\ref{D(E)} becomes
$
\mbf h'_i\otimes \mbf F''_i \mbf H''_{n+1-i} v^{-a'_{i+1} - a_{n-i} + \delta_{i, r}}|_{\mbf b', \mbf a', \mbf b'', \mbf a''}
= \K'^{-1}_{i} \otimes v^{\delta_{i, r}}  \mbf F''_i   \mbf K''_{n-i}|_{\mbf b', \mbf a', \mbf b'', \mbf a''}
= \K'^{-1}_{i} \otimes \mbf K''_{n-i}  \mbf F''_i |_{\mbf b', \mbf a', \mbf b'', \mbf a''} .
$

Assume that we have a quadruple $(\mbf b', \mbf a', \mbf b'', \mbf a'')$ such that
$b'_k = a'_k$ and $b''_k = a''_k - \delta_{k, n-i} + \delta_{k, n+1-i}$ for all $k \in [1, n]$.
Then we have
$ \sum_{1 \leq k \leq j \leq n} b'_k b''_j - a'_k a''_j =  a'_{n+1-i} = a'_i$ and $u(\mbf b'', \mbf a'') = a''_i$.
So, after the twist,  the third term on the right of $\widetilde \Delta^{\B}_{\bv} (\F_i)$ in Proposition ~\ref{D(E)} becomes
$
\mbf h'^{-1}_{i} \otimes \mbf E''_{n-i} \mbf H''^{-1}_{i} v^{a'_i + a''_i} |_{\mbf b', \mbf a', \mbf b'', \mbf a''} = 1 \otimes \mbf E''_{n-i}|_{\mbf b', \mbf a', \mbf b'', \mbf a''}.
$

The above analysis implies the second equality in the proposition.
Since the twist will not affect the original term when $\mbf b'= \mbf a'$ and $\mbf b''=\mbf a''$, we have the third equality.
\end{proof}

Moreover, we have

\begin{prop} \label{Delta-A-B}
$(1 \otimes \Delta_{\bv}) \Delta^{\B}_{\bv} = (\Delta^{\B}_{\bv} \otimes 1) \Delta^{\B}_{\bv}$. More precisely, we have the following commutative diagram for
the quadruple $d, d', d'', d'''$ such that $d= d' + d'' + d'''$
\[
\begin{CD}
\mbf S^{\B}_d @> \Delta^{\B}_{\bv} >> \mbf S^{\B}_{d'} \otimes \mbf S_{d'' + d'''} \\
@V \Delta^{\B}_{\bv} VV @VV 1\otimes \Delta_{\bv} V\\
\mbf S^{\B}_{d'+d''} \otimes \mbf S_{d'''} @> \Delta^{\B}_{\bv} \otimes 1 >>   \mbf S^{\B}_{d'} \otimes \mbf S_{d''} \otimes \mbf S_{d'''}.
\end{CD}
\]
\end{prop}

\subsection{The imbedding  $\jmath_{d, \bv}: \mbf S^{\B}_d \to \mbf S_d$}

In this section, we set $d'=0$ and $d''=d$, then the comulitiplication  $\Delta^{\B}_{\bv}$ in (\ref{Delta-B}) becomes
$
\Delta^{\B}_{\bv}: \mbf S^{\B}_d \to \mbf S^{\B}_{0} \otimes \mbf S_d.
$
Observe that $\mbf S^{\B}_{0}$ consists of only one basis element, so we have
$
\mbf S^{\B}_0 \simeq \mathcal A.
$
Thus the coproduct  $\Delta^{\B}_{\bv}$ becomes the following algebra homomorphism, denoted by $\jmath_{d, \bv}$,
\begin{align} \label{jD}
\jmath_{d, \bv}: \mbf S^{\B}_d \to \mbf S_d.
\end{align}

The following corollary is  by  Proposition ~\ref{D(E)-renorm},  $\E_i=0$, $\F_i=0$ and $\mbf k_i' = \bv^{\delta_{i, r}}$ in $\mbf S^{\B}_0$.

\begin{cor}
\label{J-A}
We have
\begin{align*}
\begin{split}
\jmath_{d, \bv} ( \E_i) & =  \mbf  E_i + \mbf K_i \mbf F_{n-i}, \\
\jmath_{d, \bv} (\F_i) & = \mbf F_i \mbf K_{n-i} + \mbf E_{n-i} , \\
\jmath_{d, \bv} (\mbf k_i ) & = \bv^{\delta_{i, r}} \mbf K_i \mbf K^{-1}_{n-i}, \quad \forall 1\leq i\leq r.\\
%
\end{split}
\end{align*}
\end{cor}

\begin{proof}
By Proposition ~\ref{D(E)-renorm},
we have
$$
\jmath_{d,\bv} ( \E_i) =  \E_i' \otimes  \mbf K''_{i} + 1 \otimes \mbf E_i''   +  \mbf k_i'  \otimes \mbf F''_{n-i} \mbf K''_i
= 0 + \mbf E''_i + v^{\delta_{i, r}} \mbf F''_{n-i} \mbf K''_i
= \mbf E''_i + \mbf K''_i \mbf F''_{n-i},
$$
Which is the first identity if we skip the superscripts.
The rest two are obtained in exactly the same manner, and skipped.
\end{proof}

\begin{rem}
The homomorphism $\jmath_{d, \bv}$ matches with the imbedding $\jmath$ in  ~\cite[Proposition 4.5]{BKLW13}. The only difference is an involution $\omega$ on $\U$ defined by $(\mbb E_i, \mbb F_i, \mbb  K_i) \mapsto ( \mbb F_i, \mbb E_i, \mbb K^{-1}_{i})$.
\end{rem}

\begin{lem} \label{j-inj}
$\jmath_{d, \bv}$ in (\ref{jD})  is injective.
\end{lem}

\begin{proof}
Recall from \cite[Theorem 3.10]{BKLW13} that $\mbf S^{\jmath}_d$ has  a monomial basis  $m^{\jmath}_A$ indexed by $A\in \Xi^{\jmath}_d$ (which is denoted $m_A$ therein).
It is enough to show that  the set
$\{ \jmath_{d, \bv}  (m_A^{\jmath}) | A\in \Xi^{\jmath}_d \}$
is linearly independent in $\mbf S_d$.
We set
\[
\deg(1_{\lambda}) =0,\  \deg (\E_i 1_{\lambda} ) = i,\  \deg (\F_i 1_{\lambda}) = n-i, \quad \forall  \lambda \in \Lambda_{d, n}^{\jmath}, 1\leq i \leq r.
\]
Similarly, we define
\[
\deg (1_{\lambda} ) =0,\  \deg (\mbf E_i1_{\lambda}) = i,\  \deg (\mbf F_i1_{\lambda}) = -i,  \quad \forall \lambda \in \Lambda_{d, n}, 1\leq i \leq n.
\]
We write $\nu' < \nu$ if $\nu'_i \leq \nu_i$ for all $i$ and $\nu'_{i_0} < \nu_{i_0}$ for some $i_0$.
Suppose that $\deg (m_A^{\jmath}) = \nu \in \mbb Z_{\geq 0} [I]$. By Corollary \ref{J-A}, we have
\[
\jmath_{d, \bv} ( m_A^{\jmath})  \in
\oplus_{\overline{\mbf b} -\overline{ \mbf a} = \nu}  \mbf S_d (\mbf b, \mbf a) \oplus
\oplus_{\overline{\mbf d} -\overline{\mbf c} < \nu} \mbf S_d (\mbf d, \mbf c).
\]
For $A=(a_{ij}) \in \Xi_d^{\jmath}$, we set
\[
\Xi_d(A) =\{ B=(b_{ij}) \in \Xi_d | b_{ij} =0, \forall i < j, b_{ij}  = a_{ij}, \forall i > j, \co (B) \vdash \co (A) \},
\]
where $\mbf b \vdash \mbf a$ if $b_i + b_{n+1-i} + \delta_{i, r+1}= a_i$ for all $1\leq i\leq n$.
By Corollary ~\ref{J-A}, we see  that
$$
\jmath_{d, \bv} (m_A^{\jmath}) = \sum_{B \in \Xi_d (A)} m_B + \mbox{lower terms},
$$
where $m_B$ denotes the monomial attached to $B$ in ~\cite[Proposition 3.9]{BLM90} and  `lower term' is the remaining summand  in $\oplus_{\overline{\mbf d} -\overline{\mbf c} < \nu} \mbf S_d (\mbf d, \mbf c)$.
Now suppose that we have
\[
\sum_{A\in \Xi_d^{\jmath}} c_A \jmath_{d, \bv} (m^{\jmath}_A) =0, \quad c_A \in \mathcal A.
\]
Let $\mathcal M$ be the set of maximal $\nu \in \mbb Z[I]$ in the set
$\{ \deg (m^{\jmath}_A) | A\in \Xi^{\jmath}_d\}$
with respect to the natural partial order in $\mbb Z[I]$, i.e., $\nu' \leq \nu$ if and only if $\nu_i'\leq \nu_i$ for all $i$.
We have
\[
0= \sum_{A\in \Xi_d^{\jmath}} c_A \jmath_{d, \bv} (m^{\jmath}_A)
= \sum_{A: \deg (m^{\jmath}_A) \in \mathcal M}  c_A \jmath_{d, \bv} (m^{\jmath}_A) + \mbox{lower term}.
\]
So we have
$
\sum_{A: \deg (m^{\jmath}_A) \in \mathcal M }  c_A \jmath_{d, \bv} (m^{\jmath}_A) =0.
$
By ~\cite[Proposition 3.9]{BLM} and the fact that $\Xi_d(A) \cap \Xi_d(A') =\emptyset$ if $A\neq A'$,
the set $\{ \sum_{B\in \Xi_d(A)} m_B \}$, where $A$ runs over all matrices in $\Xi_d^{\jmath}$
such that $ \deg(m_A^{\jmath})  \in \mathcal M$, is linearly independent in $\mbf S_d$.
Thus, $c_A=0$ for all $A\in \Xi_d^{\jmath}$ such that $\deg (m_A^{\jmath}) \in  \mathcal M$.
Inductively,  $c_A=0$ for all $A\in \Xi_d^{\jmath}$.
Therefore, the set $\{ \jmath_{d, \bv} (m^{\jmath}_A ) | A\in \Xi_d^{\jmath} \}$ is linearly independent.
\end{proof}

The following   is nothing but a special case of Proposition ~\ref{Delta-A-B}.

\begin{cor} \label{D-B-J}
Suppose that $d' + d'' =d$.
We have the following commutative diagram.
\[
\begin{CD}
\mbf S^{\B}_d @>\Delta^{\B}_{\bv}>> \mbf S^{\B}_{d'} \otimes \mbf S_{d''}\\
@V\jmath_{d, \bv} VV @VV\jmath_{d', \bv} \otimes 1V\\
\mbf S_d @>\Delta_{\bv} >> \mbf S_{d'}\otimes \mbf S_{d''}.
\end{CD}
\]
\end{cor}

\begin{rem}
$\mbf S^{\B}_d$ can be regarded as a `coideal' subalgebra of $\mbf S_d$ in view of
Lemma ~\ref{j-inj} and Corollary ~\ref{D-B-J}.
\end{rem}

\subsection{Type $A$ duality vs type $B$ duality}

In this section, we  use the algebra homomorphism $\jmath_{d,\bv}$ to establish a direct connection between the geometric type $A$ duality in ~\cite{GL92} and
the geometric type $B$ duality in ~\cite{BKLW13}.

For any nonnegative integers $a, b$, we write $1^a 0^b$ for the sequence $(1, \cdots, 1, 0, \cdots, 0)$ containing $a$ copies of $1$'s and $b$ copies of $0$'s.
Similarly, we can define $1^a 0^b 1^c$, etc.

Recall $X_d$, $X_d(\mbf b)$ for $\mbf b\in \Lambda_{d,n}$  from Section ~\ref{vSA} and ~\ref{CoSA}.
We set
\[
\mbf T_{d, n} = \mathcal A_{\G_d} ( X_d \times X_d ( 1^d) ),
\quad \mbox{and} \quad
\mbf H_{A_d} = \mathcal A_{\G_d} (X_d(1^d) \times X_d (1^d)).
\]
By ~\cite{GL92}, we know that $\mbf H_{A_d} $ is a Hecke algebra of type $A_d$ and $\mbf T_{d, n}$  is a tensor space $\mbf V_n^{\otimes d}$ where $\mbf V_n$ is a free $\mathcal A$-module of rank $n$. Now the standard convolution defines commuting actions of $\mbf S_d$ and $\mbf H_{A_d}$ on
$\mbf T_{d,n}$ from the left and the right, respectively, which is captured in the following diagram.
\begin{align}
\label{type-A-duality}
\mbf S_d \times \mbf T_{d,n} \to \mbf T_{d, n} \overset{\psi}{\leftarrow} \mbf T_{d,n} \times \mbf H_{A_d}.
\end{align}
Moreover, the two actions centralize each other.

We shall recall a similar picture in ~\cite{BKLW13} if the $X_d$ is replaced by its $\jmath$-analogue.
Recall $X^{\jmath}_d$, $X^{\jmath}_d(\mbf b)$ for $\mbf b\in \Lambda^{\jmath}_{d,n}$  from Section ~\ref{vSB} and ~\ref{secrenormal}.
We set
\begin{align}
\label{Tj}
\mbf T^{\jmath}_{d, n} = \mathcal A_{\G^{\jmath}_d} ( X^{\jmath}_d \times X^{\jmath}_d ( 1^{2d+1}) )
\quad \mbox{and} \quad
\mbf H_{B_d} = \mathcal A_{\G^{\jmath}_d} (X^{\jmath}_d(1^{2d+1}) \times X^{\jmath}_d (1^{2d+1})).
\end{align}
Then we have that $\mbf T^{\jmath}_{d,n}$ is also isomorphic to the tensor space $\mbf V_n^{\otimes d}$, and the following diagram of commuting actions.
\begin{align}
\label{type-j-duality}
\mbf S^{\jmath}_d \times \mbf T^{\jmath}_{d,n} \to \mbf T^{\jmath}_{d, n} \leftarrow \mbf T^{\jmath}_{d,n} \times \mbf H_{B_d}.
\end{align}

A slight variant of the imbedding $\jmath_{d, \bv}$ yields the following linear map
\[
\zeta'_{d, \mbf b, \bv}: \mathcal A_{\G^{\jmath}_d} (X^{\jmath}_d (\mbf b) \times X^{\jmath}_d (1^{2d+1}) )
\to
\oplus_{\mbf b'' \models \mbf b, \mbf a'' \models 1^{2d+1} }  \mathcal A_{\G_d} (X_d(\mbf b'') \times X_d (\mbf a'')),
\forall \mbf b\in \Lambda^{\jmath}_{d,n},
\]
where $\mbf b'' \models \mbf b$ stands for $b_i = b_i'' + b''_{n+1-i} + \delta_{i, r+1}$ for all $i$.

For $\mbf a'', \mbf b''  \models 1^{2d+1}$, we set
\[
\mbf T_{d, n}^{\mbf a''} = \mathcal A_{\G_d} ( X_d \times X_d ( \mbf a'') )
\quad \mbox{and} \quad
^{\mbf b''} \mbf H_{A_d} = \mathcal A_{\G_d} (X_d(\mbf b'') \times X_d (1^d)).
\]

Let $\zeta_{d, \mbf b, \bv}$ denote the composition of $\zeta'_{\mbf b, \bv}$ with the projection to the components of $\mbf a'' =1^d 0^{d+1}$, i.e.,
\[
\zeta_{d, \mbf b, \bv}: \mathcal A_{\G^{\jmath}_d} (X^{\jmath}_d (\mbf b) \times X^{\jmath}_d (1^{2d+1}) )
\to
\oplus_{\mbf b'' \models \mbf b }  \mathcal A_{\G_d} (X_d(\mbf b'') \times X_d (1^d)),
\]
where we identify $X_d(1^d 0^{d+1})$ with $X_d (1^d)$.
Summing over all $\mbf b\in \Lambda^{\jmath}_{d,n}$, we get a linear map
\[
\zeta_{d, \bv} \equiv \oplus_{\mbf b\in \Lambda_{d, n}} \zeta_{d, \mbf b, \bv}: \mbf T^{\jmath}_{d, n} \to \mbf T_{d, n}.
\]
Take $n=2d+1$, $\mbf b = 1^{2d+1}$, we obtain a linear map
\[
\zeta^1_{d, \bv}: \mbf H_{B_d} \to \oplus_{\mbf b'' \models 1^{2d+1}} \ ^{\mbf b''} \mbf H_{A_d},
\]
which is not necessarily an algebra homomorphism.
Note that we identify $\mbf T_{d, n}^{1^{d} 0^{d+1}}$ and $^{1^d 0^{d+1}} \mbf H_{A_d}$ with $\mbf T_{d, n}$ and $\mbf H_{A_d}$, respectively.


\begin{prop}
\label{typeA-typeB}
We have the following commutative diagram relating the geometric type $A$ duality with the geometric type $B$ duality.
\[
\begin{CD}
\mbf S^{\jmath}_d \times \mbf T^{\jmath}_{d,n} @>>>  \mbf T^{\jmath}_{d, n} @<<< \mbf T^{\jmath}_{d,n} \times \mbf H_{B_d} \\
@V\jmath_{d, \bv} \times \zeta_{d, \bv} VV @V\zeta_{d, \bv}VV @VV \zeta_{d, \bv} \times \zeta^1_{d, \bv}V \\
\mbf S_d \times \mbf T_{d,n} @>>>  \mbf T_{d, n}  @<<\psi_1<  \oplus_{\mbf a'' \models 1^{2d+1} } \mbf T^{\mbf a''}_{d,n} \times  \oplus_{\mbf b'' \models 1^{2d+1}} \ ^{\mbf b''} \mbf H_{A_d} @<<\psi_2< \mbf T_{d, n} \times \mbf H_{A_d}.
\end{CD}
\]
where $\psi_2$ is the natural imbedding and $\psi_2\psi_1 $ is the $\psi$ in (\ref{type-A-duality}).
\end{prop}

We can describe the linear map $\zeta_{d,\bv}$ explicitly.
Let $\Pi_{d, n}$ be the set of $n\times d$ matrices $A$ such that
$a_{ij} \in \{0, 1\}$ and $\sum_{1\leq i \leq n} a_{ij} =1$ for all $1\leq j \leq d$.
Then we have
\begin{align}
\label{Td}
\mbf T_{d, n}  = \mbox{span}_{\mathcal A} \{ \ ^a[A] | A \in \Pi_{d,n}\}
\end{align}
where $^a[A] = v^{d_A} \zeta_A$ and $d_A = \sum_{i\geq k, j < l} a_{ij} a_{kl}$.

Let $\Pi^{\jmath}_{d, n}$ be the subset of $\Pi_{2d+1, n}$ such that $a_{ij} = a_{n+1-i, 2d+2 -j}$ for all $1\leq i \leq n$ and $1\leq j \leq 2d+1$.
(In particular, we have $a_{r+1, d+1}=1$.)
We have
\begin{align}
\label{Tj-standard}
\mbf T^{\jmath}_{d, n} = \mbox{span}_{\mathcal A} \{ [A] | A \in \Pi^{\jmath}_{d, n} \},
\end{align}
where $[A] = v^{\ell_A} \zeta^{\jmath}_A$ and $\ell_A = \frac{1}{2} \left ( \sum_{i \geq j, k < l} a_{ij} a_{kl} - \sum_{i \geq n+1, d+1 >j} a_{ij} \right )$.

Let $J_m$ be the $m\times m$ matrix whose $(i, j)$-th component is $\delta_{i, n+1-j}$ for all $1\leq i, j \leq m$.
To a matrix $A\in \Pi_{d, n}$, we define a matrix
\[
A^J = ( A | \epsilon_{r+1} | J_n A J_d ),
\]
where $\epsilon_{r+1}$ is the column vector whose entries are zero except at $r+1$ which is $1$.
Then the assignment $A \mapsto A^J$ defines a bijection $\Pi_{d, n} \overset{\simeq}{\to} \Pi^{\jmath}_{d,n}$.

\begin{prop}
\label{zeta-standard}
$\zeta_{d, \bv} ( [ A^{J}] ) = \ ^a[A]$, for all $A \in \Pi_{d, n}$.
\end{prop}

\begin{proof}
Suppose that $\ro(A^J) = \mbf b$.
We set $\mbf a'' = 1^d 0^{d+1}$, $\mbf b' = 0^{r} 1^1 0^r$ and $\mbf a' = 0^d 1^1 0^d$.
Then by the definition of $\zeta_{d, \bv}$, we have
\begin{align}
\label{z-A}
\zeta_{d, \bv} ( [A^J]) = \bv^{t_{\mbf b''} }
\widetilde \Delta^{\jmath}_{\mbf b', \mbf a', \mbf b'',\mbf a'' } ([A^J])
\end{align}
where
\[
t_{\mbf b''} = \sum_{1\leq i \leq j \leq n} b'_i b''_j - \sum_{1\leq i \leq j \leq 2d+1} a'_i a''_i
+ \frac{1}{2} ( \sum_{i+j \geq n+1} b''_i b''_j - \sum_{i\geq r+1} b''_i  - \sum_{i+j \geq 2d+1} a''_i a''_j + \sum_{i \geq d+1} a''_i).
\]
Note that the following formula $t_{\mbf b''}$ is compatible with the twist in (\ref{Delta-J}), since we need to rescale from $n$ components to $2d+1$ components for $\mbf a'$ and $\mbf a''$.
Now using the fact that $\mbf a'' = 1^d 0^{d+1}$, $\mbf b' = 0^r 1^1 0^r$ and $\mbf a' = 0^d 1^1 0^d$, the twist $t_{\mbf b''}$ can be simplified to
\[
t_{\mbf b''} = \frac{1}{2} ( \sum_{i+j \geq n+1} b''_i b''_j - \sum_{i\geq r+1} b''_i).
\]
By the definition of $A^J$, we can also simplify the numeric $\ell_{A^J}$ as follows.
\begin{align}
\begin{split}
\ell_{A^J} & = \frac{1}{2} ( \sum_{i \geq k, j<l} a^J_{ij} a^J_{kl}  - \sum_{i\geq r+1, d+1 >j} a^J_{ij}), \quad A^J=(a^J_{ij}) \\
& = \frac{1}{2}  (  ( \sum_{i \geq k, j<l < d+1} + \sum_{i \geq k, j < d+1 \leq l} + \sum_{i\geq k, d+1 \leq j <l} ) a^J_{ij} a^J_{kl} - \sum_{i\geq r+1, d+1 >j} a_{ij}).
\end{split}
\end{align}
The first sum simplifies to
$
 \sum_{i\geq k, j < l < d+1} a_{ij} a_{kl}.
 $
The third sum simplifies to
\[
  \sum_{i \geq k, d+1 \leq j < l} a^J_{ij} a^J_{kl}
 =  \sum_{i \geq k, j < d+1 \leq l} a_{n+1-i, 2d+2-j} a_{n+1-k, 2d+2-l}
 =\sum_{i\geq k, j < l < d+1} a_{ij} a_{kl} + \sum_{i \geq r+1, d+1 >j} a_{ij}.
\]
The second sum is reduced to
\[
 \sum_{i \geq k, j < d+1 \leq l} a^J_{ij} a^J_{kl}
 =
 \sum_{i +k\geq n+1, j, l < d+1} a_{ij} a_{kl}  + \sum_{i\geq r+1, d+1> j} a_{ij}
 = 2 t_{\ro(A)}.
\]
So we get
$
t_{\mbf b''} - \ell_{A^J} = - d_A + t_{\mbf b''} - t_{\ro(A)},
$
Thus the identity (\ref{z-A}) can be rewritten as
\[
\zeta_{d, \bv} ( [A^J])  =
v^{-d_A + t_{\mbf b''} - t_{\ro(A)}}  \widetilde \Delta^{\jmath}_{\mbf b', \mbf a', \mbf b'',\mbf a''} (\zeta^{\jmath}_{A^J})
\]
where $\zeta^{\jmath}_{A^J}$ denote the characteristic function attached to the $\G^{\jmath}_d$-orbit indexed by $A^J$.

Recall that $\mbf a'' = 1^d 0^{d+1}$. This implies that for any $\tilde L'' \in X_d (\mbf a'') $, we have $Z^{\jmath}_{\tilde L', \tilde L''}$ consists of only one point, i.e., the flag $\tilde L$ such that $\tilde L_i = \tilde L''_{i}$ for all $i\leq r$ and $\tilde L_i = (\tilde L''_{n+1-i})^{\perp}$ for all $i\geq r+1$.
Furthermore, if $(L'', \tilde L'') \in \mathcal O_{A}$, then $(L, \tilde L) \in \mathcal O_{A^J}$
for any $L\in Z^{\jmath}_{L', L''}$ and $\tilde L \in Z^{\jmath}_{\tilde L', \tilde L''}$ because
$L_i \cap \tilde L_j = L''_i \cap \tilde L''_j, \quad \forall j \leq d$.
Hence we have
$
\widetilde \Delta^{\jmath}_{\mbf b', \mbf a', \mbf b'',\mbf a''} (\zeta^{\jmath}_{A^J}) =\delta_{\mbf b'', \ro(A)} \zeta_A.
$
The proposition is proved.
\end{proof}

By Proposition ~\ref{zeta-standard}, we have

\begin{thm}
\label{tensor-positivity}
For all $A \in \Pi_{d,n}$, we have
$\zeta_{d, \bv} (\{A^J\}) = \ ^a \{A\} + \sum_{B^J \prec A^J, \ro(B) \neq \ro(A)} c_{B, A} \ ^a \{B\}$ where $c_{B, A} \in \mbb Z_{\geq 0} [\bv, \bv^{-1}]$.
\end{thm}

Recall the parabolic Kazhdan-Lusztig  polynomials $P_{B^J, A^J}$ and $P_{B, A}$ of type $B_d$ and $A_d$, respectively.

\begin{cor}
$P_{B^J, A^J} = P_{B, A}$ if $\ro(B) = \ro(A)$.
\end{cor}

More generally, we have the following commutative diagram of algebras.
\[
\begin{CD}
\mbf H_{B_d}=\mathcal A_{\G^{\jmath}_d} ( X^{\jmath}_d (1^{2d+1}) \times X^{\jmath}_d(1^{2d+1}))
@>\zeta_{d,\bv}>>
\oplus_{\mbf b'', \mbf a'' \models 1^{2d+1}} \mathcal A_{\G_d} (X_d(\mbf b'') \times X_d(\mbf a'')) \\
@VVV @VVV\\
\mbf S_{d, 2d+1}^{\jmath} @>>\jmath_{d,\bv} >
 \mbf S_{d, 2d+1}.
\end{CD}
\]



\subsection{Transfer maps on $\mbf S^{\B}_d$}

The transfer map
\begin{align*}
\phi^{\B}_{d, d-n, \bv}: \mbf S^{\B}_d \to \mbf S^{\B}_{d-n}
\end{align*}
is defined to be  the  composition:
$
\xymatrix{
\mbf S^{\B}_d \ar[r]^-{\widetilde \Delta^{\B}}&\mbf S^{\B}_{d-n} \otimes \mbf S_{n}
\ar[r]^-{1\otimes \chi}&
\mbf S^{\B}_{d - n }\otimes \mathcal A \equiv \mbf S^{\B}_{d-n},}
$
where $\chi$ is in (\ref{chi}).
It is clear that  $\phi^{\B}_{d, d-n, \bv}$ is an algebra homomorphism. Moreover, we have

\begin{prop}
\label{Phi-gene}
$\phi^{\B}_{d, d-n, \bv} (\E_i ) =\E_i'$, $\phi^{\B}_{d, d-n, \bv}(\F_i) =\F_i', $
and  $ \phi^{\B}_{d, d-n, \bv}(\K_i^{\pm 1})=\K_i'^{\pm 1}$, for any $i\in [1, r]$.
\end{prop}

\begin{proof}
By definitions, we have
$
\chi(\mbf E''_i) =0, \quad
\chi (\mbf F''_i) =0,
\quad \mbox{and}\quad
\chi ( \mbf H''_i) = \bv.
$
So we have
\begin{align*}
\phi^{\B}_{d, d-n, \bv} (\E_i) &=
\E_i' \chi(\mbf H''_{i+1}\mbf H''^{-1}_{n-i})
+ \mbf h'^{-1}_{i+1} \chi (\mbf E_i''  \mbf H''^{-1}_{n-i})
+ \mathbf h'_{i+1} \chi( \mbf F''_{n-i} \mbf H''_{i+1})
=\E'_i,\\
\phi^{\B}_{d, d-n, \bv} (\F_i) &=
\F'_i \chi(\mbf H''^{-1}_{i} \mbf H''_{n+1-i}) + \mbf h'_i\chi( \mbf F''_i \mbf H''_{n+1-i}) + \mbf h'^{-1}_{i} \chi(\mbf E''_{n-i} \mbf H''^{-1}_{i})=\F'_i,\\
\phi^{\B}_{d, d-n, \bv} ( \K_i) &= \mbf k'_i \chi( \mbf K''_i \mbf K''^{-1}_{n+1-i}) = \K'_i.
\end{align*}
The lemma is proved.
\end{proof}

Together with Theorem 3.10 in ~\cite{BKLW13}, we have

\begin{cor}
The homomorphism $\phi^{\B}_{d, d-n, \bv}$ is surjective.
\end{cor}

\subsection{Generic version} \label{generic-B}

Recall the definition of $\bv=\sqrt{q}$ from (\ref{A}) and $\mbb A=\mbb Z[v,  v^{-1}]$.
Recall from  ~\cite{BKLW13} that one can construct an associative algebra $\mbb S^{\B}_d$ over $\mbb A$ such that
\[
\mbf S^{\B}_d =\mathcal A \otimes_{\mbb A} \mbb S^{\B}_d,
\]
where $\mathcal A$ is regarded as an $\mbb A$-module with $v$ acting as $\bv$.
Let us make the algebra $\mbb S^{\B}_d$ more precise. Recall $\Xi_d$ from Section ~\ref{Transfer}.
Consider the set
\begin{align}
\label{Xij}
\Xi_d^{\jmath} =\{ M\in \Xi_d | m_{ij} = m_{n+1-i, n+1-j}, \forall 1\leq i, j\leq n\}.
\end{align}
Then $\mbb S^{\B}_d$ is a free $\mbb A$-module with basis $\zeta^{\jmath}_M$ for any $M\in \Xi_d^{\jmath}$ whose multiplication is defined by the condition that  if
$\zeta^{\jmath}_{M_1} \zeta^{\jmath}_{M_2} = \sum_{M \in \Xi_d^{\jmath}} h^M_{M_1, M_2}(v) \zeta^{\jmath}_M$, where $h^M_{M_1, M_2}(v) \in \mbb A$,
then
$\eta^{\jmath}_{M_1} \eta^{\jmath}_{M_2} = \sum_{M \in \Xi_d^{\jmath}} h^M_{M_1, M_2}(v)|_{v=\bv}  \eta^{\jmath}_M$, in $\mbf S_d^{\B}$,
where $\eta^{\jmath}_M$ is the characteristic function of the $\G^{\B}_d$-orbit in $X^{\B}_d\times X^{\B}_d$ indexed by $M$ via (\ref{para}).
Let $\mbb S^{\jmath}_d(\mbf b, \mbf a)=\mbox{span}_{\mbb A} \{ \zeta_M | \ro (M) =\mbf b, \co (M) = \mbf a\}$. We have
$\mbb S^{\jmath}_d(\mbf c, \mbf b')\mbb S^{\jmath}_d(\mbf b, \mbf a) \subseteq \delta_{\mbf b', \mbf b} \mbb S^{\jmath}_d(\mbf c, \mbf a)$.

By using the monomial basis in ~\cite[Theorem 3.10]{BKLW13}, one can show that $\mbb S^{\B}_d$ enjoys the same results for $\mbf S^{\B}_d$ from the previous sections. The following is a generic version of Proposition ~\ref{D(E)-renorm}.

\begin{prop}
\label{J-1}
There is  a unique algebra homomorphism
\begin{align}
\label{Generic-D-B}
\mbb \Delta^{\B}: \mbb S^{\B}_d \to \mbb S^{\B}_{d'} \otimes \mbb S_{d''}
\end{align}
such that
$\mathcal A\otimes_{\mbb A} \mbb  \Delta^{\B} = \Delta^{\B}_{\bv}$ and
\begin{align*}
\begin{split}
\mbb \Delta^{\B} ( \E_i) & = \E_i' \otimes  \mbf K''_{i} + 1 \otimes \mbf E_i''   +  \mbf k_i'  \otimes \mbf F''_{n-i} \mbf K''_i, \\
\mbb \Delta^{\B}  (\F_i) & = \F'_i \otimes \mbf K''_{n-i}  +  \K'^{-1}_{i} \otimes \mbf K''_{n-i}  \mbf F''_i  + 1 \otimes \mbf E''_{n-i}, \\
\mbb \Delta^{\B} (\mbf k_i) & = \mbf k'_i \otimes \mbf K''_i \mbf K''^{-1}_{n-i}, \quad \forall 1 \leq i \leq r.
\end{split}
\end{align*}
\end{prop}

Recall the canonical basis $\{ \{ M \}_d | M \in \Xi_d^{\jmath} \}$ of $\mbb S^{\B}_d$ from ~\cite[3.6]{BKLW13}.
We have the following positivity result of the canonical basis of $\mbb S^{\B}_d$ with respect to the coproduct $\Delta^{\B}$.

\begin{prop}
\label{J-2}
If
$
\mbb \Delta^{\B} ( \{M\}_d) = \sum_{M'\in \Xi_{d'}^{\jmath}, M'' \in \Xi_{d''}} h_M^{M', M''} \{M'\}_{d'} \otimes \{M''\}_{d''},
$
then we have $h_M^{M', M''} \in \mbb Z_{\geq 0} [v, v^{-1}]$.
\end{prop}

\begin{proof}
The proof is similar to that of Proposition ~\ref{Pos-S}.
We consider the orthogonal group $\overline{\mbf G}_d^{\jmath}$   and the isotropic flag variety $\overline{\mbf X}_d^{\jmath}(\mbf a)$ over $\overline{\mbb F_q}$, whose $\mbb F_q$-points are exactly $\G_d^{\jmath}$ and  $X_d^{\jmath} (\mbf a)$, respectively.
The linear form $Q^{\jmath}$ can be extended naturally to a form $\mbf Q^{\jmath}$ on $\overline{\mbb F_q}^D$.
Suppose that $\mathcal D''$ is isotropic with respect to $\mbf Q^{\jmath}$.
We can fix a decomposition $\overline{\mbb F_q}^D = \D'' \oplus T \oplus W$ such that Lemma ~\ref{auxi} (a)-($\text{c}$) hold.
With respect to the bases in Lemma ~\ref{auxi} ($\text{c}$) and fix a basis for $T$, we can further assume that the associated matrix of $\mbf Q^{\jmath}$ is of the form
$
\begin{pmatrix}
0 & 0 & 1 \\
0 & 1 & 0\\
1 & 0 & 0
\end{pmatrix}
$,
since $\mbf Q^{\jmath}$ is defined over an algebraic closed field.
Recall that $\overline{\mbf G}_m = \mrm{GL}(1, \overline{\mbb F_q})$.
We define an imbedding $\overline{\mbf G}_m \to \overline{\mbf G}^{\jmath}_d$ by
$t \mapsto
\begin{pmatrix}
0 & 0 & t. 1\\
0 & 1 & 0\\
t^{-1}. 1& 0 & 0
\end{pmatrix}
$
where the $1$s denote the identity matrix of the desired rank.
Then the $\overline{\mbf G}_m$-fixed-point set of $\overline{\mbf X}^{\jmath}_d(\mbf a)$ consists of all flags $L$ such that
$L_i = (L_i \cap \D'' ) \oplus (L_i \cap T) \oplus (L_i \cap W)$ for all $i$, hence
is $\sqcup_{(\mbf a', \mbf a'')\vdash \mbf a} \overline{\mbf X}_{d'}^{\jmath} (\mbf a') \times \overline{\mbf X}_{d''} (\mbf a'')$.
Furthermore,  the attracting set of  $\overline{\mbf X}_{d'}^{\jmath} (\mbf a') \times \overline{\mbf X}_{d''} (\mbf a'')$, for all $(\mbf a', \mbf a'') \vdash \mbf a$,
is the algebraic variety $\overline{\mbf X}^{\jmath+}_{\mbf a, \mbf a', \mbf a''}$ whose $\mbb F_q$-point set  is
$X^{\jmath +}_{\mbf a, \mbf a', \mbf a''}
= \{ L \in X^{\jmath}_d | (\pi^{\natural} (L), \pi''(L)) \in  X_{d'}^{\jmath} (\mbf a') \times  X_{d''} (\mbf a'')\}$.
Thus we have a diagram
\[
\begin{CD}
\overline{\mbf X}^{\jmath}_d(\mbf a) @<\iota_{\jmath}<<  \overline{\mbf X}^{\jmath+}_{\mbf a, \mbf a', \mbf a''}
@> \pi^{\jmath} >>
\overline{ \mbf X}_{d'}^{\jmath} (\mbf a') \times \overline{\mbf X}_{d''} (\mbf a''),
\end{CD}
\]
where the first arrow is an inclusion and the second is induced by the definitions. Arguing in a similar way as Lemma
~\ref{Delta-ik}, we see that $\Delta^{\jmath}_{\mbf b', \mbf a', \mbf b'', \mbf a''}$ is the function version of the hyperbolic localization  functor  $\pi^{\jmath}_! \iota^*_{\jmath}$.
Now applying Braden's result ~\cite{B03}, we are done.
\end{proof}

\begin{rem}
Note that the rank of the vector bundle $\pi^{\jmath}$ is
$$\sum_{1 \leq i < j \leq n} a_i' a_j''  + \frac{1}{2} \left ( \sum_{i + j > n+1} a_i'' a_j''  - \sum_{i > r +1} a_i'' \right ). $$
This provides an explanation of the twist in (\ref{Delta-J}).
\end{rem}

Moreover, we have the following generic version of Proposition ~\ref{Phi-gene}.

\begin{prop}
\label{J-3}
There is  a unique  algebra homomorphism
\[
\phi^{\B}_{d, d-n}: \mbb S^{\B}_d \to \mbb S^{\B}_{d-n}
\]
such that $\mathcal A \otimes_{\mbb A} \phi^{\B}_{d, d-n} = \phi^{\B}_{d, d-n, \bv}$ and
\[
\phi^{\B}_{d, d-n} (\E_i ) =\E_i',\
\phi^{\B}_{d, d-n}(\F_i) =\F_i',\
\ \mbox{and}  \
\phi^{\B}_{d, d-n}(\K^{\pm 1}_i)=\K'^{\pm 1}_i,  \forall 1 \leq i \leq r.
\]
\end{prop}

By setting $d'=0$ in (\ref{Generic-D-B}), we have the following generic version of Corollary ~\ref{J-A}.

\begin{prop}
\label{J-4}
There is a unique algebra imbedding
\[
\jmath_d: \mbb S^{\B}_d \to \mbb S_{d}
\]
such that
$\mathcal A\otimes_{\mbb A} \jmath_d = \jmath_{d, \bv}$ and
\begin{align*}
\begin{split}
\jmath_d ( \E_i) & =  \mbf  E_i + \mbf K_i \mbf F_{n-i}, \\
\jmath_d (\F_i) & = \mbf F_i \mbf K_{n-i} + \mbf E_{n-i} , \\
\jmath_d (\mbf k_i ) & = v^{\delta_{i, r}} \mbf K_i \mbf K^{-1}_{n-i}, \quad \forall 1\leq i\leq r.\\
%
\end{split}
\end{align*}
\end{prop}

Suppose that $\{B\}_d$ is a canonical basis element in $\mbb S^{\B}_d$, we have
\begin{align}
\label{j-D-g}
\jmath_d (\{B\}_d) = \sum g_{B, A} \{A\}_d
\end{align}
where the sum runs over the set of canonical basis elements $\{ A\}_d$ in $\mbb S_d$.
By Propositions \ref{J-2} and \ref{J-4}, we have

\begin{cor}
\label{J-5}
$g_{B,A} \in \mbb Z_{\geq 0}[v, v^{-1}]$.
\end{cor}

\section{Positivity  for the modified coideal subalgebra  $\dot{\mbb U}^{\B}$}

\subsection{The coideal subalgebra $\mbb U^{\jmath}$}
\label{coideal}

By definition, $\mbb U^{\B}\equiv \mbb U^{\jmath} (\mathfrak{sl}_n)$ is an associative algebra over $\mbb Q(v)$ generated by
$\mbb e_i$, $\mbb f_i$, $\mbb k_{\pm i}$ for $1 \leq i \leq r$ and subject to the following defining relations.
For  any $1\leq i, j \leq r$ and $a_{ij} = 2 \delta_{i, j} - \delta_{i, j+1}  - \delta_{i, j -1} $,
\allowdisplaybreaks
\begin{eqnarray*}
\mbb k_i \mbb k_j   &= & \mbb k_j \mbb k_i,\\
\mbb k_i \mbb k^{-1}_{i} &= & \mbb k^{-1}_{i} \mbb K_i =1,\\
\mbb k_i \mbb e_j & = & v^{ a_{ij}  + \delta_{i,r}\delta_{j,r} } \mbb e_j \mbb k_i, \\
\mbb k_i \mbb f_j & = & v^{-a_{ij}    - \delta_{i,r}\delta_{j,r} } \mbb f_j \mbb k_i,\\
\mbb  e_i \mbb f_j - \mbb f_j \mbb e_i & =& \delta_{i, j} \frac{\mbb k_i - \mbb k^{-1}_{i} }{v-v^{-1}},  \quad \hspace{1.75cm} \,   {\rm if} \ (i, j)\neq (r,r),\\
\mbb e^2_r \mbb f_r + \mbb f_r \mbb e_r^2 & = & (v+v^{-1}) (\mbb e_r \mbb f_r \mbb e_r - \mbb e_r (v \mbb k_r +v^{-1} \mbb k^{-1}_{r} )),\\
\mbb f^2_r \mbb e_r + \mbb e_r \mbb f_r^2 & = & (v+v^{-1}) (\mbb f_r \mbb e_r \mbb f_r -(v \mbb k_r +v^{-1} \mbb k^{-1}_{r} ) \mbb f_r),\\
\mbb e_i \mbb e_j & = & \mbb e_j \mbb e_i,  \quad \hspace{3.2cm}  {\rm if}\ |i-j|>1,\\
\mbb f_i \mbb f_j & = & \mbb f_j \mbb f_i, \quad \hspace{3.25cm}  {\rm if}\ |i-j|>1,\\
\mbb e_i^2 \mbb e_j + \mbb e_j \mbb e_i^2 & =&  (v+v^{-1}) \mbb e_i \mbb e_j \mbb e_i ,\quad \hspace{1.22cm}  {\rm if}\ |i-j|=1,\\
\mbb f_i^2 \mbb f_j + \mbb f_j \mbb f_i^2 & =& (v+v^{-1}) \mbb f_i \mbb f_j \mbb f_i,\quad \hspace{1.3cm}  {\rm if}\ |i-j|=1.
\end{eqnarray*}

Recall the algebra $\mbb U$ from Section ~\ref{Type-A-positivity} from ~\cite[Proposition 4.5]{BKLW13}, see also ~\cite{BW13}, we have an injective algebra homomorphism
\begin{align*}
\jmath: \mbb U^{\jmath} \to \mbb U,
\end{align*}
defined by
\begin{align*}
\begin{split}
\jmath ( \mbb e_i)  =  \mbb  E_i + \mbb K_i \mbb F_{n-i}, \
\jmath (\mbb f_i)  = \mbb F_i \mbb K_{n-i} + \mbb E_{n-i} , \
\jmath (\mbb k_i )  = v^{\delta_{i, r}} \mbb K_i \mbb K^{-1}_{n-i}, \quad \forall 1\leq i\leq r.\\
%
\end{split}
\end{align*}
Here $n=2r+1$.
By composing $\jmath$ with $\Delta$ in (\ref{D-K}), we have an algebra homomorphism
$
\mbb \Delta^{\B} : \mbb U^{\jmath} \to \mbb U^{\jmath} \otimes \mbb U
$
defined by
\begin{align*}
\begin{split}
\mbb \Delta^{\B} ( \mbb e_i) & = \mbb e_i \otimes  \mbb K_{i} + 1 \otimes \mbb E_i   +  \mbb k_i  \otimes \mbb F_{n-i} \mbb K_i, \\
\mbb \Delta^{\B}  (\mbb f_i) & = \mbb f_i \otimes \mbb K_{n-i}  +  \mbb k^{-1}_{i} \otimes \mbb K_{n-i}  \mbb F_i  + 1 \otimes \mbb E_{n-i}, \\
\mbb \Delta^{\B} (\mbb k_i) & = \mbb k_i \otimes \mbb K_i \mbb K^{-1}_{n-i}, \quad \forall 1 \leq i \leq r.
\end{split}
\end{align*}

\subsection{The algebra $\dot{\mbb U}^{\jmath}$}
On $\mbb Z^n$, we define an equivalence relation ``$\approx$'' by $\mu \approx \lambda$ if and only if
$\mu - \lambda = m ( 2, \cdots, 2)$ for some $m\in \mbb Z$.
Let $\widehat \mu$ denote the equivalence class of $\mu$ with respect to $\approx$.
Consider the following subset in the set  $\mbb Z^n \! /\! \approx$ of equivalence classes.
\begin{align}
\label{Xj}
\mbb X^{\jmath} = \{\widehat  \mu \in \mbb Z^n \! / \! \approx  | \mu_i = \mu_{n +1 - i}, \forall 1\leq i\leq n, \mu_{r+1} \ \mbox{is odd}\}.
\end{align}
We define
\begin{align*}
\begin{split}
\dot{\mbb U}^{\jmath} & = \oplus_{\widehat \mu, \widehat  \lambda \in \mbb X^{\jmath}} \  _{\widehat \mu}\mbb U^{\jmath}_{ \widehat \lambda}, \\
_{\widehat \mu} \mbb U^{\jmath} _{\widehat \lambda}
&=  \mbb U^{\jmath}  / \left ( \sum_{1 \leq i \leq r}  (\mbb k_i - v^{-\mu_i + \mu_{i+1}} ) \mbb U^{\jmath}   +
\sum_{1\leq i \leq r} \mbb U^{\jmath} ( \mbb k_i - v^{- \lambda_i +  \lambda_{i+1}}) \right ).
\end{split}
\end{align*}
The algebra $\dot{\mbb U}^{\jmath}$ is the modified form of $\mbb U^{\jmath}$ (see ~\cite[4.6]{BKLW13} for the $\mathfrak{gl}_n$ version).
Let
\[
\pi_{\widehat \mu, \widehat \lambda} : \mbb U^{\jmath} \to \ _{\widehat \mu}\mbb U^{\jmath}_{ \widehat \lambda}
\]
be the natural projection.

Recall the set $\mbb X^{\jmath}$ from (\ref{Xj}) and $s_i$ the $i$-standard basis element of $\mbb Z^n$. We define an abelian group structure on $\mbb X^{\jmath}$ by
$
\widehat \mu + \widehat \lambda = \widehat{\pi},
\quad\mbox{with} \  \pi=\mu + \lambda - s_{r+1}.
$
We set
\[
I^{\jmath} =\{ 1, \cdots, r\}.
\]
The assignment
$
i \mapsto   -s_i + s_{i+1} + s_{n - i} - s_{n+1-i} + s_{r+1} , \quad \forall i\in I^{\jmath},
$
defines an embedding of abelian groups
$
\mbb Z[I^{\jmath}] \hookrightarrow \mbb X^{\jmath}.
$
We shall identify elements in $\mbb Z [I^{\jmath}]$ with their images in $\mbb X^{\jmath}$.
Then the algebra $\mbb U^{\jmath}$ in Section ~\ref{coideal} admits a $\mbb Z[I^{\jmath}]$-graded decomposition
$\mbb U^{\jmath} = \oplus_{\omega \in \mbb Z[I^{\jmath}]} \mbb U^{\jmath}(\omega)$ defined by
\[
\mbb e_i \in \mbb U^{\jmath} (i),\
\mbb f_i \in \mbb U^{\jmath} (-i),\
\mbb k^{\pm 1}_{i} \in \mbb U^{\jmath}(0),\
\mbb U^{\jmath} (\omega) \mbb U^{\jmath} (\omega') \subseteq
\mbb U^{\jmath}(\omega+\omega'),
\quad \forall i\in I^{\jmath}.
\]
Let $_{\widehat \mu} \mbb U^{\jmath}_{\widehat \lambda} (\omega) = \pi_{\widehat \mu, \widehat \lambda} ( \mbb U^{\jmath}(\omega) )$.
By a standard argument, we have

\begin{lem}
\label{j-w-0}
$_{\widehat \mu} \mbb U^{\jmath}_{\widehat \lambda} (\omega) =0$ unless $\widehat \mu - \widehat \lambda = \omega \in \mbb Z[I^{\jmath}] \subseteq \mbb X^{\jmath}$.
\end{lem}

\subsection{Positivity with respect to $\mbb \Delta^{\jmath}$}

We introduce the following notations to simplify the presentation.
For any $\mu, \mu', \mu'' \in \mbb Z^n$, we write
\begin{align}
\label{vdash}
(\mu', \mu'') \vdash \mu,
\end{align}
if and only if
$\mu_i' + \mu_i'' + \mu_{n+1-i}'' = \mu_i$, for all  $1\leq i \leq n$.
If $\mu' = s_{r+1}$, we simply write $\mu'' \vdash \mu$.

Assume that  $(\mu', \mu'') \vdash \mu$, then by definition we have
\begin{align*}
\begin{split}
\mDj  ( \mk_i - v^{-\mu_i+ \mu_{i+1} } ) = &
(\mk_i - v^{-\mu_i' + \mu_{i+1}'}) \otimes \mK_i \mK^{-1}_{n-i}
+
v^{-\mu_i' + \mu_{i+1}'}
\otimes (\mK_i - v^{-\mu_i'' + \mu_{i+1}''} )  \mK^{-1}_{n-i} \\
&
+ v^{-\mu_i'  -\mu_i''  + \mu_{i+1}' + \mu_{i+1}''}  \otimes (\mK^{-1}_{n-i} - v^{\mu_{n-i}'' - \mu_{n+1-i}'' }).
\end{split}
\end{align*}
This induces   a unique linear map
\begin{align*}
\mDj_{\widehat{\mu'}, \widehat{\lambda'}, \overline{\mu''}, \overline{\lambda''}} :
\ _{\widehat{\mu}} \mUj_{\widehat{\lambda}} \to \  _{\widehat{\mu'}} \mUj_{\widehat{\lambda'}}
\otimes \ \! _{\overline{\mu''}} \mU_{\overline{\lambda''}},
\quad  \forall (\mu', \mu'') \vdash \mu, (\lambda',\lambda'') \vdash \lambda
\end{align*}
such that  the following diagram commutes.
\begin{align}
\label{mDj-comm}
\begin{CD}
\mUj @> \mDj >> \mUj \otimes \mU \\
@VVV @VVV \\
_{\widehat{\mu}} \mUj_{\widehat{\lambda}}
@> \mDj_{\widehat{\mu'}, \widehat{\lambda'}, \overline{\mu''}, \overline{\lambda''}} >>
_{\widehat{\mu'}} \mUj_{\widehat{\lambda'}}
\otimes \ \! _{\overline{\mu''}} \mU_{\overline{\lambda''}}
\end{CD}
\end{align}

Recall  $\mbb B$ is the canonical basis for $\dot{\mU}$.
Let $\mbb B^{\jmath}$ be the canonical basis for $\dot{\mU}^{\jmath}$ defined in ~\cite{LW14}.

\begin{thm}
\label{J-conj-25.4.2}
Let $a\in \mBj$. If
$\mDj_{\widehat{\mu'}, \widehat{\lambda'}, \overline{\mu''}, \overline{\lambda''}} (a) =\sum_{b\in \mBj, c\in \mB} n^{b, c}_a b\otimes c$, then
$n^{b, c}_a \in \mbb Z_{\geq 0} [v, v^{-1}]$.
\end{thm}

The rest of this section is devoted to the proof of Theorem ~\ref{J-conj-25.4.2}.

For $\omega, \omega' \in \mbb Z[I^{\jmath}]$ and $\nu\in \mbb Z[I]$, we write
\begin{align}
\label{models}
(\omega', \nu) \models \omega,
\end{align}
if and only if
$\omega_i' + \nu_i - \nu_{n-i} = \omega_i,$ for all $1\leq i \leq r$.
If $\omega' =0$, we simply write $\nu \models \omega$.
By the definition of $\mDj$, we have

\begin{lem}
\label{mDj-w}
For any $\omega \in \mbb Z[I^{\jmath}]$,
$
\mDj ( \mbb U^{\jmath} (\omega) ) \subseteq
\oplus_{(\omega', \nu) \models \omega}
\mUj (\omega') \otimes \mbb U(\nu)
$.
\end{lem}

The following lemma is a refinement of (\ref{mDj-comm}), and  follows from
Lemmas ~\ref{mDj-w} and ~\ref{j-w-0}.

\begin{lem}
\label{mDj-U-U-dot}
Assume that $\widehat \mu, \widehat \lambda ,  \widehat{\mu'}, \widehat{\lambda'} \in \mbb X^{\jmath}$,
$\omega, \omega'  \in \mbb Z[I^{\jmath}]$,
$\overline{\mu''}, \overline{\lambda''} \in \mbb X$, and $\nu \in \mbb Z[I]$  such that
$
\widehat \mu - \widehat \lambda =\omega,
\widehat{\mu'} - \widehat{\lambda'}  =\omega',
\overline{\mu''} - \overline{\lambda''} = \nu,
(\mu', \mu'') \vdash \mu,
(\lambda', \lambda'') \vdash \lambda,
(\omega', \nu) \models \omega.
$
The following diagram commutes.
\begin{align*}
\begin{CD}
\mbb U^{\jmath} (\omega)  @>\mDj_{\omega, \nu}>> \mUj(\omega') \otimes  \mU(\nu)\\
@V\pi_{\widehat \mu, \widehat \lambda} VV @VV\pi_{\widehat{\mu'}, \widehat{\lambda'}} \otimes \pi_{\overline{\mu''}, \overline{\lambda''}}V\\
_{\widehat \mu} \mbb U^{\jmath}_{\widehat \lambda}
@>\mDj_{\widehat \mu, \widehat \lambda, \overline{\mu''}, \overline{\lambda''}}>>
_{\widehat{\mu'}} \mUj_{\widehat{\lambda'}}  \otimes  \ \!
_{\overline{\mu''}}  \mU_{\overline{\lambda''}}
\end{CD}
\end{align*}
\end{lem}

The assignment of sending generators of $\mbb U^{\jmath}$ to the respective generators of $_{\mbb Q(v)} \mbb S^{\jmath}_d$ defines an algebra homomorphism, denoted by
\begin{align}
\phi_d^{\jmath}: \mbb U^{\jmath} \to \ _{\mbb Q(v)} \mbb S^{\jmath}_d.
\end{align}
Moreover, this algebra homomorphism is compatible with the gradings. In particular,
\begin{align} \label{j-Phi-S}
\phi_d^{\jmath} (\mbb U^{\jmath} (\omega) )
\subseteq
\oplus_{\mbf b, \mbf a \in \Lambda^{\jmath}_{d, n}: \overline{\mbf b} -\overline{\mbf a} = \omega} \ _{\mbb Q(v)} \mbb S^{\jmath}_d(\mbf b, \mbf a),
\quad \forall \omega \in \mbb Z[I^{\jmath}].
\end{align}
On the other hand, we have
\begin{align}
\label{mDj-Sj-S}
\mDj  ( \ \! \! _{\mbb Q(v)} \mbb S^{\jmath}_d (\mbf b, \mbf a) ) \subseteq
\bigoplus_{
\substack{(\mbf b', \mbf b'') \vdash \mbf b \\ (\mbf a', \mbf a'')\vdash \mbf a}
}
\ _{\mbb Q(v)} \mbb S_{d'}^{\jmath}(\mbf b', \mbf a') \otimes
\!\! \ _{\mbb Q(v)} \mbb S_{d''}(\mbf b'', \mbf a'')
\end{align}
where $d= d' + d''$.
By Lemma  ~\ref{mDj-w},  (\ref{j-Phi-S}), and (\ref{mDj-Sj-S}), we have the following diagram.

\begin{lem}
\label{mDj-Delta-U-S}
Assume that $\mbf b, \mbf a \in \Lambda^{\jmath}_{d, n}$, $\omega, \omega' \in \mbb Z[I^{\jmath}]$,
$\nu\in \mbb Z[I]$ such that $(\omega', \nu) \models \omega$. The following diagram commutes.
\begin{align*}
\begin{CD}
\mUj (\omega) @> \mDj_{\omega', \nu} >> \mUj (\omega') \otimes  \mU(\nu) \\
@V\phi^{\jmath}_d VV @VV \phi^{\jmath}_{d'} \otimes \phi_{d''} V\\
\oplus_{
\substack{
\mbf b, \mbf a \in \Lambda^{\jmath}_{d, n}\\
\widehat{\mbf b} -\widehat{\mbf a} = \omega
}
}
\ \! _{\mbb Q(v)} \mbb S^{\jmath}_d (\mbf b, \mbf a)
@>\mDj >>
\oplus_{\substack{ (\mbf b', \mbf b'') \vdash \mbf b,
(\mbf a', \mbf a'') \vdash \mbf a\\
\widehat{\mbf b'} - \widehat{\mbf a'} = \omega',
\overline{\mbf b''} -\overline{\mbf a''} = \nu} }\
_{\mbb Q(v)} \mbb S^{\jmath}_{d'}  (\mbf b', \mbf a') \otimes \ \! \!
_{\mbb Q(v)} \mbb S_{d''} (\mbf b'', \mbf a'')
\end{CD}
\end{align*}
\end{lem}

Recall from ~\cite{LW14} that we have an algebra homomorphism
\begin{align}
\widetilde \phi^{\jmath}_d: \dot{\mbb U}^{\jmath} \to \ _{\mbb Q(v)} \mbb S^{\jmath}_d
\end{align}
defined by
\begin{align}
\begin{split}
\widetilde \phi^{\jmath}_d( 1_{\widehat \lambda}) & =
\begin{cases}
\zeta^{\jmath}_{M_{\mbf a}}, & \mbox{if} \ \widehat \lambda =\widehat{\mbf a}, \   \mbf a \in \Lambda^{\jmath}_{d, n},\\
0, & \mbox{o.w.}
\end{cases}\\
\widetilde \phi^{\jmath}_d( \mbb e_i 1_{\widehat \lambda}) & =
\begin{cases}
\mbf e_i \zeta^{\jmath}_{M_{\mbf a}}, & \mbox{if} \ \widehat \lambda =\widehat{\mbf a}, \ \mbf a \in \Lambda^{\jmath}_{d, n},\\
0, & \mbox{o.w.}
\end{cases} \
\widetilde \phi^{\jmath}_d( \mbb f_i 1_{\widehat \lambda})  =
\begin{cases}
\mbf f_i \zeta^{\jmath}_{M_{\mbf a}}, & \mbox{if} \ \widehat \lambda =\widehat{\mbf a}, \ \mbf a \in \Lambda^{\jmath}_{d, n},\\
0, & \mbox{o.w.}
\end{cases}
\end{split}
\end{align}
By restricting to $_{\widehat \mu} \mbb U^{\jmath}_{\widehat \lambda}$, it induces a linear map
\[
\widetilde \phi^{\jmath}_d :
 \ _{\widehat \mu} \mbb U^{\jmath}_{\widehat \lambda}
\to
\ _{\mbb Q(v)} \mbb S^{\jmath}_d(\mbf b, \mbf a),
\ \mbox{if}\
\widehat \mu =\widehat{\mbf b}, \widehat \lambda =\widehat{\mbf a}.
\]
In particular, we have the following lemma.

\begin{lem}
Assume that
$\mbf b^0, \mbf a^0 \in \Lambda^{\jmath}_{d, n}$ such that
$\widehat{\mbf b^0} -\widehat{\mbf a^0} = \omega \in \mbb Z[I^{\jmath}]$.
The following diagram is commutative.
\begin{align}
\label{j-U-U-dot-S}
\begin{CD}
\mbb U^{\jmath}(\omega) @> \pi_{\widehat{\mbf b^0},\widehat{\mbf a^0}} >> _{\widehat{\mbf b^0}} \mbb U^{\jmath}_{\widehat {\mbf a^0}} \\
@V\widetilde \phi^{\jmath}_d VV @VV\widetilde \phi^{\jmath}_dV\\
\oplus_{\widehat{\mbf b} -\widehat{\mbf a} =\omega} \ _{\mbb Q(v)} \mbb S^{\jmath}_d (\mbf b, \mbf a)
@>>>
_{\mbb Q(v)} \mbb S^{\jmath}_d (\mbf b^0, \mbf a^0)
\end{CD}
\end{align}
where the arrow in the bottom is the natural projection.
\end{lem}

By putting together Lemmas \ref{mDj-Delta-U-S}, \ref{mDj-U-U-dot} and (\ref{j-U-U-dot-S}),  we have the following cube.
\begin{align}
\label{mDj-cube}
\begin{split}
\xymatrixrowsep{.1in}
\xymatrixcolsep{.1in}
\xymatrix{
& _{\widehat \mu}\mbb U^{\jmath}_{\widehat \lambda}  \ar@{->}[rr]  \ar@{.>}[dd]
&& _{\widehat{\mu'}}\mUj_{\widehat{\lambda'}} \otimes \ \! \!  _{\overline{\mu''}}\mbb U_{\overline{\lambda''}}
\ar@{->}[dd] \\
\mbb U^{\jmath} (\omega) \ar@{->}[ur]^{\pi_{\widehat \mu, \widehat \lambda}}  \ar@{->}[rr]  \ar@{->}[dd]
& & \mUj (\omega') \otimes \mU(\nu)  \ar@{->}[ur] \ar@{->}[dd] & \\
& \mbb S_d^{\jmath} (\mbf b, \mbf a)  \ar@{.>}[rr]  && \mbb S^{\jmath}_{d'}(\mbf b', \mbf a') \otimes \mbb S_{d''} (\mbf b'', \mbf a'') \\
\oplus  \mbb S_d^{\jmath} (\mbf b, \mbf a) \ar@{.>}[ur] \ar@{->}[rr]  && \oplus   \mbb S^{\jmath}_{d'}(\mbf b', \mbf a') \otimes \mbb S_{d''} (\mbf b'', \mbf a'') \ar@{->}[ur] &
}
\end{split}
\end{align}
where the sum on the bottom left is over all $\mbf b, \mbf a \in \Lambda^{\jmath}_{d, n}$ such that
$\widehat{\mbf b} -\widehat{\mbf a} = \omega$, while the sum on the bottom right is over
$\mbf b', \mbf a' \in \Lambda^{\jmath}_{d', n}$ and $\mbf b'', \mbf a'' \in \Lambda_{d'', n}$ such that
$(\mbf b', \mbf b'' ) \vdash \mbf b$, $(\mbf a', \mbf a'') \vdash \mbf a$, $\widehat{\mbf b'} -\widehat{\mbf a'} =\omega$,
and $\overline{\mbf b''} -\overline{\mbf a''} = \nu$.

From (\ref{mDj-cube}) and the surjectivity of $\pi_{\widehat\mu, \widehat\lambda}$,  we have the following proposition.

\begin{prop}
\label{mDj-U-dot-S-A}
The square in the back of (\ref{mDj-cube}) is commutative.
\begin{align*}
\begin{split}
\xymatrix{
& & _{\widehat \mu}\mbb U^{\jmath}_{\widehat \lambda}
\ar@{->}[rr]   \ar@{->}[d]
&&  _{\widehat{\mu'}}\mUj_{\widehat{\lambda'}} \otimes \ \! \!  _{\overline{\mu''}}\mbb U_{\overline{\lambda''}}
\ar@{->}[d] \\
&& \mbb S_d^{\jmath} (\mbf b, \mbf a)  \ar@{->}[rr]  &&
\mbb S^{\jmath}_{d'}(\mbf b', \mbf a') \otimes \mbb S_{d''} (\mbf b'', \mbf a'')  \\
}
\end{split}
\end{align*}
\end{prop}

By using Proposition ~\ref{mDj-U-dot-S-A}, we deduce
Theorem \ref{J-conj-25.4.2} via a similar argument for Theorem ~\ref{conj-25.4.2}.

\subsection{Positivity with respect to $\jmath$}

Notice that for $\mu, \mu'' \in \mbb Z^n$  such that $\mu'' \vdash \mu$,
we have
\[
\jmath ((\mbb k_i - v^{-\mu_i + \mu_{i+1}} ) \mbb U^{\jmath}  ) \subseteq
\sum_{1\leq i\leq n} (\mbb K_i - v^{-\mu_i'' + \mu_{i+1}'' }) \mbb U.
\]
Similarly, for any $\lambda, \lambda'' \in \mbb Z^n$  such that $\lambda'' \vdash \lambda$
 we have
\[
\jmath (\mbb U^{\jmath}  (\mbb k_i - v^{-\lambda_i + \lambda_{i+1}} )  ) \subseteq
\sum_{1\leq i\leq n}  \mbb U
( \mbb K_i - v^{-\lambda_i'' + \lambda_{i+1}'' }).
\]
The above  observations  induce a linear map
\begin{align}
\label{U-jmath}
\jmath_{\widehat  \mu, \widehat  \lambda, \overline{\mu''}, \overline{\lambda''}}:
\ _{\widehat  \mu} \mbb U^{\jmath} _{\widehat \lambda}
\longrightarrow
_{\overline{\mu''}} \mbb U  _{\overline{ \lambda''}},
\quad  \forall \mu'' \vdash \mu, \lambda'' \vdash \lambda
\end{align}
such that the following diagram commutes.
\begin{align}
\label{j-comm}
\begin{CD}
\mbb U^{\jmath} @>\jmath >> \mbb U \\
@V\pi_{\widehat \mu, \widehat \lambda} VV @VV\pi_{\overline{\mu''}, \overline{\lambda''}} V\\
_{\widehat \mu} \mbb U^{\jmath} _{\widehat \lambda}
@> \jmath_{\widehat \mu, \widehat \lambda, \overline{\mu''}, \overline{\lambda''}} >>
_{\overline{\mu''}} \mbb U  _{\overline{ \lambda''}}.
\end{CD}
\end{align}


\begin{thm}
\label{j-positivity}
Let $b \in \mBj$. If
$\jmath_{\widehat \mu, \widehat  \lambda, \overline{\mu''}, \overline{\lambda''}}  (b)
=\sum_{a\in \mB}  g_{b, a} a$, then
$g_{b, a} \in \mbb Z_{\geq 0}[v, v^{-1}]$.
\end{thm}

The proof of Theorem ~\ref{j-positivity} is a degenerate version of the proof of Theorem ~\ref{J-conj-25.4.2}.
For the sake of completeness, we provide it here.
The following lemma is due to the fat that
\[
\jmath (\mbb e_i) \in \mbb U(i) + \mbb U (-(n-i)),\
\jmath (\mbb f_i) \in  \mbb U(-i) + \mbb U(n-i),\
\jmath (\mbb k^{\pm 1}_{i}) \in \mbb U(0).
\]

\begin{lem}
\label{j-w}
For any $\omega \in \mbb Z[I^{\jmath}]$,
$
\jmath ( \mbb U^{\jmath} (\omega) ) \subseteq
\oplus_{\nu \models \omega}
\mbb U(\nu).
$
\end{lem}

From Lemmas ~\ref{j-w} and ~\ref{j-w-0}, we have
the following refinement of (\ref{j-comm}).

\begin{lem}
\label{j-U-U-dot}
Assume that $\widehat \mu, \widehat \lambda \in \mbb X^{\jmath}$,  $\omega \in \mbb Z[I^{\jmath}]$,
$\overline{\mu''}, \overline{\lambda''} \in \mbb X$, and $\nu \in \mbb Z[I]$  such that
\[
\widehat \mu - \widehat \lambda =\omega,
\overline{\mu''} - \overline{\lambda''} = \nu,
\mu'' \vdash \mu, \lambda'' \vdash \lambda, \nu \models \omega .
\]
The following diagram commutes.
\begin{align*}
\begin{CD}
\mbb U^{\jmath} (\omega)  @>\jmath_{\omega, \nu}>> \mbb U(\nu)\\
@V\pi_{\widehat \mu, \widehat \lambda} VV @VV\pi_{\overline{\mu''}, \overline{\lambda''}}V\\
_{\widehat \mu} \mbb U^{\jmath}_{\widehat \lambda}
@>\jmath_{\widehat \mu, \widehat \lambda, \overline{\mu''}, \overline{\lambda''}}>> _{\overline{\mu''}}
\mbb U_{\overline{\lambda''}}
\end{CD}
\end{align*}
where $\jmath_{\omega, \nu}$ is the one induced from $\jmath$ by restricting to $\mbb U^{\jmath}(\omega)$ and projecting down to $\mbb U(\nu)$.
\end{lem}

Note that  we have
\begin{align}
\label{j-Sj-S}
\jmath_d ( \ \! _{\mbb Q(v)} \mbb S^{\jmath}_d (\mbf b, \mbf a) ) \subseteq
\bigoplus_{
\substack{
\mbf b'' \vdash \mbf b,
\mbf a'' \vdash \mbf a
}}
\ _{\mbb Q(v)} \mbb S_d(\mbf b'', \mbf a''),
\end{align}
By Lemma  ~\ref{j-w},  (\ref{j-Phi-S}) and (\ref{j-Sj-S}), we have the following commutative diagram.

\begin{lem}
\label{j-Delta-U-S}
Assume that $\mbf b, \mbf a \in \Lambda^{\jmath}_{d, n}$, $\omega \in \mbb Z[I^{\jmath}]$,
$\nu\in \mbb Z[I]$ such that
$\nu_i - \nu_{n-i} = \omega_i$ for any $i\in I^{\jmath}$. The following diagram commutes.
\begin{align*}
\begin{CD}
\mbb U^{\jmath} (\omega) @> \jmath_{\omega, \nu} >> \mbb U(\nu) \\
@V\phi^{\jmath}_d VV @VV \phi_d V\\
\oplus_{\substack{
\mbf b, \mbf a \in \Lambda^{\jmath}_{d, n}\\
\widehat{\mbf b} -\widehat{\mbf a} = \omega}}
\ _{\mbb Q(v)} \mbb S^{\jmath}_d (\mbf b, \mbf a) @>\jmath_d>>
\oplus_{\substack{\mbf b'', \mbf a'' \in \Lambda_{d, n}: (\star) }} \
_{\mbb Q(v)} \mbb S_d(\mbf b'', \mbf a'')
\end{CD}
\end{align*}
where the condition ($\star$) is
$\overline{\mbf b''}-\overline{\mbf a''} =\nu$,
$\mbf b'' \vdash \mbf b$ and $\mbf a'' \vdash \mbf a$.
\end{lem}

By putting together Lemmas \ref{j-Delta-U-S}, \ref{j-U-U-dot} and (\ref{j-U-U-dot-S}),  we have the following cube.
\begin{align}
\label{j-cube}
\begin{split}
\xymatrixrowsep{.1in}
\xymatrixcolsep{.1in}
\xymatrix{
 & _{\widehat \mu}\mbb U^{\jmath}_{\widehat \lambda}  \ar@{->}[rr]  \ar@{.>}[dd]
&&  _{\overline{\mu''}}\mbb U_{\overline{\lambda''}}
\ar@{->}[dd] \\
 \mbb U^{\jmath} (\omega) \ar@{->}[ur]^{\pi_{\widehat \mu, \widehat \lambda}}  \ar@{->}[rr]  \ar@{->}[dd]
& & \mbb U(\nu)  \ar@{->}[ur] \ar@{->}[dd] & \\
& \mbb S_d^{\jmath} (\mbf b, \mbf a)  \ar@{.>}[rr]  &&\mbb S_d (\mbf b'', \mbf a'') \\
\oplus_{\widehat{\mbf b}-\widehat{\mbf a} =\omega}  \mbb S_d^{\jmath} (\mbf b, \mbf a) \ar@{.>}[ur] \ar@{->}[rr]  && \oplus_{\overline{\mbf b''} -\overline{\mbf  a''} =\nu''}  \mbb S_d (\mbf b'', \mbf a'') \ar@{->}[ur] &
}
\end{split}
\end{align}

From (\ref{j-cube}) and the surjectivity of $\pi_{\widehat\mu, \widehat\lambda}$,  we have the following proposition.

\begin{prop}
\label{j-U-dot-S-A}
The square in the back of (\ref{j-cube}) is commutative.
\begin{align*}
\begin{split}
\xymatrix{
& & _{\widehat \mu}\mbb U^{\jmath}_{\widehat \lambda}
\ar@{->}[rr]^{\jmath_{\widehat \mu, \widehat \lambda,  \overline{\mu''}, \overline{\lambda''}}}   \ar@{->}[d]
&&  _{\overline{\mu''}}\mbb U_{\overline{\lambda''}}
\ar@{->}[d] \\
&& \mbb S_d^{\jmath} (\mbf b, \mbf a)  \ar@{->}[rr]^{\jmath_d}  &&\mbb S_d (\mbf b'', \mbf a'') \\
}
\end{split}
\end{align*}
\end{prop}

Recall that  for any $b\in \mbb B^{\jmath}$, we suppose that
\[
\jmath_{\widehat \mu, \widehat \lambda,  \overline{\mu''}, \overline{\lambda''}}
 (b) = \sum_{a \in \mbb  B} g_{b, a} a,
\]
where $g_{b, a} \in \mbb Z[ v, v^{-1}]$ is zero except for finitely many terms.
Let $S = \{ a | g_{b, a} \neq 0\}$.
Since the set $S$ is finite, we can find a large enough $d$ using ~\cite{LW14} and ~\cite{M10} such that
\[
\phi^{\jmath}_d (b)=\{ B\}_d,
\phi_d ( a) =\{A\}_d,
\quad \forall a\in S
\]
where $\{B\}_d$ and $\{A\}_d$ are certain  canonical basis elements
in $\mbb S^{\B}_d$ and $\mbb S_d$, respectively.
Applying $\phi^{\jmath}_d$ and Lemma ~\ref{j-U-dot-S-A}, we have
\[
\jmath_d ( \{B\}_d)
= \jmath_d (\phi^{\jmath}_d (b) )
= \phi_d \jmath_{\widehat \mu, \widehat \lambda,  \overline{\mu''}, \overline{\lambda''}} (b)
= \sum_{a \in \mbb  B} g_{b, a} \phi_d( a)
=\sum_{a \in \mbb B} g_{b, a} \{A\}_d.
\]
By comparing  the above with (\ref{j-D-g}), we have $g_{b, a} = g_{B,A}$.
So we have $g_{b, a} \in \mbb Z_{\geq 0} [v, v^{-1}]$ by Corollary ~\ref{J-5}. Theorem \ref{j-positivity} follows.

\subsection{The imbedding $\tilde \jmath$}

For any pair $(\widehat \mu, \widehat \lambda) $ in $\mbb X^{\jmath}$, we define
\begin{align*}
 \jmath_{\widehat \mu, \widehat \lambda}\equiv \prod
 \jmath_{\widehat \mu, \widehat \lambda, \overline{\mu''}, \overline{\lambda''}}:
\  _{\widehat \mu} \mbb U^{\jmath} _{\widehat \lambda}
 \longrightarrow
 \prod \ _{\overline{\mu''}} \mbb U  _{\overline{ \lambda''}},
\end{align*}
where  the product runs over all $\overline{\mu''}, \overline{\lambda''}$ in $\mbb X$ such that
$\mu'' \vdash \mu$ and $\lambda'' \vdash \lambda$.
We set
\begin{align*}
\widetilde  \jmath \equiv \bigoplus_{\widehat \mu, \widehat \lambda \in \mbb X^{\jmath}}  \jmath_{\widehat \mu, \widehat \lambda}:
\dot{\mbb U}^{\jmath} \longrightarrow
\bigoplus_{\widehat \mu, \widehat \lambda \in \mbb X^{\jmath}}  \prod \ _{\overline{\mu''}}
\mbb U  _{\overline{ \lambda''}}.
\end{align*}

\begin{prop}
$\widetilde \jmath$ is injective.
\end{prop}

\begin{proof}
It suffices to show that for any  nonzero element $x$ in $_{\widehat \mu} \mbb U^{\jmath} _{\widehat \lambda} $,
there is $\overline{\mu''}$ and $\overline {\lambda''}$ such that
$ \jmath_{\widehat \mu, \widehat \lambda, \overline{\mu''}, \overline{\lambda''}} (x)$ is nonzero.
Suppose that $\widehat \mu - \widehat \lambda =\omega$. Let us pick an element $u \in \mbb U^{\jmath}(\omega)$ such that $\pi_{\widehat \mu, \widehat \lambda} (u) = x$.
Since $\jmath $ is injective, we have $\jmath (u) \neq 0 \in \oplus_{\nu\in \mbb Z[I]} \mbb U(\nu)$.
Thus there is $\nu$ such that the $\nu$-component $\jmath(u)_{\nu}$ of $\jmath (u)$ is nonzero.
It is well-known (see \cite{L00}) that we can then find a large enough $d$ such that $\phi_d ( \jmath(u)_{\nu}) \neq 0$.
In particular, there is a pair $\mbf b''$, $\mbf a''$ in $\Lambda_{d,n}$ such that
the $(\mbf b'', \mbf a'')$-component of $\phi_d(\jmath(u)_{\nu})$ is nonzero. Take $\overline{\mu''}=\overline{\mbf b''}$ and $\overline{\lambda''} =\overline{\mbf a''}$. By chasing along the cube (\ref{j-cube}), we see immediately that
$\jmath_{\widehat \mu, \widehat \lambda, \overline{\mu''}, \overline{\lambda''}} (x) \neq 0$.
\end{proof}

\begin{rem}
$\widetilde \jmath$ can be regarded as an idempotented version of $\jmath$.
\end{rem}

\section{$\imath$-version}

\label{i-version}

In this section, we show the positivity of the  i-canonical basis of the modified coideal subalgebra of quantum $\mathfrak{sl}_{\ell}$ for $\ell$ even.
Since the arguments are more or less the same as the $n$ odd situation, the presentation will be brief.

\subsection{$\imath$-Schur algebras and related results}

Recall   $n=2r+1$ and $D=2d+1$. We set
\[
\ell = n-1.
\]

Recall $\Xi^{\jmath}_d$ from (\ref{Xij}). Let
$\Xi^{\imath}_d =\{ A\in \Xi^{\jmath}_d | a_{r+1, j} = \delta_{j, r+1}, a_{i, r+1} = \delta_{i, r+1} \}$.
Let  $\bj=\sum [A]_d$ where the sum runs over all diagonal matrices  in $\Xi^{\imath}_d$.
Let $\cSi = \bj \Sj \bj$. It is a subalgebra in $\Sj$ and admits a basis $[A]_d$ for all $A \in \Xi^{\imath}_d$.
In particular, $\cSi$ contains the following elements.
\begin{align}
\begin{split}
\check \E_{i, d} & = \bj  \E_i  \bj, \quad
\check \F_{i, d}  = \bj \F_i \bj, \quad
\check \K_{i, d}  = \bj \K_i \bj, \quad \forall i\in [1, r-1],\quad
\check \bh_{a, d}  = \bj \bh_a \bj, \quad \forall a \in [1, r],\\
\check \bt_d &  =\bj ( \F_r \E_r + \frac{\K_r - \K_r^{-1}}{v-v^{-1}} )\bj.
\end{split}
\end{align}

Similarly, we consider the subset $\Xi_{d, \ell}$ of $\Xi_d$ defined by the condition $a_{r+1, j}=0$ and $a_{i, r+1} =0$ for all $i, j$.
Let $\bJ=\sum [A]_d$  where the sum runs over all diagonal matrices $A \in \Xi_{d,\ell}$.
Then $\mbb S_{d, \ell} = \mbf J \mbb S_d \mbf J$ is a subalgebra of $\mbb S_d$ with a basis $[A]_d$ indexed by $\Xi_{d, \ell}$.
$\mbb S_{d, \ell}$ contains the following.
\begin{align*}
\check \bE_{i, d} & =
\begin{cases}
\bJ \bE_i \bJ, & \mbox{if} \  i \in [1, r-1],\\
\bJ \bE_{r+1} \bE_{r} \bJ, &\mbox{if} \ i = r, \\
\bJ \bE_{i+1} \bJ, & \mbox{if} \ i \in [r+1,  \ell -1 ].
\end{cases}
\quad
\check \bF_{i, d} =
\begin{cases}
\bJ \bF_i \bJ, & \mbox{if} \  i \in [1, r-1],\\
\bJ \bF_r \bF_{r+1} \bJ, &\mbox{if} \ i = r, \\
\bJ \bF_{i+1} \bJ, & \mbox{if} \ i \in [r+1, \ell -1].
\end{cases}
\\
\check \bK_{i, d} & =
\begin{cases}
\bJ \bK_i \bJ,  & \mbox{if} \  i\in [1, r-1],\\
\bJ \bK_r \bK_{r+1} \bJ, & \mbox{if} \ i = r, \\
\bJ \bK_{i+1} \bJ , & \mbox{if} \ i \in [r+1, \ell -1].
\end{cases}
\quad
\check \bH_{a,d}  =
\begin{cases}
\bJ \bH_a \bJ, & \mbox{if} \  a \in  [1, r],\\
\bJ \bH_{a+1} \bJ, & \mbox{if} \ a \in [r+1, \ell].
\end{cases}
\end{align*}

Notice that we have $\bJ \bH_{r+1} \bJ =1$ and $\check \bK_{r, d} = \check \bH^{-1}_{r, d} \check \bH_{r+1, d}$.

\begin{lem}
\label{tDelta}
Let $d' + d'' =d$.
$\tDj (\cSi) \subseteq \mbb S^{\imath}_{d', \ell } \otimes \mbb S_{d'', \ell}$.
Moreover, for all  $i \in [1, r-1]$.
\begin{align}
\begin{split}
\tDj ( \check \E_{i, d}) &
= \check \E_{i, d'} \otimes \check \bH_{i+1, d''} \check \bH^{-1}_{\ell -i, d''}
+ \check \bh^{-1}_{i+1, d'} \otimes \check \bE_{i, d''} \check \bH^{-1}_{\ell -i, d''}
+ \check \bh_{i+1, d'} \otimes \check \bF_{\ell -i, d''} \check \bH_{i+1, d''}.\\
\tDj (\check \F_{i, d} )
& = \check \F_{i, d'} \otimes \check \bH^{-1}_{i, d''} \check \bH_{\ell +1 -i, d''}
+ \check \bh_{i, d'} \otimes \check \bF_{i, d''} \check \bH_{\ell +1-i, d''}
+ \check \bh^{-1}_{i, d'} \otimes \check \bE_{\ell -i, d''} \check \bH^{-1}_{i, d''}.\\
\tDj (\check \K_{i, d} ) & = \check \K_{i, d'} \otimes \check \bK_{i, d''} \check \bK^{-1}_{\ell -i, d''}.\\
\tDj (\check \bt_d)
& =  \check \bt_{d'} \otimes \check \bK_{r, d''}
+ v^2 \check \K^{-1}_{r, d'} \otimes \check \bH_{r+1, d''} \check \bF_{r, d''}
+ v^{-2} \check \K_{r, d'} \otimes \check \bH^{-1}_{r, d''} \check \bE_{r, d''}.
\end{split}
\end{align}
\end{lem}

\begin{proof}
For convenience, we shall drop the subscript $d$ and replace $d', d''$ by superscript $'$ and $''$ respectively in the proof.
The first three equalities are from definitions and $\widetilde \Delta^{\jmath} (\mbf j) = \mbf j' \otimes \mbf J''$. We now show the last one.
By using $ \bj \F_r \bj =0$ and $\bj \E_r \bj=0$, we have
\begin{align*}
\begin{split}
\tDj (\bj \F_r \E_r \bj )
& = \bj \F_r \E_r \bj \otimes \bJ \bH^{-1}_r \bH_{r+2} \bJ
+ \bj \bh_r \bh^{-1}_{r+1} \bj \otimes \bJ \bF_r \bH_{r+2} \bE_r \bH^{-1}_{r+1} \bJ \\
&
+ \bj \bh_r \bh_{r+1} \bj \otimes \bJ \bF_r \bH_{r+2} \bF_{r+1} \bH_{r+1} \bJ
+ \bj \bh^{-1}_r \bh_{r+1} \bj \otimes \bJ \bE_{r+1}  \bH^{-1}_r \bF_{r+1} \bH_{r+1} \bJ \\
&+ \bj \bh^{-1}_r \bh^{-1}_{r+1} \bj \otimes \bJ \bE_{r+1} \bH^{-1}_r \bE_r \bH^{-1}_{r+1} \bJ.
\end{split}
\end{align*}
We observe that
$
\bj \bh^{-1}_r \bh_{r+1} \bj = \check \bk_r
$
and
$
\bj \bh_r \bh_{r+1} \bj = v^2 \check \bk^{-1}_r .
$
We further  observe that
\begin{align*}
\begin{split}
\bJ \bF_r \bH_{r+2} \bE_r \bH^{-1}_{r+1} \bJ & = \check \bH_{r+1} \frac{\check \bH_r - \check \bH^{-1}_r}{v-v^{-1}}, \\
\bJ \bF_r \bH_{r+2} \bF_{r+1} \bH_{r+1} \bJ & = \check \bH_{r+1} \check \bF_r, \\
\bJ \bE_{r+1} \bH^{-1}_r \bF_{r+1} \bH_{r+1} \bJ &= \check \bH^{-1}_r \frac{\check \bH_{r+1} - \check \bH^{-1}_{r+1}}{v-v^{-1}}, \\
\bJ \bE_{r+1} \bH^{-1}_r \bE_r \bH^{-1}_{r+1} \bJ &= \check \bH^{-1}_r \check \bE_r.
\end{split}
\end{align*}
So we have
\begin{align*}
\tDj (\check \bt ) & = \check \bt' \otimes \check \bK''_r
+ v^2 \check \K'^{-1}_r \otimes \check \bH''_{r+1} \check \bF''_r
+ v^{-2} \check \K'_r \otimes \check \bH''^{-1}_r \check \bE''_r + R,
\end{align*}
where the remainder  $R$ is equal to
\begin{align*}
\begin{split}
& - \frac{\check \bk_r - \check \bk^{-1}_r}{v-v^{-1}}  \otimes \check \bK_r
+ \check \bk^{-1}_r \otimes \check \bH_{r+1} \frac{\check \bH_r - \check \bH^{-1}_r}{v-v^{-1}}
+ \check \bk_r \otimes \check \bH^{-1}_r   \frac{\check \bH_{r+1} - \check \bH^{-1}_{r+1}}{v-v^{-1}}  \\
& +\frac{ \check \bk_r \otimes \check \bH_r \check \bH^{-1}_{r+1} - \check \bk^{-1}_r \otimes \check \bH_r \check \bH_{r+1}}{v-v^{-1}}.
\end{split}
\end{align*}
We combine the terms with $\check \bk_r$ together and we get zero.
So is the case when we combine the terms with $\check \bk^{-1}_r$.
Hence $R$ is zero. Therefore, we have the last equality in the lemma.
\end{proof}

We define the transfer map
\begin{align}
\cphi : \cSi \to \cSii
\end{align}
to be the composition
$
\cSi \overset{\tDj}{\longrightarrow} \cSii \otimes \mbb S_{\ell, \ell}  \overset{1\times \chi_{\ell}}{\longrightarrow} \cSii,
$
where $\chi_{\ell}: \mbb S_{\ell, \ell} \to \mbb A$ is the signed representation.
By Lemma ~\ref{tDelta}, we have

\begin{lem}
\label{i-transfer}
$\cphi (\check{\E}_{i, d}) = \check{\E}_{i, d-\ell}$, $\cphi (\check{\F}_{i, d}) =\check{\F}_{i, d-\ell}$, $\cphi (\K^{\pm 1}_{i, d}) = \K^{\pm 1}_{i, d-\ell}$
and $\cphi (\check \bt_d) = \check \bt_{d-\ell}$ for all $i\in [1, r-1]$.
\end{lem}

Now we handle the case of $\mbb \Dj$.

\begin{prop}
\label{Dj-cSi}
For $i\in [1, r-1]$,
\begin{align}
\begin{split}
\mbb \Dj (\check \E_{i, d} ) & =\check \E_{i, d'} \otimes \check \bK_{i, d''} + 1 \otimes \check \bE_{i, d''} + \check  \bk_{i, d'} \otimes \check \bF_{\ell -i, d''} \check  \bK_{i, d''}. \\
\mbb \Dj (\check \F_{i, d} ) & = \check \F_{i, d'} \otimes \check \bK_{\ell -i, d''} + \check \bk^{-1}_{i, d'} \otimes \check \bK_{\ell -i, d''} \check \bF_{i, d''} + 1 \otimes \check \bE_{\ell -i, d''}.
\\
\mbb \Dj (\check \K_{i, d} ) & = \check \K_{i, d'} \otimes \check \bK_{i, d''} \check \bK^{-1}_{\ell - i, d''}. \\
\mbb \Dj (\check \bt_d) & = \check \bt_{d'} \otimes \check \bK_{r, d''} + 1 \otimes v \check \bK_{r, d''} \check \bF_{r, d''} + 1 \otimes \check \bE_{r, d''}.
\end{split}
\end{align}
\end{prop}

\begin{proof}
Again only the last equality is nontrivial, and we drop the subscript $d$ and $d', d''$ are replaced by $'$ and $''$, respectively.
Suppose that we have a quadruple $(\mbf b', \mbf a', \mbf b'', \mbf a'')$ such that $ b'_k = a'_k$ and $b''_k = a''_k$ for all $k$, then the twists $\sum_{1 \leq k \leq j \leq n} b'_k b''_j - a'_k a''_j$ and $u(\mbf b'', \mbf a'')$ are zero.
Hence we have the first term $\check \bt \otimes \check \bK_r$ after the twist.

Suppose that we have a quadruple $(\mbf b', \mbf a', \mbf b'', \mbf a'')$ such that
$b'_k = a'_k$ and $b''_k = a''_k + \delta_{k, r} - \delta_{k, r+2}$, then we have
$\sum_{1 \leq k \leq j \leq n} b'_k b''_j - a'_k a''_j= - (a'_{r+2} +1)$ and $u(\mbf b'', \mbf a'') =  - a''_r$.
So after the twist, we have
$
v^2 \check \K'^{-1}_r \otimes \check \bH''_{r+1} \check \bF''_r |_{\mbf b', \mbf a', \mbf b'', \mbf a''}  v^{-(a'_{r+2} +1) - a''_r}
= 1 \otimes v \check \bK_r \check \bF_r |_{\mbf b', \mbf a', \mbf b'', \mbf a''}
$,
hence we have the second term.

Suppose that we have a quadruple $(\mbf b', \mbf a', \mbf b'', \mbf a'')$ such that
$b'_k = a'_k$ and $b''_k = a''_k - \delta_{k, r} + \delta_{k, r+2}$. Then we have
$\sum_{1 \leq k \leq j \leq n} b'_k b''_j - a'_k a''_j= a'_{r+2} +1$ and $u(\mbf b'', \mbf a'') = a''_r -1$.
Thus, adding the twists, we have
$ v^{-2} \check \K'_r \otimes \check \bH''^{-1}_r \check \bE''_r|_{\mbf b', \mbf a', \mbf b'', \mbf a''} v^{a'_{r+2} +1 + a''_r -1} =  1 \otimes \check \bE_r |_{\mbf b', \mbf a', \mbf b'', \mbf a''}$.
Whence we obtain the third term.

By the above analysis, we have the last equality. The proposition is proved.
\end{proof}

Now we take care of the degenerate version when $d'=0$ and $d'' = d$.
In this case, $\Dj$ degenerates to an algebra homomorphism
\[
 \imath_d: \cSi \to  \cbS,
\]
since $\mbb S^{\imath}_{0, n-1} \simeq \mathcal A$.
Observe that $\check \E_{i, 0} =0$, $\check \F_{i, 0} =0$ and $\check \K_{i, 0} = v^{\delta_{i, r}}$, for all $i\in [1, r]$, and $\bt_{0}=1$ in $\mbb S^{\imath}_{0, n-1}$
from which the statements in  Proposition ~\ref{Dj-cSi} now read as follows.

\begin{cor}
For all $i\in [1, r-1]$,
\begin{align}
\begin{split}
\imath_d (\check \E_{i, d}) & = \check \bE_{i, d} + \check \bK_{i, d} \check \bF_{\ell -i, d},  \quad
\imath_d (\check \F_{i, d} )  = \check \bE_{\ell -i, d} +
\bK_{\ell -i, d} \check \bF_{i, d}  , \quad
\imath_d  (\check \K_{i, d} )  =  \check \bK_{i, d} \check \bK^{-1}_{\ell - i, d}. \\
\imath_d (\check \bt_d) & =   \check \bE_{r, d} +   v \check \bK_{r, d} \check \bF_{r, d} + \check \bK_{r, d} .
\end{split}
\end{align}
\end{cor}

\begin{prop}
$\imath_d$ is injective.
\end{prop}

\begin{proof}
This is because $\jmath_{d}$  is injective by Proposition ~\ref{J-4}.
\end{proof}

Let $\mbb \Di: \cSi \to \mbb S^{\imath}_{d', \ell} \otimes \mbb S_{d'', \ell}$ be the  homomorphism induced from $\mbb \Delta^{\jmath}$.
By Proposition ~\ref{J-2},

\begin{prop}
\label{I-2}
Let $M\in \Xi^{\imath}_d$.
If
$
\mbb \Delta^{\imath} ( \{M\}_d) = \sum_{M'\in \Xi_{d'}^{\imath}, M'' \in \Xi_{d'', \ell}} h_M^{M', M''} \{M'\}_{d'} \otimes \{M''\}_{d''},
$
then we have $h_M^{M', M''} \in \mbb Z_{\geq 0} [v, v^{-1}]$.
\end{prop}

Write
$
\imath_d (\{B\}_d) = \sum_{A\in \Xi_{d, \ell}}  g_{B, A} \{A\}_d.
$
By Propositions \ref{I-2}, we have

\begin{cor}
\label{I-5}
$g_{B,A} \in \mbb Z_{\geq 0}[v, v^{-1}]$.
\end{cor}

Recall $\mbf T^{\jmath}_{d,n}$ and $\Pi^{\jmath}_{d, n}$ from (\ref{Tj-standard}).
Note that $\mbf T^{\jmath}_{d, n}$ is defined over $\mcal A$, but can be lifted to its generic version $\mbb T^{\jmath}_{d, n}$.
Let $\Pi^{\imath}_{d, \ell}$ be the subset of $\Pi^{\jmath}_{d, n}$ defined by $a_{r+1, d+1} =1$.
Let $\mbb T^{\imath}_{d, \ell}$ be the space of $\mbb T^{\jmath}_{d, n}$ spanned by $[A]_d$ where $A\in \Pi^{\imath}_{d, \ell}$.
In the same fashion, 
let $\mbb T_{d, n}$ be the generic version of $\mbf T_{d, n}$ in (\ref{Td}), and let $\Pi_{d, \ell}$ be the subset of $\Pi_{d, n}$ defined by
$a_{r+1, d+1} =1$. Similarly, we have $\mbb T^{\mbf a''}_{d, \ell}$.
Let $\mbb H_{A_d}$ and $\mbb H_{B_d}$ be the generic version of the Hecke algebra
$\mbf H_{A_d}$ and $\mbf H_{B_d}$ used in Proposition ~\ref{typeA-typeB}.
The following is the $\imath$-analogue of Proposition ~\ref{typeA-typeB}, obtained by restricting the digram therein to the desired subspaces.

\begin{prop}
\label{typeA-typeB-ii}
We have the following commutative diagram.
\[
\begin{CD}
\cSi \times \mbb T^{\imath}_{d, \ell} @>>>  \mbb T^{\imath}_{d, \ell} @<<< \mbb T^{\imath}_{d,\ell} \times \mbb H_{B_d} \\
@V\imath_{d, \ell} \times \zeta_{d} VV @V\zeta_{d}VV @VV \zeta_{d} \times \zeta^1_{d}V \\
\mbb S_{d, \ell} \times \mbb T_{d, \ell} @>>>  \mbb T_{d, \ell}  @<<\psi^{\imath}_1<  
\oplus_{\mbf a'' \models 1^{2d+1} } \mbb T^{\mbf a''}_{d,\ell} \times  \oplus_{\substack{\mbf b'' \models 1^{2d+1}\\b''_{d+1}=0}} \ ^{\mbf b''} \mbb H_{A_d} @<<\psi^{\imath}_2< \mbb T_{d, \ell} \times \mbb H_{A_d}.
\end{CD}
\]
Moreover, $\zeta_{d}( [A^J]_d) = [A]_d$ for all $A\in \Pi^{\imath}_{d, \ell}$.
\end{prop}

\subsection{Positivity in the projective limit}

Consider the projective system $(\cSi, \cphi)_{d\in \mbb Z_{\geq 0}}$ of associative algebras.
We define an element $\mbb e_i$ in the projective system whose $d$-th component  is $\E_{i, d}$ for all $d\in \mbb Z_{\geq 0}$.
This is well-defined by Lemma ~\ref{i-transfer}.
Similarly, we can define $\mbb f_i$, $\mbb t$ and $\mbb k_i^{\pm 1}$.

Let $\mbb U^{\imath}_{\ell}$ be the subalgebra of the projective system generated
by  $\mbb e_i$, $\mbb f_i$, $\mbb k_i^{\pm 1}$ for all $i\in [1, r-1]$ and $\mbb t$.
By ~\cite{LW14}, this is a coideal subalgebra of the quantum $\mathfrak{sl}_n$ for $n$ even.
A presentation of this algebra by generators and relations can also be found in ~\cite{LW14}.

Set $\Xi^{\imath}_{\infty, \ell} = \sqcup_{d\in \mbb Z_{\geq 0}} \Xi^{\imath}_{d, \ell}$.
We say two matrices are equivalent if they differ by an even multiply of the identity matrix $I_{\ell}$.
We denote $\Xi^{\imath}_{\infty, \ell}/\approx$ for the set of equivalence classes.

By ~\cite{LW14}, we know that $\cphi (\{A\}_d)=\{ A- 2I_{\ell}\}_{d - \ell} $ if the diagonal entries of $A$ are large enough.
To an element $\widehat A \in \Xi^{\imath}_{\infty, \ell}/\approx$,
we define an element $b_{\widehat A}$ in the projective system whose $d$-th component is $\{A+pI_{\ell} \}_d$ for some $p$ if $d$ is big enough.

Let $\dot{\mbb U}^{\imath}\equiv \dot{\mbb U}^{\imath}_{\ell}$ be the space spanned by $b_{\widehat A}$ for $\widehat A \in \Xi^{\imath}_{\infty, \ell}/\approx$.
By ~\cite{LW14}, $\dot{\mbb U}^{\imath}$ is an associative algebra, the idempotented version of $\mbb U^{\imath}_{\ell}$ and $b_{\widehat A}$ forms the canonical basis $\mBi$ of $\dot{\mbb U}^{\imath}$ defined in ~\cite{LW14}.
Let $\mbb X^{\imath}_{\ell}$ be the subset of $\widehat A \in \Xi^{\imath}_{\infty, \ell}/\approx$ parametrized by the diagonal matrices.
This algebra admits a decomposition
$\dot{\mbb U}^{\imath}  = \oplus_{\widehat \mu, \widehat  \lambda \in \mbb X^{\imath}} \  _{\widehat \mu}\mbb U^{\imath}_{ \widehat \lambda}$
where $_{\widehat \mu}\mbb U^{\imath}_{ \widehat \lambda}
=b_{\widehat \mu} \dot{ \mbb U}^{\imath} b_{ \widehat \lambda}
$.

Replace the projective system $(\cSi, \cphi)$ by $(\mbb S_{d, \ell}, \phi_{d, d-\ell})$, we can define the elements
$\check{ \mbb E}_i, \check{ \mbb F}_i$ and $\check{\mbb K}_i^{\pm 1}$ in this  projective system and they generate
over $\mbb Q(v)$ the quantum $\mathfrak{sl}_l$: $\mbb U_{\ell}$.

Set $\Xi_{\infty, \ell} = \sqcup_{d\in \mbb Z_{\geq 0}} \Xi_{d, \ell}$.
We say two matrices are equivalent if they differ by a multiply of the identity matrix $I_{\ell}$.
We denote $\Xi_{\infty, \ell}/\sim$ for the set of equivalence classes.
To an element $\overline A \in \Xi_{\infty, \ell}/\sim$,
we define an element $b_{\overline A}$ in the projective system whose $d$-th component is $\{A+pI_{\ell} \}_d$ for some $p$ if $d$ is big enough.
Then the space $\dot{\mbb U}_{\ell}$ spanned by $b_{\overline A}$ for all $\overline A \in \Xi_{\infty, \ell}/\sim$ is an associative algebra,  the idempotented version of $\mbb U_{\ell}$
by ~\cite{M10} and $b_{\overline A}$ forms the canonical basis $\mB_{\ell}$.
Let $\mbb X_{\ell}$ be the subset of $\Xi_{\infty, \ell}/\sim$ consisting of all diagonal matrices.
$\dot{\mbb U}_{\ell} = \oplus_{\overline \mu, \overline \lambda \in \mbb X_{\ell}} \ _{\overline \mu} \mbb U_{\overline \lambda}$,
where $_{\overline \mu} \mbb U_{\overline \lambda} = b_{\overline \mu} \dot{\mbb U} b_{\overline \lambda}$.

The linear map $\Delta^{\imath}$ on the $\imath$Schur algebra level induces a linear map
\begin{align}
\label{mDi-comm}
\mDi_{\widehat{\mu'}, \widehat{\lambda'}, \overline{\mu''}, \overline{\lambda''}} :
\ _{\widehat{\mu}} \mUi_{\widehat{\lambda}} \to \  _{\widehat{\mu'}} \mUi_{\widehat{\lambda'}}
\otimes \ \! _{\overline{\mu''}} \mU_{\overline{\lambda''}},
\quad  \forall (\mu', \mu'') \vdash \mu, (\lambda',\lambda'') \vdash \lambda
\end{align}
where $\vdash$ is defined similar to  (\ref{vdash}) with row vectors replaced by diagonal matrices.
Write $\mDi_{\widehat{\mu'}, \widehat{\lambda'}, \overline{\mu''}, \overline{\lambda''}} (a) =
\sum_{b \in \mBi, c\in \mB} n^{b, c}_a b\otimes c$,  for all $a \in \mBi$, then we have the $\imath$-analogue of Theorem ~\ref{J-conj-25.4.2}, whose proof is the same as the Theorem using Proposition ~\ref{I-2}.

\begin{thm}
\label{I-conj-25.4.2}
$n^{b, c}_a \in \mbb Z_{\geq 0} [v, v^{-1}]$.
\end{thm}

The linear map $\imath_d$ induces a  linear map
\begin{align}
\label{U-imath}
\imath_{\widehat  \mu, \widehat  \lambda, \overline{\mu''}, \overline{\lambda''}}:
\ _{\widehat  \mu} \mbb U^{\imath} _{\widehat \lambda}
\longrightarrow
\ _{\overline{\mu''}} \mbb U_{\overline{ \lambda''}},
\quad  \forall \mu'' \vdash \mu, \lambda'' \vdash \lambda
\end{align}
We have the $\imath$-analogue of Theorem \ref{j-positivity} by using Corollary ~\ref{I-5}.

\begin{thm}
\label{i-positivity}
Let $b \in \mBi$. If
$\imath_{\widehat \mu, \widehat  \lambda, \overline{\mu''}, \overline{\lambda''}}  (b)
=\sum_{a\in \mB}  g_{b, a} a$, then
$g_{b, a} \in \mbb Z_{\geq 0}[v, v^{-1}]$.
\end{thm}

\end{document}